\documentclass[11pt]{amsart}
\pdfoutput=1

\usepackage[a4paper,bottom=35mm]{geometry}
\usepackage[dvipsnames]{xcolor}
\usepackage{amsmath,amssymb}
\usepackage{url}
\usepackage{graphicx}
\usepackage{hyperref}
\usepackage{multirow}
\usepackage{enumitem}

\newtheorem{thm}{Theorem}[section]
\newtheorem{prop}[thm]{Proposition}
\newtheorem{lem}[thm]{Lemma}
\newtheorem{cor}[thm]{Corollary}
\newtheorem{rmk}[thm]{Remark}
\newtheorem{example}[thm]{Example} 
 
\newtheorem{defn}[thm]{Definition}
\newtheorem{con}[thm]{Conjecture} 
\newtheorem{ex}[thm]{Example} 
\newtheorem*{theorem*}{Theorem}

\newcommand{\T}{\mathbb{T}}
\newcommand{\bd}{\partial}
\newcommand{\Z}{\mathbb{Z}}
\newcommand{\Q}{\mathbb{Q}}
\newcommand{\C}{\mathbb{C}}
\newcommand{\R}{\mathbb{R}}
\def\H{\mathbb{H}}
\newcommand{\latt}{\mathcal{L}}
\newcommand{\tri}{\mathcal{T}}
\newcommand{\ind}{\mathcal{I}\nobreak\hspace{.06em}}
\let\Im\relax
\DeclareMathOperator{\Im}{Im}
\DeclareMathOperator{\Ker}{Ker}
\newcommand{\E}{\mathcal{E}}
\newcommand{\cusp}{\mathcal{C}}
\newcommand{\itet}{I_{\Delta}}
\newcommand{\jtet}{J_{\Delta}}
\newcommand{\qex}{\mathcal{E}_q}
\newcommand{\qsin}{\mathcal{S}_q}
\newcommand{\qcos}{\mathcal{C}_q}
\newcommand{\hh}{\mathrm{H}}

\renewcommand{\L}{\ell}
\def\quads{\square} 

\author{Daniele Celoria}
\author{Craig D.~Hodgson}
\author{J.~Hyam Rubinstein}

\title{The $3$D index and Dehn filling}

\begin{document}
\begin{abstract}
We provide a rigorous proof of the Gang-Yonekura formula describing the transformation of the $3$D index under Dehn filling a cusp in an orientable $3$-manifold. The $3$D index, originally introduced by Dimofte, Gaiotto and Gukov, is a physically inspired $q$-series that encodes deep topological and geometric information about cusped 3-manifolds. 
Building on the interpretation of the $3$D index as a generating function over $Q$-normal surfaces, we introduce a relative version of the index for ideal triangulations with exposed boundary. This notion allows us to formulate a relative Gang-Yonekura formula, which we prove by developing a gluing principle for relative indices and establishing an inductive framework in the case of layered solid tori. Our approach makes use of Garoufalidis-Kashaev's meromorphic extension of the index, along with new identities involving $q$-hypergeometric functions. As an application, we study the limiting behaviour of the index for large fillings. 
We also develop code to perform certified computations of the index, guaranteeing correctness up to a specified accuracy. Our extensive computations support the topological invariance of the $3$D index and suggest a well-defined extension to closed manifolds.
\end{abstract}
\maketitle
\tableofcontents 

%&&&&&&&&&&&&&&&&&&&&&&&&&&&&&&&&&&&&&&&&&&&&&&&&&&&&&&&&&&&&&&&&&&&&&&&&&&&&&
%&&&&&&&&&&&&&&&&&&&&&&&&&&&&&&&&&&&&&&&&&&&&&&&&&&&&&&&&&&&&&&&&&&&&&&&&&&&&&

\section{Introduction}

The study of quantum invariants associated with $3$-manifolds has seen remarkable developments over the past few decades, particularly through the interplay of geometric, topological, and physics-inspired techniques. Among these, Dimofte, Gaiotto and Gukov's $3$D index~\cite{dimofte20133, dimofte2014gauge}, has emerged as a powerful invariant, encoding both topological and geometric information about orientable, cusped $3$-manifolds. The $3$D index stems from superconformal field theory, and is defined as a formal power series with integer coefficients in the parameter $q^\frac12$. In its original formulation, it is constructed from combinatorial data associated to ideal triangulations of a $3$-manifold with torus cusps. 
The $3$D index is known to coincide with the Frohman-Kania-Bartoszynska invariant~\cite{frohman2008quantum, garoufalidis2023fkb}, and is closely related to quantum invariants such as the $\hat{A}$-polynomial~\cite{garoufalidis2022periods} and skein modules~\cite{garoufalidis20243d}.

Physically motivated arguments suggest that the $3$D index is invariant under changes of triangulation. However, despite partial progress~\cite{garoufalidis20163d_angle,garoufalidis20151,garoufalidis2019meromorphic}, a complete proof of its topological invariance is still lacking. Moreover, the index is only well-defined for a restricted class of triangulations of cusped manifolds. It was proved in~\cite{garoufalidis20151} that the index converges if and only if the triangulation is $1$-\textit{efficient}, 
that is it does not contain any embedded closed normal surfaces of non-negative Euler characteristic, except for vertex-linking tori. 
Following Thurston's hyperbolisation theorem, a manifold admitting such a triangulation can only be hyperbolic or a small Seifert fibred space. 

The relation with $1$-efficiency ultimately led to the reinterpretation by Garoufalidis, Hodgson, Hoffman and  Rubinstein of the $3$D index as a `generating function' over certain normal surfaces~\cite{garoufalidis20163d}. In this setting, the index is refined to a family of invariants $\ind^\gamma_\tri(q)$, indexed by boundary homology classes $\gamma \in \mathrm{H}_1(\partial M; \mathbb{Z})$. These encode contributions from all surfaces with prescribed boundary $\gamma$.\\

Dehn filling is a fundamental operation in $3$-manifold topology, used to construct new manifolds by attaching a solid torus to torus cusps. In our setting, Dehn filling a cusp involves finding a triangulation with standard combinatorics near the cusp and replacing it with an appropriate layered solid torus with specified meridian curve~\cite{jaco2006layered}. 

In~\cite{gang2018symmetry}, based on physical considerations, Gang and Yonekura proposed a formula (see equation~\eqref{eqn:GY_original}) determining the change of the index under Dehn filling of one cusp. The main aim of this paper is to give a mathematically sound proof of this result. More precisely we will prove the following (see Theorem~\ref{thm:main} for the precise statement).
\begin{theorem*}
Let $M$ be a compact orientable $3$-manifold with boundary consisting of at least two tori, and let $T$ be one component of $\partial M$. Let $\tri$ be a $1$-efficient triangulation of $M$ with a standard cusp at $T$.
Given a simple closed curve $\alpha \subset T$, let $\tri(\alpha)$ be the ideal triangulation of $M(\alpha)$ obtained from $\tri$ by replacing the standard cusp with a layered solid torus with filling curve $\alpha$. 
If $\tri(\alpha)$ is $1$-efficient, then \begin{equation*}
\ind_{\tri(\alpha)}(q) = \frac12 \left( \sum_{\substack{\gamma \in \hh_1(T;\Z)\\ \alpha \cdot \gamma = 0}} (-1)^{|\gamma|} \left( q^{\frac{|\gamma|}{2}} + q^{-\frac{|\gamma|}{2}}\right) \ind_{\tri}^\gamma(q) -  \sum_{\substack{\gamma \in \hh_1(T;\Z)\\ \alpha \cdot \gamma = \pm 2}} (-1)^{|\gamma|} \ind_{\tri}^\gamma(q) \right),
\end{equation*}
where $|\gamma|$ denotes the number of components of $\gamma$, and $\cdot$ denotes the algebraic intersection number on $T.$ 
\end{theorem*}

We start by showing in Theorem~\ref{thm:algorithm_efficiency} that Dehn filling a $1$-efficient triangulation will, with at most finitely many exceptions, yield a $1$-efficient triangulation. Furthermore, there exists an algorithm to determine which slopes yield non-$1$-efficient triangulations. 

We then introduce a special notion related to efficiency, which we call \textit{collar efficiency}. In Theorem~\ref{thm:collar_eff} we prove that the degenerate Dehn fillings obtained by filling along a edge loop in the boundary of a standard cusp neighbourhood can  yield either a non-collar efficient triangulation,  $S^1 \times B^2$ or $S^1 \times S^1 \times (0,1)$. This result will be key for understanding the behaviour of the $3$D index on these degenerate surgeries. 

We define a \textit{relative index} for manifolds admitting triangulations with ideal vertices and exposed boundary. Roughly speaking, the definition is constructed by keeping track of surfaces in the manifold, as well as some boundary edge weights. We use this relative version of the index to prove a general \textit{gluing formula} (Theorem~\ref{thm:gluing}), describing how to obtain the $3$D index of a triangulation from the relative indices of a decomposition. 

In order to prove the main result, we consider a \textit{relative version} of the Gang-Yonekura formula (Proposition~\ref{prop:gang_cusp}), and prove it in the case of the replacement of a standard cusp with a layered solid torus. This proof takes the form of an induction over the number of tetrahedra in the triangulation of the layered solid torus,  
and occupies most of the paper. 

Another key component of the proof involves extracting specific generating functions for quantities related to the index from the meromorphic extension of the $3$D index developed by Garoufalidis and Kashaev~\cite{garoufalidis2019meromorphic}. This viewpoint will lead us through deep connections with the theory of $q$-hypergeometric functions; in the course of the proof we will derive certain relations (see \textit{e.g.}~Theorems~\ref{thm:new_thm_112}, \ref{thm:gang112} and Corollary~\ref{cor:linear_diagonal_relations}) between $q$-series that are of independent interest.\\

As a consequence of our main result, we can easily understand the limiting behaviour of the $3$D index for manifolds obtained as \textit{large} Dehn fillings. In Theorem~\ref{thm:index_asymptotics} we look at Dehn fillings along curves $(a_n,b_n)$ approaching $\infty$, where the slopes $\frac{a_n}{b_n}$ approach a \emph{rational} limit and find, quite surprisingly, that 
the asymptotic behaviour of the $3$D index depends on the \emph{direction} in which the curves $(a_n,b_n)$ approach $\infty$. 
This behaviour is quite different to  Thurston’s hyperbolic Dehn surgery theory~\cite{thurston1982three}, where only the size of the surgery coefficients is important. On the other hand, in Theorem~\ref{thm:real_fillings} 
we examine Dehn fillings along curves $(a_n,b_n)$ where the slopes $\frac{a_n}{b_n}$ approach an \emph{irrational} limit, and show 
the $3$D index of the Dehn fillings converge to the index of the unfilled manifold.

We then turn to computations. Using our main result, we partially confirm a conjecture from~\cite{dimofte20133}, by explicitly computing the $3$D index of alternating torus knots. We have also implemented~\cite{github} \textit{certified code} to compute the index up to some threshold. We use this code to provide strong experimental confirmation of the index's invariance. 
For example, we compute in Section~\ref{ssec:whitehead} the first terms of the $3$D index for surgeries on one cusp of the Whitehead link, recovering the index of census manifolds and the index of twist knot complements.

Notably, this approach allows extremely quick computations of the index for many cusped manifolds admitting a surgery description on a low-complexity manifold.
We also determine in Section~\ref{ssec:magic} that the first terms of the index for Dehn fillings on one cusp of the `magic' manifold~\cite{martelli2006dehn, thompson2025triangulations} resulting in census manifolds, are correct. \\

The next part of this paper takes a more speculative point of view.  There is currently no mathematical definition of the $3$D index for closed 3-manifolds. Nevertheless, by applying the Gang-Yonekura formula to Dehn filling on a singly-cusped manifold, one can `force' an extension of the definition to the closed case  (\textit{cf.}~\cite{gang2018symmetry}). Surprisingly, this procedure does appear to consistently yield coherent and meaningful answers.
Again, we support this claim with extensive computations. In Sections~\ref{sec:filling_trefoil}
and~\ref{sec:filling_fig8} we provide \textit{exact computations} of the $3$D index for certain closed manifolds, obtained as surgeries on the complements of $3_1$ and $4_1$.  We then determine the low degree terms of the index for several surgeries on the complement of $4_1$ and $12n_{242}$, including all their exceptional fillings. 
These computations are a first step towards a mathematically well-founded definition of the index for closed $3$-manifolds.

The last section contains a thorough description of the computational aspects of this project, including the implementation of our certified code and a description of the datasets associated to the paper.\\

The paper contains two appendices dwelling on certain technical aspects of computations related to small and degenerate layered solid tori. More precisely, in Appendix~\ref{sec:gang011} we prove that $(0,1,1)$-surgeries will (when the index is defined) only result in trivial $3$D index, consistently with the relation between degenerate surgeries and collar efficiency proved in Section~\ref{sec:collar_eff}. In Appendix~\ref{sec:relative_degenerate}, as a sanity check, we directly compute the relative index of certain degenerate layered solid tori using alternative non-minimal triangulations, thus giving an independent confirmation of the initial case in our main proof. Finally, Appendix~\ref{computation:fig8} presents a detailed and explicit account of the steps involved in computing the Gang-Yonekura formula in the case of the figure eight knot complement.~\\

\noindent\textbf{Acknowledgments:}
The authors wish to thank Ben Burton, Tudor Dimofte, Jan de Gier, Neil Hoffman, Rinat Kashaev, Christoph Koutschan, Daniel Mathews, Jonathan Spreer, and Ole Warnaar  for interesting discussions and for sharing their expertise. This research has been partially supported by grant DP190102363 from the Australian Research Council.

%&&&&&&&&&&&&&&&&&&&&&&&&&&&&&&&&&&&&&&&&&&&&&&&&&&&&&&&&&&&&&&&&&&&&&&&&&&&&&
%&&&&&&&&&&&&&&&&&&&&&&&&&&&&&&&&&&&&&&&&&&&&&&&&&&&&&&&&&&&&&&&&&&&&&&&&&&&&&

\section{Triangulations of 3-manifolds and normal surfaces}\label{sec:$3$-man}

In what follows, $M$ will denote a compact, oriented $3$-manifold, whose boundary consists of $r$ tori. Removing the boundary yields a \emph{cusped manifold} with $r$ cusps.  

It is well-known that the interior of any such manifold $M$ can be triangulated with ideal tetrahedra. We let $\tri$ be such an oriented ideal triangulation. If $\tri$ is obtained by gluing together $n$ ideal tetrahedra $\Delta_1, \ldots, \Delta_n$, then a standard Euler characteristic argument implies that there are $n$ edge classes in $\tri$. These edge classes will be denoted by $\mathcal{E}$.

In each tetrahedron there are three \emph{normal quadrilateral types}, each specified by the pair of opposite edges in the tetrahedron facing the quadrilateral (see Figure~\ref{fig:quad_types}).  So the set of normal quadrilateral types  $\quads$ in $\tri$ has $3n$ elements.
\begin{figure}[!htbp]
\centering
\includegraphics[width=0.7\linewidth]{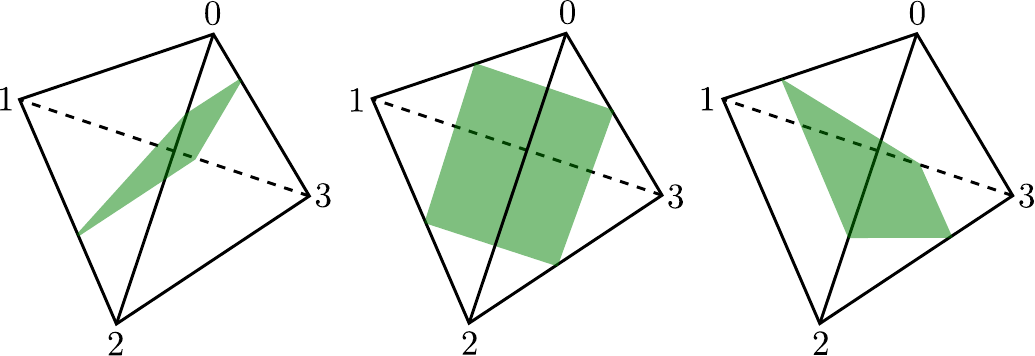}
\caption{The three normal quad types in a tetrahedron: separating vertices 01/23, 02/13 and 03/12 respectively.}
\label{fig:quad_types}
\end{figure}

Following Thurston \cite{Thu78notes}, hyperbolic structures on $M$ can be constructed by finding shapes of ideal hyperbolic tetrahedra satisfying Thurston's gluing equations, consisting of edge equations with coefficients $E_i$, $i=1,\ldots,n$, and peripheral curve equations with coefficients $M_k, L_k$, $k=1,\ldots,r$. 
After ordering the tetrahedra and the normal quadrilateral types in each tetrahedron, these equations can be encoded in a matrix with columns indexed by normal quadrilateral types and rows indexed by edge classes and a choice of two peripheral curves for each cusp.  
The $j$-th coefficient in the edge equation $E_i$ records the number of times the edges opposite quadrilateral $j$ appear in the $i$-th edge class; the coefficients in each peripheral equation record the number of times the peripheral curve winds anti-clockwise around the quadrilaterals.
The coefficients of these equations are the rows of the $(3n+2k)\times 3n$ gluing matrix $G_\tri$ output by the
\texttt{gluing\_equations()} command in SnapPy~\cite{SnapPy}.~\\

%&&&&&&&&&&&&&&&&&&&&&&&&&&&&&&&&&&&&&&&&&&&&&&&&&&&&&&&&&&&&&&&&&&&&&&&&&&&&&
%&&&&&&&&&&&&&&&&&&&&&&&&&&&&&&&&&&&&&&&&&&&&&&&&&&&&&&&&&&&&&&&&&&&&&&&&&&&&&

\subsection{Q-normal surfaces}\label{ssec:Q-normal}~\\

Next we recall some background on $Q$-normal surfaces in an ideal triangulation $\tri$. This material will be used throughout the paper, as we will usually regard the $3$D-index as a kind of generating function for $Q$-normal surfaces. 

Denote by $\mathcal{Q}(\tri;\Z) \subset \Z^\quads \cong \Z^{3n}$ the lattice of integer solutions to Tollefson's $Q$-matching equations~\cite{tollefson1998normal}, 
and by $\mathcal{Q}(\tri;\Z_+)$ the subset consisting of non-negative integer solutions.
Each element $x \in \mathcal{Q}(\tri;\Z_+)$ determines a (possibly singular) \textit{spun normal surface} $S_x$ in $\tri$;
see \cite[Sec.~6]{garoufalidis20163d} for the details.
Note that $S_x$ is embedded if and only if $x$ has at most one non-zero entry for each tetrahedron in $\tri$. 

Let $\L: \R^{3n} \to \R^{3n}$ be the linear map 
\begin{equation}
\L(a_1, b_1, c_1,  \ldots , a_n, b_n, c_n) 
= (-b_1+c_1, -c_1+a_1,-a_1+b_1, \ldots, -b_n+c_n, -c_n+a_n,-a_n+b_n),
\end{equation}
which converts quad coordinates to corresponding \emph{shear parameters} as in \cite[Sec.~6.1]{garoufalidis20163d} or \emph{leading-trailing deformations} as in \cite{FuterGueritaud}.
Then the $Q$-matching equations for a normal surface $S$ are given by 
\begin{equation}\label{Q-matching-eqns}
\L(E_i) \cdot S = 0,\,\, \text{ for } i=1, \ldots , n,
\end{equation}  
(see \cite{garoufalidis20163d}, \cite{tillmann}).

Further, the skew-symmetric pairing of Neumann-Zagier~\cite{NZ85} is given by
\begin{equation}\label{eqn:pairing}
\omega : \R^{3n} \times \R^{3n} \to \R, \qquad \omega(x,y)= \ell(x) \cdot y,
\end{equation}
where $\cdot$ denotes the dot product. Then the {\em Neumann-Zagier symplectic relations}  say that all 
symplectic products of $E_i, M_k, L_k$ for $i=1,\ldots, n$, 
$k=1,\ldots,r$ are {\em zero} except that
\begin{equation}
\omega(L_k,M_k) = 2 = -\omega(M_k,L_k)  \text{ for each }k=1, \ldots , r .
\end{equation}

It follows that the $Q$-matching equations are satisfied by the following vectors:  
\begin{itemize}
\item edge solutions $E_i$ (whose coefficients give the edge equation for the $i$-th edge),
\item tetrahedron solutions $T_j$ (that are $(1,1,1)$ in the $j$-th tetrahedron and $0$ otherwise),
\item two peripheral curve solutions $M_k, L_k$ for each cusp, corresponding to a basis of $\hh_1(T;\R)$ for each component $T \subseteq \partial M$. 
\end{itemize}
In fact, by \cite{kangrubinstein}, these span the vector space $\mathcal{Q}(\tri;\R) \subset \R^\quads \cong \R^{3n}$ of all real solutions to the $Q$-matching equations. 

Therefore, every \emph{$Q$-normal class} $S\in\mathcal{Q}(\tri;\R)$ can be written as 
\begin{equation}
S = \sum_{i=1}^n \left(x_i E_i+ y_i T_i\right) + \sum_{k=1}^r \left( p_k M_k + q_k L_k \right),
\end{equation}
for some coefficients $x_i,y_i,p_k,q_k \in \R$.\\

\medskip
There are two important functionals (\textit{cf.}~\cite[Sec.~6.2]{garoufalidis20163d}) on $\mathcal{Q}(\tri;\R) $ that arise when studying the degree of terms in $3$D index sums.
The first is the linear map $\chi\colon \mathcal{Q}(\tri;\R) \rightarrow \R$ known as the \textit{formal Euler characteristic}. This is simply defined as 
\begin{equation}\chi(S) = -\sum_{i = 1}^n \left(2x_i +y_i\right) .\end{equation}
Note that $\chi$ reduces to the usual Euler characteristic when $S$ is an embedded closed or spun normal surface.

The second functional, the \textit{double arc function} $\delta \colon \mathcal{Q}(\tri;\R)  \rightarrow \R$, is a quadratic function:
 if $$S = (a_1, b_1, c_1, \ldots, a_n,b_n,c_n),$$ then 
\begin{equation}\delta (S) = \sum_{i = 1}^r \left( a_i b_i + b_i c_i + c_i a_i\right).\end{equation}
It is easy to see that $\delta$ is keeping track of self-intersections of quads, hence $\delta (S) = 0$ if and only if $S$ is embedded. 
There is also an associated bilinear function
$$
\delta : Q(\tri;\R) \times Q(\tri;\R) \to \R
$$
such that 
$$
\delta(S+S')=\delta(S)+\delta(S') + 2 \delta(S,S') 
\text{ for all } S, S' \in Q(\tri;\R) .
$$

Finally, we are going to consider the boundary map $$\partial\colon \mathcal{Q}(\tri;\R) \rightarrow \hh_1(\partial M;\R)$$
associating to $S$ the formal sum 
\begin{equation}\partial (S) = 2 \sum_{k=1}^r \left( p_k M_k + q_k L_k \right).\end{equation}
 This quantity can also be computed from the quad coordinates of $S$ in terms of the skew-symmetric pairing~\eqref{eqn:pairing}:
\begin{equation} 
\bd(S)= \sum_k \left(  -\omega(S,L_k) M_k + \omega(S,M_k) L_k \right),
\end{equation} 
 as explained in \cite[Sec.~6.2]{garoufalidis20163d}.~\\ 

%&&&&&&&&&&&&&&&&&&&&&&&&&&&&&&&&&&&&&&&&&&&&&&&&&&&&&&&&&&&&&&&&&&&&&&&&&&&&&
%&&&&&&&&&&&&&&&&&&&&&&&&&&&&&&&&&&&&&&&&&&&&&&&&&&&&&&&&&&&&&&&&&&&&&&&&&&&&&

\subsection{Standard cusps}\label{sec:standard_cusp}~\\

A \textit{standard cusp} $\cusp$ is the simplest possible triangulated neighbourhood of a cusp, consisting of two tetrahedra glued together as shown in Figure~\ref{fig:std_cusp}. Topologically such a neighbourhood is a solid torus with its core removed and with a puncture on its boundary.  \\

\begin{figure}[!htbp]
\centering
\includegraphics[width=0.5\linewidth]{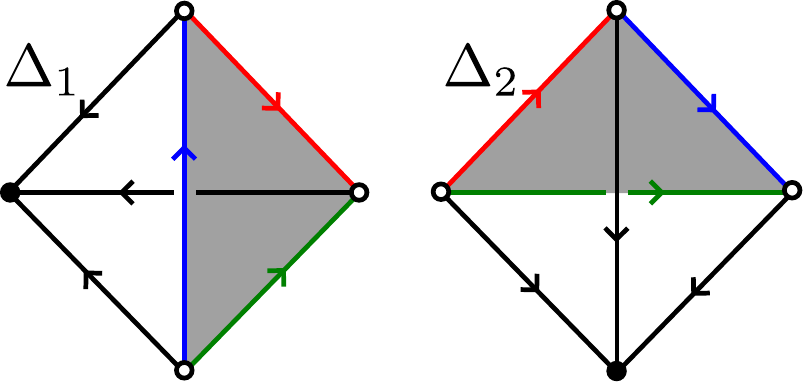}
\caption{Gluing instructions for the two tetrahedra forming a standard cusp. The cusp vertex is in black, and the punctured torus forming the boundary of its neighbourhood is shaded.}
\label{fig:std_cusp}
\end{figure}
In Table~\ref{table:cusp_gluing} we show the gluing matrix for the two tetrahedra ideal triangulation of $\cusp$, using the labelling conventions of Figure~\ref{fig:cusp_conventions}. 
The vertical edge $e_0$ is the only one connecting the cusp within the standard neighbourhood to the cusp on the punctured torus boundary. Moreover, since all the coefficients in the corresponding edge equation $E_0$ are $1$, the matching equation along this edge is always automatically satisfied. Therefore all of the $Q$-normal classes for $\cusp$ are generated by the six quads in the two tetrahedra.

\begin{figure}[!htbp]
\centering
\includegraphics[width= 4cm]{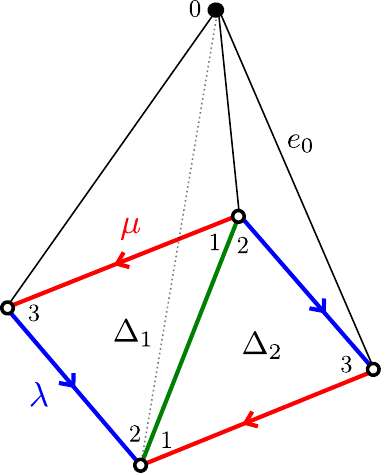}
\caption{The labelling conventions used for $\cusp$. The unique internal edge is $e_0$, while $\lambda$ and $\mu$ are chosen as the longitude and meridian for the punctured torus boundary.} 
\label{fig:cusp_conventions}
\end{figure}

\begin{center}
\begin{table}[!htbp]
\begin{tabular}{|c | c c c | c c c|} 
 \,tet &  & $\Delta_1$ &  && $\Delta_2$& \\\hline
 quad & {\small 01\slash 23} & {\small 02\slash 13} & {\small 03\slash 12} & {\small 01\slash 23} & {\small 02\slash 13} & {\small 03\slash 12}\\
\hline
 $E_0$ & 1 & 1 & 1 &1& 1& 1\\
 $E_1$ & 1 & 0 & 0 & 1 & 0 & 0 \\
 $E_2$ & 0 & 1 & 0 & 0 & 1 & 0 \\
 $E_3$ & 0 & 0 & 1 & 0 & 0 & 1 \\
\hline
 $L$ & -1 & 0 & 0 & 1 & 0 & 0 \\
 $M$ & 0 & -1 & 0 & 0 & 1 & 0 
\end{tabular}
\caption{Gluing equations for $\cusp$, with rows corresponding to the labelling of curves from Figure~\ref{fig:cusp_conventions}.}
\label{table:cusp_gluing}
\end{table}
\end{center}

The triangulation $\cusp$ has many symmetries, including an extension of the hyperelliptic involution on the boundary which interchanges the two tetrahedra. This gives an involution on the set of $Q$-normal classes in $\cusp$ which acts on quad coordinates by 
$$(a_1,b_1,c_1,a_2,b_2,c_2) \mapsto (a_2,b_2,c_2,a_1,b_1,c_1)$$ 
and reverses the boundary map. Geometrically, each non-trivial curve on the punctured torus boundary of $\cusp$ admits exactly two extensions to spun annuli in the interior of the cusp, with opposite spinning directions; see Figure~\ref{fig:spun_quad} for an example.

\begin{figure}[!htbp]
\includegraphics[width=10cm]{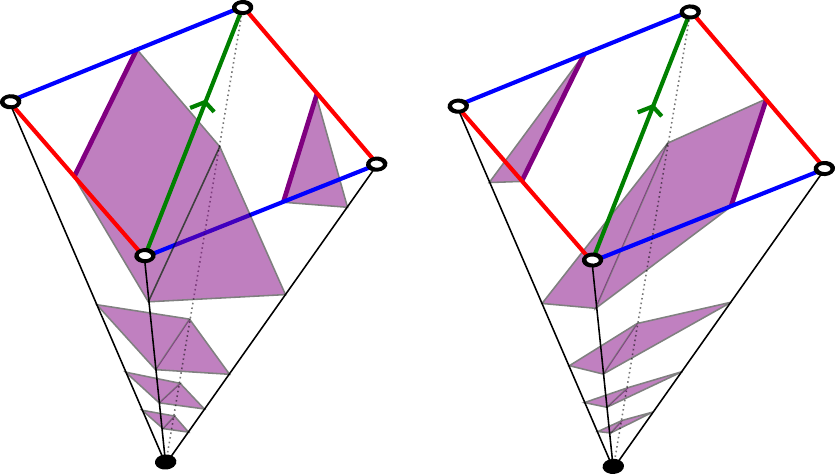}
\caption{There are two ways of completing a boundary curve to a spun annulus in the standard cusp. The two possibilities spin in opposite directions, and are related by an involution. }
\label{fig:spun_quad}
\end{figure}

\medskip
Given that standard cusps are quite special from a combinatorial point of view, it is reasonable to ask which manifolds admit triangulations whose cusps are standard. The following result guarantees the existence of such triangulations for all the manifolds we are presently interested in.

\begin{prop}[\cite{howie2020polynomials}]\label{std_cusp1}
Let $M$ be a connected, compact, orientable, irreducible, $\partial$-irreducible $3$-manifold with boundary consisting of $r \ge 2$ tori. Then, for any torus boundary component $T$, there exists an ideal triangulation $\mathcal{T}$ of the interior of $M$ such that all cusps corresponding to components in $\partial M \setminus T$ are standard.
\end{prop}

Even better, such triangulations of cusped hyperbolic $3$-manifolds turn out to be quite generic from a geometric point of view. 

\begin{prop}\label{std_cusp2}
Let $M$ be a cusped hyperbolic $3$-manifold with at least two cusps, and assume that one cusp $T$ is generic, \textit{i.e.}~the corresponding Euclidean torus is not rectangular and there exists a unique shortest geodesic from $T$ to itself. Then $M$ admits a geometric ideal triangulation with a standard cusp at $T$.
\end{prop}
\begin{proof}
By first expanding an embedded horospherical cusp at $T$ until it bumps into itself, then expanding the other cusps, we obtain a canonical (Epstein-Penner~\cite{epstein1988euclidean}) ideal triangulation $\tri$ with a standard cusp at $T$, as in Gu\'{e}ritaud-Schleimer~\cite{gueritaud2010canonical}. 
\end{proof}

Recall from the introduction that for computations of the $3$D-index we need \emph{$1$-efficient} triangulations.
\begin{defn}\label{def:1eff}
A triangulation $\tri$ of a cusped $3$-manifold is called $1$-efficient if it does not contain any embedded closed normal surfaces of non-negative Euler characteristic, except for vertex-linking tori or Klein bottles.
\end{defn}
The triangulations given by Proposition~\ref{std_cusp1} are not guaranteed to be $1$-efficient, while the geometric triangulations given by Proposition~\ref{std_cusp2} are. These geometric triangulations can be found explicitly using Snap (see~\cite{GHH08}) or SnapPy~\cite{SnapPy}. 
Ben Burton~\cite{benprivate} has also produced a census of $1$-efficient triangulations with up to 8 ideal tetrahedra of cusped $3$-manifolds, which is very useful for finding $1$-efficient triangulations with a standard cusp.~\\ 

%&&&&&&&&&&&&&&&&&&&&&&&&&&&&&&&&&&&&&&&&&&&&&&&&&&&&&&&&&&&&&&&&&&&&&&&&&&&&&
%&&&&&&&&&&&&&&&&&&&&&&&&&&&&&&&&&&&&&&&&&&&&&&&&&&&&&&&&&&&&&&&&&&&&&&&&&&&&&

\subsection{Layered solid tori}\label{sec:LSTs}~\\

In order to implement Dehn filling on triangulations, we will perform the following operation: we first remove the standard cusp, and replace it with a special one-vertex triangulation of a solid torus. This solid torus will be constructed in such a way that the curve we are filling along will bound a disc. There is a combinatorial way of creating these triangulated solid tori, known as \textit{layered solid tori} extensively studied in~\cite{jaco2006layered} and \cite{gueritaud2010canonical}.
For our purposes, these provide a two-parameter family, denoted LST$(p,q,p+q)$, of one-vertex ideal triangulations of the solid torus with one boundary point removed $V$, whose boundary is a once-punctured torus $T_0=\partial V$. 

Recall that there is a unique curve in $\partial V$, the so-called \textit{meridional slope}, that bounds a disc in $V$. 
Let $\tri_0$ denote a one vertex triangulation for the boundary $T_0$, consisting of two triangles as shown in Figure~\ref{fig:cusp_conventions}. 
Then a normal representative for the meridional slope in LST$(p,q,p+q)$ will intersect the three edges in 
$\tri_0$ in a triple $\{p,q,p+q\}$ of non-negative integers, where $p,q$ are relatively prime.

Each triangulation (with $p,q \ge 1$) can be constructed by choosing a path of triangles in the Farey tessellation of the hyperbolic plane $\mathbb{H}^2$. Portions of this are shown in the Poincar\'e disc model in Figure~\ref{fig:farey_tree} and in the upper half-plane model in Figure~\ref{fig:farey_plane}. 

\begin{figure}[!htbp]
\centering
\includegraphics[width=0.5\linewidth]{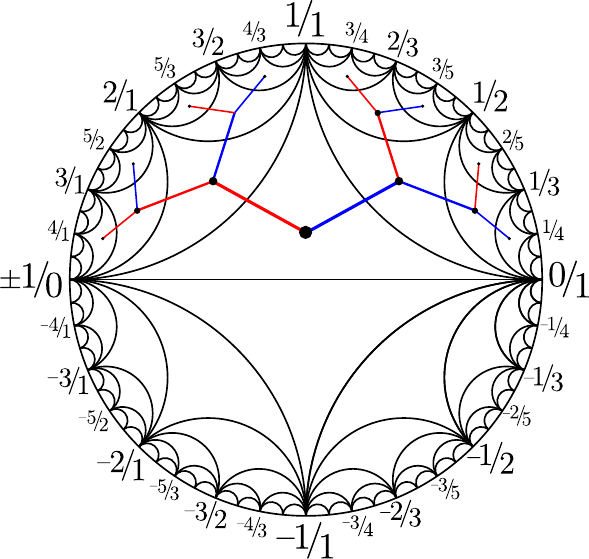}
\caption{The Farey ideal tessellation of the hyperbolic plane. Vertices are in bijection with reduced fractions in $\Q \cup \infty$. The dual tree is drawn with red/blue corresponding to left/right moves.}
\label{fig:farey_tree}
\end{figure}

\begin{figure}[!htbp]
\centering
\includegraphics[width=0.9\linewidth]{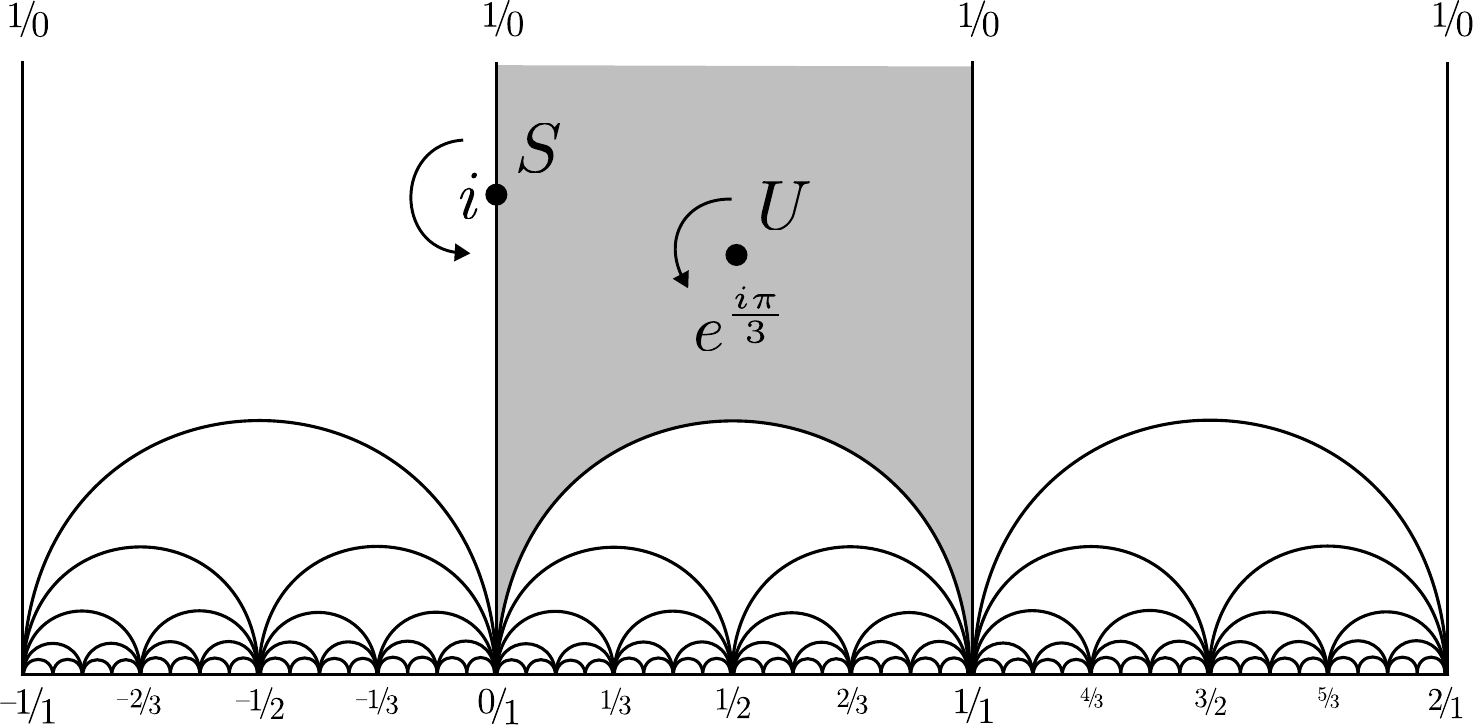}
\caption{Part of the Farey ideal tessellation of the upper half plane. The centres of the two symmetries $S$ and $U$, of orders two and three respectively, are shown.}
\label{fig:farey_plane}
\end{figure}

Each triangle $\Delta$ in the Farey tessellation induces a two-triangle triangulation of the torus $T_0$ using the edges with slopes given by the three vertices of $\Delta$ in $\R \cup \infty$ as follows. We choose an initial triangle. 
Then moving along a path from the initial triangle (without backtracking) describes a way to attach a sequence of tetrahedra  by \emph{layering above} via right and left (often abbreviated as $R\slash L$) moves as shown in the top part of Figure~\ref{fig:LST_layer}).  Folding the initial triangulation of $T_0$ onto a M\"obius band as in the lower part of Figure~\ref{fig:LST_layer} then gives a layered triangulation of the solid torus $V$.

\begin{figure}[!htbp]
\centering
\includegraphics[width=0.8\linewidth]{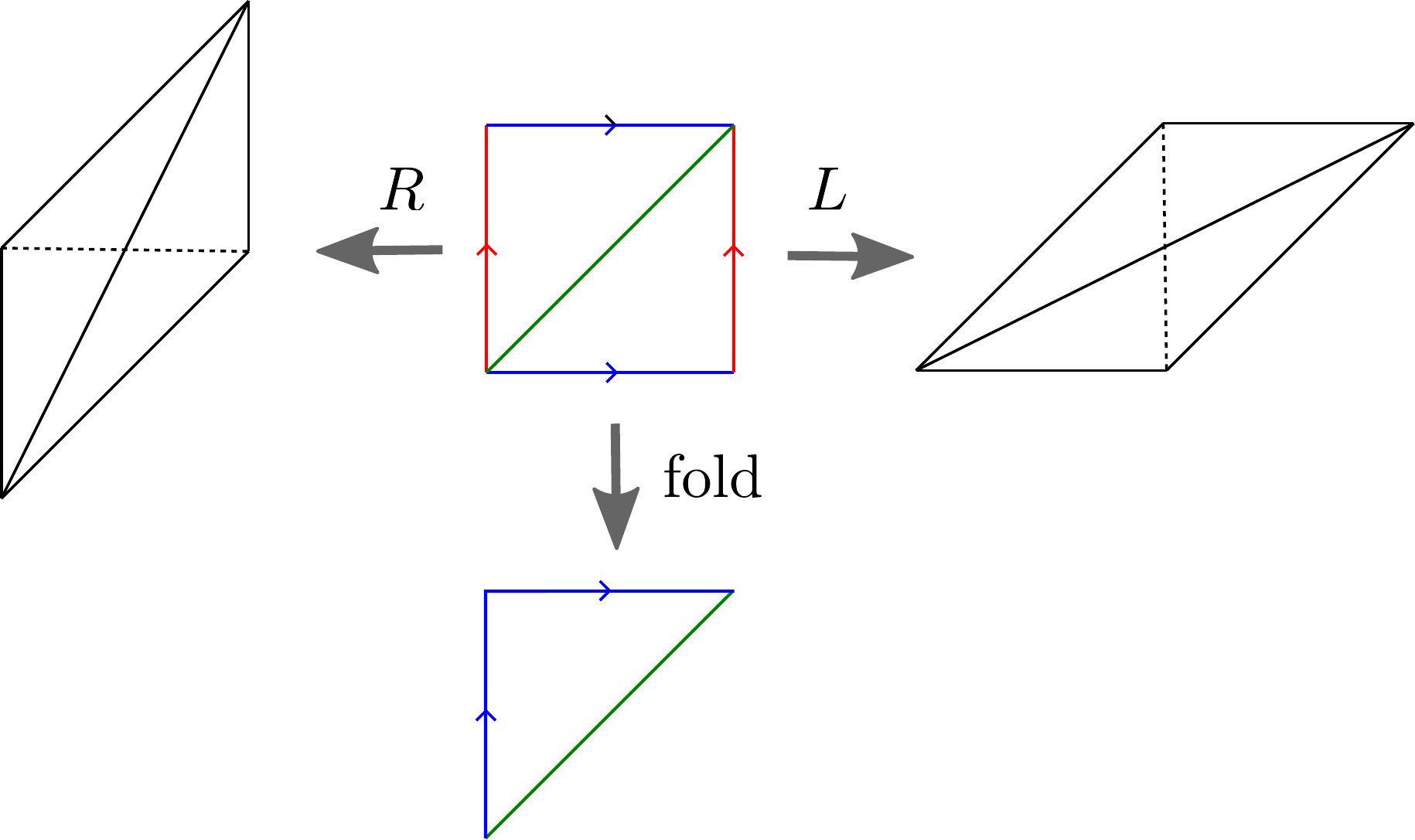}
\caption{Schematic for the $R\slash L$ layering. The torus in the middle of the figure is folded underneath to form a M\"obius strip. }
\label{fig:LST_layer}
\end{figure}

To describe the construction more explicitly, we first recall some well-known facts about the Farey tessellation $F$ and its relation to $\mathrm{SL}_2(\Z)$. \\
Each triangle in $F$ has vertices with slopes $\frac{a}{b},\frac{c}{d},\frac{a+c}{b+d}
$ where $a,b,c,d$ are integers with $ad-bc=1$. 
This gives a triangulation of $T_0=(\R^2\setminus \Z^2)/\Z^2$ with edges given by the vectors $\begin{pmatrix}
a \\b
\end{pmatrix}, \begin{pmatrix}
c \\d
\end{pmatrix}, \begin{pmatrix}
a+c \\b+d
\end{pmatrix}$, and 
corresponds to the matrix  $$A = \begin{pmatrix}a &c \\b &d\end{pmatrix} \in \mathrm{SL}_2(\Z).$$
Moving ``right" or ``left" in the Farey tessellation corresponds to multiplication by the matrices 
$$R=
\begin{pmatrix}1 &0 \\1 &1\end{pmatrix} \,\,\,\text{ and }\,\, L = \begin{pmatrix}1 &1 \\0 &1\end{pmatrix} \in \mathrm{SL}_2(\Z)$$ acting on column vectors. Equivalently, these matrices act on slopes in $\Q \cup \infty$ by linear fractional transformations.

The Farey tessellation is highly symmetrical. One symmetry is implemented by the matrix 
$$S =\begin{pmatrix}0& -1\\1& 0\end{pmatrix},$$  which corresponds to the rotation by $\pi$ about $i=\sqrt{-1}$ in the upper half-plane $\H^2$, and takes the central triangle to the triangle with vertices $\frac01, \frac10=-\frac10, -\frac11$. 

Another is the matrix $$U = \begin{pmatrix} 0& -1\\ 1& -1\end{pmatrix},$$ which corresponds to a rotation by $2\pi/3$ about the central point $\rho = e^{\frac{\pi i}3}$ of the central triangle. Figure~\ref{fig:farey_plane} shows the centres of these two finite order symmetries acting on $F$.

\medskip
To construct a general layered solid torus, we start with an initial ``central"  triangle with vertices at slopes  $\frac10,\frac01, \frac11$, corresponding to the triangulation of $T_0=(\R^2 \setminus \Z^2)/\Z^2$ with edges $e_1=(1,0), e_2=(0,1)$ and $e_3=(1,1)$. \\At the \emph{bottom} of our construction we fold across one of these three edges to form a M\"obius strip. Overall, there are 3 cases:\\~\\
Case 1. Folding across $e_3=(1,1)$ gives a meridian $m_1=(1,-1)$ on the bottom of our construction. Here we can apply any word in $R$ and $L$ to get a matrix $\begin{pmatrix}a &c \\b &d\end{pmatrix}$ with $ad-bc=1$ and $a,b,c,d >0$. This gives a triangulation of $T_0$ with edges $e_1'=(a,b), e_2'=(c,d), e_3'=(a+c,b+d)$.

The meridian weights with respect to this new triangulation are the geometric intersection numbers of $m_1$ with the edges $e_1', e_2', e_3'$, giving $w_1=a+b$, $w_2=c+d$, $w_3=w_1+w_2=a+b+c+d$. (In other words, the meridian weight on an edge is just the sum of the numerator and denominator of the edge slope.)

The \textit{meridian slope} on the boundary of the layered solid torus is $$s = A^{-1} m_1 = -\frac{c+d}{a+b}=-\frac{p}{q},$$ say. (This will match the filling curve in the standard cusp.) This lies in the interval $(-\infty,0)$ containing the meridian slope $-\frac11$, and we get all such slopes.

Conversely, given the boundary slope $s=-\frac{p}{q}$ we can recover the matrix $A$ as follows.
Apply the rotation by $\pi$ in $\H^2$ given by $S$ to get $s'=\frac{q}{p}$, as shown in Figure~\ref{fig:rotation_relation}. Then the “inner” triangle in the Farey tessellation with vertex $s'$ gives A: if $\frac{a}{b}>s'>\frac{c}{d}$ 
are the three vertices of this triangle, then $A = \begin{pmatrix}a &c \\b &d\end{pmatrix}$. 
\begin{figure}[!htbp]
\centering
\includegraphics[width=0.5\linewidth]{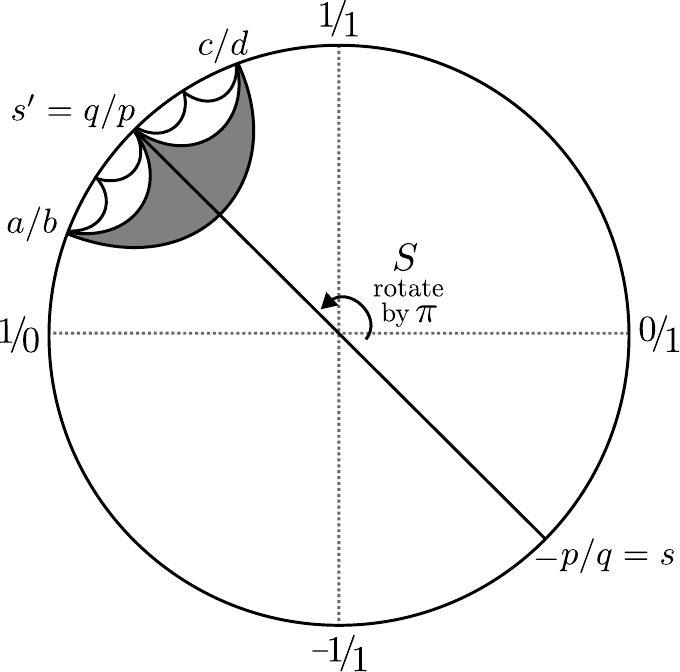}
\caption{Constructing the matrix $A$ from the meridian slope $s$.}
\label{fig:rotation_relation}
\end{figure}

Algebraically, we can find $a,b,c,d $ by applying Euclid’s algorithm to get the smallest positive integers $a, c$ such that $ap-cq=1$, then take $b=q-a$ and $d=p-c$. 
Here $A$ can be written uniquely as a word in $R$ and $L$; this follows since the Farey tessellation is dual to a trivalent tree.\\
\noindent The other two cases are similar:\\
Case 2. Folding across $e_1=(1,0)$ gives a meridian $m_2=(1,2)$. The corresponding boundary slopes $s = A^{-1} m_2$ lie in the interval $(0,1)$ containing the meridian slope $\frac12$.\\
Case 3. Folding across $e_2=(0,1)$ gives a meridian $m_3=(2,1)$. The corresponding boundary slopes $s = A^{-1} m_3$ lie in the interval $(1,\infty)$ containing the meridian slope $\frac21$.\\
The three cases are related by the order $3$ symmetry $U$ taking case (1) to (2) to (3) and back to (1). \\

In general, each layered solid torus triangulation occurs twice in each of the three cases, so occurs six times altogether; these correspond to all possible permutations of the three edge weights. Exactly one boundary slope corresponds to each of these.

When we refer to a layered solid torus we  normally mean a minimal non-degenerate layered solid torus LST$(p,q,p+q)$ obtained through the construction outlined above, and with $\{p,q\} \neq \{0,1\}$ or $\{1,1\}$.

The degenerate case of LST$(1,1,2)$ is special -- we just have a torus folded across an edge to form a M\"obius strip and there are no added tetrahedra. This corresponds to taking $A$ as the identity matrix in the three cases above and occurs only three times: giving  boundary slopes of the meridians $-\frac11$, $\frac12$ and $\frac21$. 
This matches the initial cases for relative index of the standard cusp to be used in the inductive arguments below in Section~\ref{sec:attaching_solid_tori}. Regina's triangulation of LST$(1,1,2)$ is discussed in Appendix~\ref{sec:relative_degenerate}.

The special case where $(p,q,p+q) = (0,1,1)$ corresponds to a particularly degenerate case (see~\cite[Sec.~4]{jaco2006layered}), which we treat separately in Section~\ref{sec:collar_eff} and Appendix~\ref{sec:gang011}. Instead of dealing directly with these degenerate LSTs for Dehn fillings, we are going to use the well-behaved triangulations described in Appendix~\ref{sec:relative_degenerate}.

\begin{rmk}\label{LSTnotation} In light of the correspondence between paths in the Farey graph, continued fractions, and layered solid tori, we will adopt whichever of the three standard notations is most convenient in context: edge weights, as in LST$(p,q,p+q)$; slopes or meridian curves, as in LST$(\alpha)$ and words in $L$ and $R$, as in LST$(w)$.
\end{rmk}

Layered solid tori and standard cusps are very well suited to construct triangulations of Dehn fillings~\cite{gueritaud2010canonical}. More precisely, let $M$ be a $3$-manifold with at least two cusps, $\tri$ a $1$-efficient triangulation of $M$ with a standard cusp $\cusp$ corresponding to a torus $T \subset \bd M$, and fix a basis $\mu,\lambda$ for $\hh_1(T;\Z)$.  
Then each $(a,b) \in \Z^2$ with $\gcd(a,b)=1$ describes an isotopy class of simple closed curve $\alpha$ on $T$ with \emph{homology class} $[\alpha]=a \mu+b\lambda \in \hh_1(T;\Z)$ and corresponding \emph{slope} $s_\alpha = \frac{a}{b} \in \Q \cup\infty$. 
Provided the context is clear, the terms filling curve and slope may be used interchangeably. We denote by $\tri (\alpha)$ the triangulation obtained by removing the standard cusp from $\tri$, and replacing it with the layered solid torus LST$(\alpha)$ with meridian curve $\alpha$. Topologically, this operation corresponds to the Dehn filling $M(\alpha)$ of $M$ along $\alpha$ on $T$. 

%&&&&&&&&&&&&&&&&&&&&&&&&&&&&&&&&&&&&&&&&&&&&&&&&&&&&&&&&&&&&&&&&&&&&&&&&&&&&&
%&&&&&&&&&&&&&&&&&&&&&&&&&&&&&&&&&&&&&&&&&&&&&&&&&&&&&&&&&&&&&&&&&&&&&&&&&&&&&

\section{1-efficiency and Dehn fillings}\label{sec:1eff}

Our starting point is a $1$-efficient triangulation of a cusped $3$-manifold on which we will perform Dehn filling. Since the $3$D index is not well-defined for non-$1$-efficient triangulations, we need to investigate the $1$-efficiency of ideal triangulations obtained by replacing a standard cusp with a layered solid torus. 

The following result shows that Dehn filling a $1$-efficient triangulation will,
with at most finitely many exceptions, yield a $1$-efficient triangulation. Furthermore,
there exists an algorithm to determine which slopes yield non-$1$-efficient triangulations.

\begin{thm}\label{thm:algorithm_efficiency}
Let $M$ be a compact, irreducible, atoroidal, orientable $3$-manifold with boundary consisting of at least two tori. 
Let $\tri$ be a $1$-efficient triangulation of $M$ with a standard cusp $\cusp$. 
Given a filling curve $\alpha \subset \cusp$, let $\tri(\alpha)$ be the ideal triangulation of the Dehn filled manifold $M(\alpha)$ obtained from $\tri$ by replacing the standard cusp with a layered solid torus with meridian $\alpha$.  Then the triangulation $\tri(\alpha)$ is $1$-efficient except for a finite number of choices of $\alpha$.  

Moreover, there is an algorithm to determine which filling curves $\alpha$ give a non-$1$-efficient  triangulation under Dehn filling. 
\end{thm}

The proof closely follows the arguments of Jaco-Sedgwick in~\cite{jaco2003decision}, and will use the following terminology:
\begin{defn}
Consider an ideal triangulation of a once-punctured torus with one ideal vertex and two triangles.
Two normal simple closed curves are called \emph{complementary} if their Haken sum is a multiple of the trivial vertex linking curve. 
\end{defn}

\begin{rmk}
If simple closed normal curves on a once punctured torus triangulated as above are parametrised by edge weights, then complementary pairs correspond to edge weight triples $(p,q,p+q)$ and $(r+s,s,r)$ where $p+q+r=p+r+s=q+s$ where $p\le q, r\le s$ without loss of generality. So $q=s=p+r$. An example is given by the curves with edge weights $(2,3,5)$ and $(4,3,1)$. 
\end{rmk}

\begin{proof} 
Let $\widehat{M}$ denote the result of truncating the standard cusp $\cusp$ of $M$, and $\widehat{\tri}$ the triangulation obtained from $\tri$ by removing $\cusp$. 

The first step is to observe that \cite[Prop.~3.7]{jaco2003decision} applies to properly embedded normal (and almost normal) surfaces in $\widehat{\tri}$. 
(The paper~\cite{jaco2003decision} studies a finite triangulation of a knot manifold restricting to a $1$-vertex triangulation of its boundary torus,  
but the result in our situation follows by an identical argument.) 
The conclusion is that any two compatible properly embedded normal surfaces have either the same slope or complementary slopes. (Compatible means there is no tetrahedron where the surfaces have different quadrilateral types.)

From this it follows immediately that if $\mathcal P$ is the projective solution space of closed and proper normal surfaces in $\widehat{\tri}$, then any face of $\mathcal P$ with all elements embedded surfaces has at most two boundary slopes. If there are two slopes, they are complementary. Finally if there are two complementary slopes, these will occur at vertices of the face. 

Note that a corollary, as in~\cite{jaco2003decision}, is that there is a computable finite set of boundary slopes $\mathcal{B}$ for properly embedded normal surfaces, with at most two slopes for each maximal face of embedded surfaces of $\mathcal P$.

Now, we note that if  $\tri(\alpha)$ is not $1$-efficient, then it must admit an embedded normal surface
$S$ with $\chi(S) \ge 0$ which is not an ideal vertex linking normal torus.  
Then $\widehat{S} = S \cap \widehat{M}$ is a properly embedded normal surface in $\widehat{\tri}$ with at least one boundary curve having non-trivial slope $\alpha$. Otherwise if $\widehat{S}$ has all boundary curves with trivial slope, then $\widehat{S}$ can be capped off in $M$, to give a normal surface in $\tri$ which contradicts the $1$-efficiency assumption (see \cite[Sec.~5]{jaco2003decision}).

Finally we observe that to check non-$1$-efficiency, it suffices to look at the slopes $\alpha \in \mathcal{B}$, and for each of these check the $1$-efficiency of the triangulation $\tri(\alpha)$. This last step can be efficiently carried out by standard methods, \textit{e.g.}~using Regina~\cite{regina}.
\end{proof}

%&&&&&&&&&&&&&&&&&&&&&&&&&&&&&&&&&&&&&&&&&&&&&&&&&&&&&&&&&&&&&&&&&&&&&&&&&&&&&
%&&&&&&&&&&&&&&&&&&&&&&&&&&&&&&&&&&&&&&&&&&&&&&&&&&&&&&&&&&&&&&&&&&&&&&&&&&&&&

\section{Collar efficiency}\label{sec:collar_eff}

In the course of the proof of our main Theorem~\ref{thm:main}, we will touch on a special version of $1$-efficiency for ideal triangulations, which we call \textit{collar efficiency}. Collar efficiency is closely related to annular efficiency of finite triangulations of manifolds with boundary, and is implied by $1$-efficiency (see~\cite{BJR}). 

Let us say that a triangulation is the result of $(0,1,1)$-Dehn filling, if it is obtained by gluing a degenerate layered soid torus LST$(0,1,1)$ in place of a cusp, as described in Section~\ref{sec:LSTs} (\textit{cf.}~Appendix~\ref{sec:relative_degenerate}). Note that  a $(0,1,1)$-Dehn filling can be performed by \textit{e.g.}~adding a two or three tetrahedra triangulation of a solid torus, as explained in~\cite[Fig.~5]{jaco2006layered}.
We prove that if $(0,1,1)$-Dehn filling is performed on a standard cusp in an ideal triangulation with at least two cusps, then either the resulting triangulation is not collar efficient or the filled manifold is a solid torus. So the filled triangulation is $1$-efficient only in the case the manifold is a solid torus or the product of a torus and an interval.

In the rest of this section we assume that all manifolds are irreducible, atoroidal, orientable and with non-empty boundary consisting of tori. 

\begin{defn}
An ideal triangulation is \emph{collar efficient} if it is $0$-efficient and any embedded normal torus is either a peripheral torus or is not isotopic to a peripheral torus. 
\end{defn}
\begin{lem}
Suppose $M$ is not a solid torus or a product of a torus and an open interval. 
If an ideal triangulation of $M$ has an edge which is contained in a topological collar of a peripheral torus, then the triangulation is not collar efficient.
\end{lem}

\begin{proof}
Suppose an ideal triangulation $\tri_M$ of $M$ has an edge $E$ contained in a topological collar $C = S^1 \times S^1 \times (0,1]$ of a peripheral torus. Let us denote by $T$ the torus boundary $S^1 \times S^1 \times \{1\}$ of $ C$.

We claim that the edge $E$ is a barrier to the normalisation of $T$ (see \cite[Sec.~10.3]{garoufalidis20163d} and~\cite{jaco20030}). 
Since $M$ is not a solid torus nor a product of a torus and an open interval,  normalisation of $T$ produces either a normal Klein bottle, normal $2$-sphere, normal projective plane or a normal torus which is isotopic to a peripheral torus but not normally equivalent to this peripheral torus.  Hence the triangulation of $M$ is not collar efficient. 

Note that if there is a normal Klein bottle, then the boundary of a small regular neighbourhood is a normal torus.

It is then sufficient to show that each normalisation move on $T$ can be performed without creating an intersection with $E$. We list these, allowing more than the minimal number of moves. 

\begin{enumerate}
\item Discard a component in the interior of a tetrahedron.
\item Perform a compression in the interior of a tetrahedron.
\item Isotope a disc across a face of a tetrahedron, where the boundary of the disc is in the face, and the interior of the disc is in the interior of the tetrahedron
\item Assume there is an arc $\gamma$ of intersection of $T$ with a face $F$ which has both ends on an edge; let $\delta$ be the subarc of this edge with $\partial \gamma = \partial \delta$. Suppose that the bigon bounded by $\gamma \cup \delta$  in $F$ has interior disjoint from $T$. Perform an isotopy in the neighbourhood of $F$, so that $\gamma$ slides across the bigon and across $\delta$. This removes two intersection points of $T$ with the edge. 
\item Assume there is a disc $D$ embedded in a tetrahedron so that $D \cap T = \gamma$ is an arc contained in $\partial D$ and $\partial D = \gamma \cup \delta$ with $\delta$ an arc in a face $F$ of the tetrahedron. Moreover suppose that $\partial \gamma =  \partial \delta$ are points on two different normal arcs of $T \cap F$.  Perform an isotopy of $T$ so that $\gamma$ slides along $D$ and across $\delta$.
\end{enumerate}

To conclude, we just need to observe that none of these moves will produce an intersection of $T$ and the edge $E$.
\end{proof}

\begin{rmk}
If we inflate an ideal triangulation which is collar efficient, the result is easily seen to be annular efficient. Conversely, if we crush the boundary torus of a finite triangulation, where there are no normal tori which are isotopic to the boundary but not normally parallel to the boundary, the result will be a collar efficient ideal triangulation.  The former condition clearly implies annular efficiency, by a simple barrier argument as in the above lemma. (See~\cite{BJR} for further details.) 
\end{rmk}
Finally, we apply this to the case of $(0,1,1)$-Dehn filling.

\begin{thm}\label{thm:collar_eff}
Suppose $M$ has an ideal triangulation with at least two cusps and one is standard. Perform $(0,1,1)$-Dehn filling on the standard cusp. 
If the resulting manifold is not an open solid torus or a product of a torus and an open interval then the triangulation is not collar efficient, hence not $1$-efficient. 
\end{thm}

\begin{proof}
The edge $E$ with weight $0$ on the boundary of the truncated standard cusp becomes isotopic into a peripheral torus after filling. Hence there is a topological collar containing $E$ and the triangulation is not collar efficient. 
\end{proof}
\begin{rmk}
Let $M$ be a hyperbolic $3$-manifold and $\tri$ a $1$-efficient triangulation of $M$ with a standard cusp. 
If Dehn filling along one of the three boundary edges gives a hyperbolic manifold, then $\tri$ gives an example of a triangulation which is collar efficient but not $1$-efficient. Otherwise the $3$D index of the filled hyperbolic manifold would be $0$, by Theorem~\ref{thm:trivial_slope}, but the index of a hyperbolic manifold is non-zero~\cite{garoufalidis20163d}. 
\end{rmk}

%&&&&&&&&&&&&&&&&&&&&&&&&&&&&&&&&&&&&&&&&&&&&&&&&&&&&&&&&&&&&&&&&&&&&&&&&&&&&&
%&&&&&&&&&&&&&&&&&&&&&&&&&&&&&&&&&&&&&&&&&&&&&&&&&&&&&&&&&&&&&&&&&&&&&&&&&&&&&

\section{3D index}\label{sec:3d_index}

We will provide two definitions of the $3$D index, following~\cite{garoufalidis20163d}. 
The first one will be based on $Q$-normal surfaces, while the second approach will instead use a state sum based on assignments of integer coefficients to edges. 
In the former approach we are interpreting the index as a generating function for normal surfaces with a given boundary.
For the latter, roughly speaking, just as $\tri$ is built from gluing tetrahedra as prescribed by the gluing matrix, the $3$D index will be analogously defined as a suitable sum of products of the $3$D indices' `building blocks', known as tetrahedral indices.~\\

%&&&&&&&&&&&&&&&&&&&&&&&&&&&&&&&&&&&&&&&&&&&&&&&&&&&&&&&&&&&&&&&&&&&&&&&&&&&&&
%&&&&&&&&&&&&&&&&&&&&&&&&&&&&&&&&&&&&&&&&&&&&&&&&&&&&&&&&&&&&&&&&&&&&&&&&&&&&&

\subsection{The tetrahedral index}\label{sec:tretrahedral_index}~\\

The ``building block'' of the $3$D index is the \emph{tetrahedral index}; formally this is the function $$\itet \colon \Z^2 \longrightarrow \Z(\!(q^{\frac12})\!)$$
defined as
\begin{equation}\label{eqn:tetrahedral_index}
\itet (m,e)(q) = \sum_{n = e_+}^\infty (-1)^n \frac{q^{\frac{n(n+1)}{2} - \left(n +\frac{e}{2} \right)m}}{(q;q)_n (q;q)_{n+e}},
\end{equation}
where $e_+ = \frac12 \left(|e| - e\right)$. Throughout the text, we will mostly suppress the dependency on $q$ in $\itet$ for ease of notation. 

The notation $(\alpha ;q^c)_n$, with $n \in \Z \cup \infty$, $\alpha \in \C$ and $c \in \Z$,  denotes the $q$-\emph{shifted factorial} (also known as the $q$-\emph{Pochhammer symbol}), defined as $$(\alpha ;q^c)_n = \prod_{i = 0}^{n-1} (1- \alpha q^{c i}).$$

We adopt the convention that $1/(q;q)_n = 0$ for negative values of $n$. \\Note that this convention can be seen as a direct consequence of the alternative definition $(a;q)_n = (a;q)_\infty /(a q^n;q)_\infty$. When clear from context, we will use the shorthand $(\alpha)_n$ for $(\alpha; q)_n$.

The tetrahedral index $\itet (m,e)$ can be shown (\cite[Sec.~3]{dimofte20133},\cite[App.~B]{garoufalidis20163d_angle}) to coincide with the coefficient of $z^e$ in the quotient
\begin{equation}\label{enq:generating_function_tetra}
\frac{(q^{1-\frac{m}{2}} z^{-1})_\infty}{(q^{-\frac{m}{2}} z)_\infty} = \sum_{e \in \Z} \itet (m,e) z^e .
\end{equation}

We will often make use of a more symmetric version of $\itet$, defined as the function $$\jtet\colon \Z^3 \longrightarrow \Z(\!(q^{\frac12})\!)$$
such that
\begin{equation}\label{eqn:simmetries_Itet}
\begin{gathered}
\jtet(a,b,c) =  \left(-q^\frac12\right)^{-b} \itet(b-c,a-b) = \left(-q^\frac12\right)^{-c} \itet(c-a,b-c) \\
 =\left(-q^\frac12\right)^{-a} \itet(a-b,c-a). 
\end{gathered}
\end{equation}
Note that $\jtet$ is unchanged under all permutations of its arguments. Moreover,
\begin{equation}\label{eqn_add_tets}
\jtet(a,b,c) = \left(-q^\frac12\right)^{s} \jtet(a+s,b+s,c+s),
\end{equation}
for all $s \in \Z$ (see~\cite{garoufalidis20151}).\\

The tetrahedral index is known to satisfy many relations and identities~(see \textit{e.g.}~\cite{garoufalidis20163d_angle}), which will be used throughout. For a more comprehensive approach on these relations, we refer to~\cite[Sec.~3]{garoufalidis20163d_angle}. 
The most straightforward identities are symmetries induced by an $\mathrm{SL}_2(\Z)$ action
\begin{equation}\label{eqn:tetrahedral_index_symmetries}
\itet (m,e) = \itet (-e,-m) = \left(-q^\frac12\right)^{-e} \itet(e,-e-m) = \left(-q^\frac12\right)^{m} \itet(-e-m,m).
\end{equation}

Next, there are two series of $3$-term relations, which combine the values of $\itet$ on three consecutive or adjacent and orthogonal lattice points in $\Z^2$ respectively:
\begin{equation}\label{eqn:index_3term_consecutive}
\begin{gathered}
\itet(m,e+1)  +\left( q^{e+\frac{m}{2}}  - q^{-\frac{m}{2}} -q^{\frac{m}{2}} \right)\itet(m,e)  + \itet(m,e-1) = 0\\
\itet(m+1,e)  +\left( q^{-m-\frac{e}{2}}  - q^{-\frac{e}{2}} -q^{\frac{e}{2}} \right)\itet(m,e)  + \itet(m-1,e) = 0
\end{gathered}
\end{equation}

\begin{equation}\label{eqn:index_3term_adjacent}
\begin{gathered}
q^\frac{e}{2} \itet(m+1,e)  +q^{-\frac{m}{2}} \itet(m,e+1)  - \itet(m,e) = 0\\
q^\frac{e}{2} \itet(m-1,e)  +q^{-\frac{m}{2}} \itet(m,e-1)  - \itet(m,e) = 0
\end{gathered}
\end{equation}
The two relations in equation~\eqref{eqn:index_3term_adjacent} can be combined to produce this symmetric version
\begin{equation}\label{eqn:four_term_relation}
q^{-\frac{m}2} \left[\itet(m, e-1) - \itet(m, e+1) \right] + q^{\frac{e}2}\left[\itet(m-1, e) - \itet(m+1, e) \right] =0.
\end{equation}
There are two other important relations, which play a crucial role in proving some form of invariance of the $3$D index~\cite{garoufalidis20151}. The first can be regarded as being a sort of orthogonality between series of tetrahedral indices, and is usually referred to as \emph{quadratic identity}; explicitly, if $m,c \in \Z$
\begin{equation}\label{eqn:quadratic_identity}
\sum_{e \in \Z} q^e \itet (m,e) \itet (m,e+c) = \delta_{c,0}.
\end{equation}
The last relation we need is referred to as \emph{pentagon relation}, and holds for $m_1,m_2,x_1,x_2,x_3 \in \Z$: 
\begin{equation}\label{eqn:pentagon_identity}
\begin{gathered}
\sum_{e \in \Z} q^e \itet (m_1 ,e +x_1) \itet (m_2 ,e+x_2) \itet (m_1 +m_2, e + x_3) =\\ q^{-x_3} \itet (m_1-x_2+x_3,x_1-x_3) \itet (m_2-x_1+x_3, x_2-x_3).    
\end{gathered}
\end{equation}
 
Finally, if $P(q,z)$ is a Laurent series in the variables $q$ and $z$, and $k \in \Z$, we will use the notation $\left[ P(q,z)\right]_{z^k}$ to denote the coefficient of $z^k$ in $P(q,z)$. \\

%&&&&&&&&&&&&&&&&&&&&&&&&&&&&&&&&&&&&&&&&&&&&&&&&&&&&&&&&&&&&&&&&&&&&&&&&&&&&&
%&&&&&&&&&&&&&&&&&&&&&&&&&&&&&&&&&&&&&&&&&&&&&&&&&&&&&&&&&&&&&&&&&&&&&&&&&&&&&

\subsection{Surface approach}\label{ssec:surface_index}~\\

To use $Q$-normal classes to define the $3$D index, we start by multiplicatively extending $\jtet$ to a function $\jtet:\Z^{3n} \rightarrow \Z(\!(q^{\frac12})\!)$ given by 
\begin{equation}
\jtet(a_1,b_1,c_1, \ldots, a_n,b_n,c_n) = \prod_{j=1}^n \jtet (a_j,b_j,c_j).
\end{equation}

Then, for $S \in \mathcal{Q}(\tri;\Z)$ define 
\begin{equation}\label{Isurface_def}
I(S) = \left(-q^{\frac12}\right)^{-\chi(S)} \jtet(S).
\end{equation}

This function $I$ is well-defined on the quotient $\mathcal{Q}(\tri;\Z)/\mathbb{T}$, where $\mathbb{T}$ denotes the subgroup spanned by tetrahedral solutions. We will next define the $3$D index as a sum of contributions $I(S)$ over suitable $Q$-normal classes $S$. 

\medskip

Each integer $Q$-normal class $S \in \mathcal{Q}(\tri;\Z)$  gives a 2-chain modulo $2$, representing a homology
class $[S]_2 \in \hh_2(M,\bd M; \Z_2)$, and its boundary gives a class in $[\bd S] \in \hh_1(\bd M;\Z)$ (see~\cite[Sec.~7,8]{garoufalidis20163d}). 
Taken together, they give a \emph{homology homomorphism} 
\begin{equation}
h\colon \mathcal{Q}(\tri;\Z) \longrightarrow \hh_2(M,\bd M; \Z_2) \times \hh_1(\bd M;\Z), \qquad S\mapsto ([S]_2,[\bd S]),
\end{equation}
with 
\begin{equation}\label{eq:kerh}
\Ker h = \mathbb{E}+\mathbb{T},
\end{equation}
where $\mathbb{E}$ is the subgroup of $\mathcal{Q}(\tri;\Z)$ generated by edge solutions. Moreover, 
\begin{equation}\label{eq:imh}
\Im h = \{ (a,b) \in \hh_2(M,\bd M; \Z_2) \times \hh_1(\bd M;\Z) \mid \bd_* a = b \mod{2} \},
\end{equation}
where $\bd_*: \hh_2(M,\bd M;\Z_2) \to \hh_1(\bd M;\Z_2)$ is the connecting homomorphism in the long exact sequence of the pair $(M,\bd M)$.
\begin{defn}
For $(a,b) \in \Im h$, we define the $3$D index with homology data $(a,b)$ as 
\begin{equation}
\ind_{\tri}^{(a,b)} (q) = \sum_{\substack{[S] \in\mathcal{Q}(\tri;\Z)/\mathbb{T} \\ [S]_2=a,\, \bd[S]=b}} I(S).
\end{equation}
In other words, this is sum of the index contributions $I(S)$ from coset representatives $S \in \mathcal{Q}(\tri;\Z)$ of cosets $[S] \in\mathcal{Q}(\tri;\Z)/\T$. 
\end{defn}

\begin{thm}[\cite{garoufalidis20163d}]\label{index_well_defined}
$\ind_\tri^{(a,b)}(q)$ is well-defined for every choice of $(a,b) \in \Im h$ if and only if $\tri$ is $1$-efficient. 
\end{thm}

By appealing to $3$d mirror symmetry, the authors of~\cite{dimofte2014gauge} argued that the index should not depend on the choice of a triangulation (with the exception of a subset of ``hard'' triangulations containing valence one edges), and therefore provide a topological invariant of $M$. The following result gives a restricted proof of this fact:
\begin{thm}\cite[Sec.~9]{garoufalidis20163d}\label{thm:pachner}
Each $\ind_\tri^{(a,b)}(q)$ is invariant under $2$-$3$ and $0$-$2$ Pachner moves and their inverses, provided all triangulations involved are $1$-efficient.
\end{thm}
Invariance of the index would then follow from the connectedness of the $1$-efficient Pachner graph, \textit{i.e.}~the graph whose vertices are $1$-efficient triangulations of a fixed manifold and edges are $2$-$3$ and $0$-$2$ Pachner moves preserving $1$-efficiency.

In the special case of a cusped hyperbolic manifold $M$, it is shown in~\cite{garoufalidis20151} that 
the index gives a \textit{topological invariant} $\ind_M$ by considering its value on any triangulation obtained by a regular subdivision of the \textit{Epstein-Penner} decomposition~\cite{epstein1988euclidean}. All these triangulations are related by $1$-efficient Pachner moves, hence the index is the same by Theorem~\ref{thm:pachner}.

Using a meromorphic extension of the index~\cite{garoufalidis2019meromorphic}, Garoufalidis and Kashaev were able to prove that the value of the $3$D index is constant on all triangulations of a manifold admitting a \textit{strict angle structure}.~\\

In this paper we henceforth consider the following \emph{total index for each peripheral class}: 
\begin{defn}
For $b \in \hh_1(\partial M;\Z)$ we define the $3$D index with boundary data $b$ as 
\begin{equation}\label{main_index_def}
\ind_{\tri}^b (q) =\sum_{\substack{[S] \in\mathcal{Q}(\tri;\Z)/\mathbb{T}\\ \bd[S]=b}} I([S]).
\end{equation}
\end{defn}

This is the sum $\sum_a \, \ind_{\tri}^{(a,b)} (q)$ over all $a \in  \hh_2(M,\bd M; \Z_2)$ such that $\bd_* a = b \mod 2$, so the results of Theorems~\ref{index_well_defined} and \ref{thm:pachner} also apply to this total index. This version of the index is also used in~\cite[App.~A]{hodgson2021asymptotics}, where it is denoted by $I^{\rm tot}_\tri(b)$.

Note also that the index $\ind_{\tri}^b (q)$ vanishes unless $b$ lies in the kernel 
$$K=\Ker(\hh_1(\bd M;\Z) \to \hh_1(M;\Z_2)).$$
This follows from the homology exact sequence for the pair $(M,\bd M)$.
In the special case that $M$ is the complement of a knot in $S^3$ (or any integral homology sphere), this is 
$$K=\{ 2 x \mu + y \lambda : x,y \in \Z \},$$
where $\mu$ and $\lambda$ represent the standard meridian and longitude.~\\

%&&&&&&&&&&&&&&&&&&&&&&&&&&&&&&&&&&&&&&&&&&&&&&&&&&&&&&&&&&&&&&&&&&&&&&&&&&&&&
%&&&&&&&&&&&&&&&&&&&&&&&&&&&&&&&&&&&&&&&&&&&&&&&&&&&&&&&&&&&&&&&&&&&&&&&&&&&&&

\subsection{Edge weight approach}\label{ssec:edge_index}~\\

Recall from Section~\ref{ssec:Q-normal} that each $Q$-normal class in $\mathcal{Q}(\tri;\R)$ can be written as a real linear combinations of edge, tetrahedral and peripheral curve solutions. For an integer class in $\mathcal{Q}(\tri;\Z)$ this can be refined using the work of Neumann~\cite{neumann92}, as explained in \cite[Sec.~7]{garoufalidis20163d}.

In general, by  \cite[Rmk.~7.6]{garoufalidis20163d}, every class in $\mathcal{Q}(\tri;\Z)$  can be written as 
\begin{equation}\label{half-int-combination}
\sum_{i=1}^n \left(x_i E_i+ y_i T_i\right) + \sum_{k=1}^r \left( p_k M_k + q_k L_k \right),
\end{equation}
for some \emph{half-integer} coefficients $x_i,y_i,p_k,q_k \in \frac12 \Z$.~\\

Let $$\mathcal{N}(\tri;\Z)= \{ S \in \mathcal{Q}(\tri;\Z) : \bd S = 0\}$$ denote the set of {\em closed integer $Q$-normal classes}. \\ 
Now, by~\cite[Rmk~7.2]{garoufalidis20163d}, $\mathbb{E}+\T \subseteq \mathcal{N}(\tri;\Z)$ is a subgroup of finite index with quotient $\mathcal{N}(\tri;\Z)/(\mathbb{E}+\T)  \cong \hh^1(\widehat{M};\Z_2)$, where $\widehat{M}$ is obtained from $M$ by coning each cusp to a point.

The simplest case is when the induced map $\hh_1(\bd M; \Z_2) \to \hh_1(M;\Z_2)$ is onto, 
for example, when $M$ is a knot or link exterior in a $\Z_2$ homology sphere.  Then
$\mathbb{E}+\T = \mathcal{N}(\tri;\Z)$ gives all closed normal classes (by~\cite[Rmk.~7.3]{garoufalidis20163d}), and we can compute the index $\ind_{\tri}^b (q)$
as follows.  Choose a set of $n-r$ edge solutions as in~\cite{garoufalidis20151}, say $E_1, \ldots, E_{n-r}$, whose cosets form an integer basis for $(\mathbb{E}+\T)/\T$. Given $b \in \Ker(\hh_1(\bd M;\Z) \to \hh_1(M;\Z_2))$ we can also choose a normal class $S_0 \in \mathcal{Q}(\tri;\Z)$ of the form~\eqref{half-int-combination} with $[\bd S_0]=b$.   Then
\begin{equation}
\ind_{\tri}^b (q) =\sum_{(k_1,\ldots,k_{n-r})\in\Z^{n-r}}\, I\left(S_0+\sum_{i=1}^{n-r} \, k_iE_i\right).
\end{equation}
We can think of this as a \emph{state sum} with integer weights $k_i$ assigned to the edges of $\tri$.

In general, we extend $\mathbb{E}+\T$ by adding half-integer multiples of edge and tetrahedral solutions to obtain $b_1, \ldots, b_{n-r} \in \mathcal{N}(\tri;\Z)$ giving an integer basis for $\mathcal{N}(\tri;\Z)/\T$.  Then 
\begin{equation}
\ind_{\tri}^b (q) =\sum_{(k_1,\ldots,k_{n-r})\in\Z^{n-r}}\, I\left(S_0+\sum_{i=1}^{n-r} \, k_i  b_i\right).
\end{equation}
See Section~\ref{ssec:computation} for more details and implementations.

%&&&&&&&&&&&&&&&&&&&&&&&&&&&&&&&&&&&&&&&&&&&&&&&&&&&&&&&&&&&&&&&&&&&&&&&&&&&&&
%&&&&&&&&&&&&&&&&&&&&&&&&&&&&&&&&&&&&&&&&&&&&&&&&&&&&&&&&&&&&&&&&&&&&&&&&&&&&&

\section{The Gang-Yonekura formula}\label{sec:GYformula}

Dehn filling is a fundamental operation relating $3$-manifolds, so it is natural to study how the $3$D index transforms under it. 
In~\cite{gang2018symmetry} Gang and Yonekura proposed a formula expressing the $3$D index of the manifold $M(\alpha)$ obtained by Dehn filling along a single primitive closed curve $\alpha$ on one torus boundary component $T \subset \partial M$ (see also~\cite{gang2018quantum}). 

This formula gives the index for $M(\alpha)$ as an infinite linear combination of the $3$D indices $\ind_M^\gamma$ for $M$, over certain boundary classes $\gamma \in \hh_1(T;\Z)$. 

More precisely, let $\mu, \lambda$ be a basis of $\hh_1(T;\Z)$, such that $\lambda \in \Ker\left(\hh_1(T;\Z) \to \hh_1 (M;\Z_2)\right)$, and $\alpha =x\mu + y \lambda $ be a primitive homology class in $\hh_1(T;\Z)$. Then the formula of~\cite{gang2018symmetry} to compute the index of a Dehn filling is presented in the following form
\begin{equation}\label{eqn:GY_original}
\ind_{M(\alpha)}(q) = \sum_{(m,e) \in \Z^2} \mathcal{K}(m,e,x,y,q) \ind^{(x,y)}_M (m,e;q).
\end{equation}
The key part of this formula is the kernel $\mathcal{K}$, which acts as a ``sieve'' on the pairs $(m,e)$:
\begin{equation}\label{eqn:kernel}
\mathcal{K}(m,e,x,y,q) = \frac{(-1)^{rm +2se}}{2}  \left( \delta_{xm+2ye,0} (q^\frac{rm+2se}{2} + q^{-\frac{rm+2se}{2}}) - \delta_{xm+2ye,2}  - \delta_{xm+2ye,-2}\right).
\end{equation}
Here $(r,s)$ are the coefficients representing a curve $\alpha^*$ on $T$ which is dual to the filling curve $\alpha$. 

In other words, assume the algebraic intersection number is $i(\lambda,\mu)=+1$, as with the usual orientation conventions for meridian, longitude of a knot. Then $\alpha^* = r \mu + s \lambda$ being dual to $\alpha$, means  
$$i(\alpha,\alpha^*)= \begin{vmatrix} r & s \\ x & y\end{vmatrix}=1.$$

Now $(m,e)\in \Z$ corresponds to $\gamma = 2e \mu- m \lambda \in \hh_1(\bd M; \Z)$ (see~\cite{garoufalidis20163d}), and
the classes $\gamma$ of this form are precisely the elements of $$K_T = \Ker\left(\hh_1(T;\Z) \to \hh_1\left(M;\Z_2\right) \right).$$
Then equation~\eqref{eqn:kernel} becomes 
\begin{align*}
{\mathcal K}(m,e,r,s,q) 
& = \frac{1}{2}(-1)^{ i(\alpha^*,\gamma)}\left(\delta_{i(\alpha,\gamma),0}(q^{\frac{1}{2} i(\alpha^*,\gamma)}+q^{-\frac{1}{2}i(\alpha^*,\gamma)})-\delta_{i(\alpha,\gamma),-2}-\delta_{i(\alpha,\gamma),2}\right).
\end{align*}
So any non-zero contribution to $\ind_{M(\alpha)}$ corresponds exactly to those $\gamma \in K_T$ such that $i(\alpha,\gamma)=0$ or $i(\alpha,\gamma)=\pm 2$.  

Denote by $|\gamma|$ be the minimal number of components for an embedded representative of $[\gamma]$. Then we can re-write equation~\eqref{eqn:GY_original} as
\begin{equation}\label{eqn:GYformula}
\ind_{\tri(\alpha)}(q) = \frac12 \left( \sum_{\substack{\gamma \in \hh_1(T;\Z)\\ \alpha \cdot \gamma = 0}} (-1)^{|\gamma|} \left( q^{\frac{|\gamma|}{2}} + q^{-\frac{|\gamma|}{2}}\right) \ind_{\tri}^\gamma(q) -  \sum_{\substack{\gamma \in \hh_1(T;\Z)\\ \alpha \cdot \gamma = \pm 2}} (-1)^{|\gamma|} \ind_{\tri}^\gamma(q) \right),
\end{equation}
where we interpret $\ind_\tri^\gamma$ to be $0$ if $\gamma \notin K_T$.

Here, for simplicity, we assume that the boundary data is $0$ on all components of $\partial M \setminus T$. 
(More generally, equation~\eqref{eqn:GYformula} extends to the case where we prescribe a homology class $\omega \in \hh_1(\partial M \setminus T;\Z)$.)\\ 

There are several issues with the original formulation of equation~\eqref{eqn:GYformula}; to mention a few: triangulations are not used throughout, $1$-efficiency is not used or defined for closed manifolds, and non-peripheral $\Z_2$ homology does not appear in the picture. 
Our main objective is to rigorously prove a version of equation~\eqref{eqn:GYformula}, and to address the ways in which it can fail to give a correct answer. 

\begin{rmk}\label{rmk:other_slopes}
A normal surface-based interpretation of equations~\eqref{eqn:GYformula} is not at all straightforward; indeed one would expect contributions from \textit{all} surfaces in the cusped manifold that extend to the filled in solid torus. Instead the only contributions appearing are from those surfaces whose boundary slopes algebraically intersect the filling slope $0$ or $\pm 2$ times  -- rather than every even intersection. It will follow from our general gluing formula (see Section~\ref{sec:gluing_formula}) that the contributions to the index from all other slopes cancel each other out.
\end{rmk}

We can finally state our main result:
\begin{thm}[Main theorem]\label{thm:main}
Let $M$ be a compact orientable $3$-manifold with boundary consisting of at least two tori, and let $T$ be one component of $\partial M$. Let $\tri$ be a $1$-efficient triangulation of $M$ with a standard cusp at $T$.
Given a filling curve $\alpha \subset T$, let $\tri(\alpha)$ be the ideal triangulation of $M(\alpha)$ obtained from $\tri$ by replacing the standard cusp with a layered solid torus with slope $\alpha$. 
If $\tri(\alpha)$ is $1$-efficient, the Gang-Yonekura formula~\eqref{eqn:GYformula} holds.
\end{thm}
The proof of this theorem will occupy the next three sections. 

\begin{rmk}
The main theorem applies as stated even in the degenerate cases corresponding to {\rm LST}$(0,1,1)$ and {\rm LST}$(1,1,2)$. However, care must be taken in interpreting the gluing construction in these settings. See Section~\ref{sec:collar_eff} and appendices~\ref{sec:gang011}, \ref{sec:relative_degenerate} for a detailed discussion on these degenerate fillings.
\end{rmk}

\begin{cor}
If $M$ and $\tri$ are as above, then the Gang-Yonekura formula~\eqref{eqn:GYformula} holds for almost all slopes $\alpha$. 
\end{cor}
\begin{proof}
This is a straightforward consequence of Theorem~\ref{thm:main} together with Theorem~\ref{thm:algorithm_efficiency}. 
\end{proof}

\begin{rmk}
Note that by Theorem~\ref{thm:algorithm_efficiency}, the finite number of slopes that do not yield a $1$-efficient triangulation can be algorithmically determined.
\end{rmk}

\begin{cor}
Let $M$ be a cusped hyperbolic $3$-manifold with at least two cusps, and assume that one cusp $T$ is generic, \textit{i.e.}~the corresponding Euclidean torus is not rectangular and there exists a unique shortest geodesic from $T$ to itself. Then the Gang-Yonekura formula~\eqref{eqn:GYformula} holds for almost all slopes $\alpha$ on $T$ and gives the topological invariants $\ind_M = \ind_\tri$ and $\ind_{M(\alpha)} = \ind_{\tri (\alpha)}$. 
\end{cor}
\begin{proof}
By the proof of Proposition~\ref{std_cusp2}, the canonical 
ideal triangulation $\tri$ has a standard cusp at $T$. 
Then the work of Gu\'{e}ritaud-Schleimer~\cite{gueritaud2010canonical} shows that for almost all slopes $\alpha$ on $T$, the canonical triangulation of $M(\alpha)$ is isomorphic to the triangulation $\tri (\alpha)$. 
\end{proof}

To assist with navigating the structure of the proof of Theorem~\ref{thm:main}, we include the following outline. We begin by introducing a definition of the relative index for ideal triangulations with exposed boundary (Definition~\ref{def:relative_index}). In Section~\ref{sec:gluing_formula}, we develop a gluing formula for relative indices, which enables us to formulate a relative version of the Gang-Yonekura formula (Proposition~\ref{prop:gang_cusp}). The key idea is to compare the effect on the relative index of modifying the layered solid torus versus altering the filling curve for a standard cusp.

The proof proceeds via induction on the number of tetrahedra in the added layered solid torus. 
The inductive step (Proposition~\ref{prop:inductive_step}) follows from an application of the pentagon identity (equation~\ref{eqn:pentagon_identity}). The base case, treated in Section~\ref{ssec:initial_cusp_formula}, is significantly more involved and requires new techniques drawn from the meromorphic extension of the $3$D index~\cite{garoufalidis2019meromorphic} and the theory of $q$-hypergeometric functions, as detailed in Section~\ref{sec:proof}.

%&&&&&&&&&&&&&&&&&&&&&&&&&&&&&&&&&&&&&&&&&&&&&&&&&&&&&&&&&&&&&&&&&&&&&&&&&&&&&
%&&&&&&&&&&&&&&&&&&&&&&&&&&&&&&&&&&&&&&&&&&&&&&&&&&&&&&&&&&&&&&&&&&&&&&&&&&&&&

\section{A gluing formula for the 3D index}\label{sec:gluing_formula}

Let $\tri$ be an ideal triangulation of $M$ and $\tri_0$ be a union of triangles in the $ 2$-skeleton of $\tri$, splitting $\tri$ into ideal triangulations $\tri_1,\tri_2$ of submanifolds $M_1, M_2$ respectively, as in Figure~\ref{fig:splitting}. (Note that $\tri_1$ and $\tri_2$ have exposed boundary, but with only ideal vertices.)
\begin{figure}[!htbp]
\centering
\includegraphics[scale=0.6]{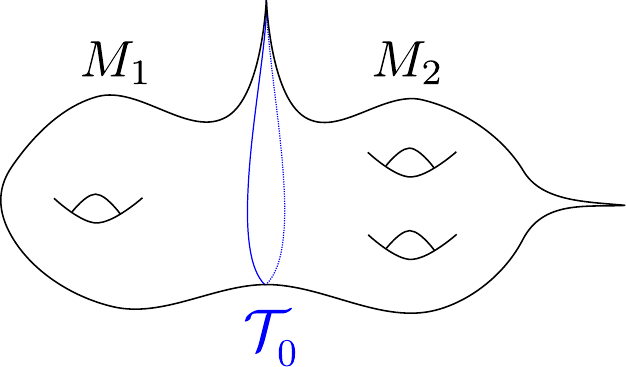}
\caption{A schematic depiction of the decomposition of $M$.}
\label{fig:splitting}
\end{figure}

Fix some $\omega \in 
K = \Ker\left(\hh_1(\partial M;\Z) \to \hh_1\left(M;\Z_2\right) \right)$
such that $\omega$ restricts to the trivial homology class on the cusps in $\tri_0$.

As explained in Section~\ref{ssec:Q-normal}, we regard an integer $Q$-normal class $S \in \mathcal{Q}(\tri;\Z)$ as a solution to the $Q$-matching equations along edges of $\tri$  which assigns an integer to each quad type in the triangulation. Restricting from quad types in $\tri$ to quad types in $\tri_1$ and $\tri_2$ gives natural projections from $\mathcal{Q}(\tri;\Z)$ to the subgroups $\mathcal{Q}(\tri_1;\Z)$ and $\mathcal{Q}(\tri_2;\Z)$, which give the solutions to the $Q$-matching equations on the \textit{interior} edges of $\tri_1,\tri_2$.

Given a $Q$-normal class $S\in \mathcal{Q}(\tri;\Z)$, this yields a decomposition $S = S_1 + S_2$, where $S_i \in  \mathcal{Q}(\tri_i;\Z)$ is supported on the quadrilaterals of $S$ contained in $\tri_i$. The boundary homology class $\omega = \partial S$ also decomposes as $\omega = \omega_1+\omega_2$, where $\omega_i \in \hh_1(\partial M \cap M_i;\Z)$.

As before, let $\E$ be the set of all edge classes of $\tri$, and \textit{choose} a collection $\E'$ of $r$ edges ``omitted'' in computing the index of $\tri$. The omitted edges need to be chosen in such a way that their union is a ``maximal tree with an odd cycle'' in the $1$-skeleton of $\tri$, see~\cite{garoufalidis20151}. Let $\E_i$ be the set of edges in $\tri_i$, and $\E_i'=\E_i \cap \E'$ for $i = 0,1,2$.

By~\cite[Rmk.~7.6]{garoufalidis20163d}, each integer $Q$-normal class $S \in\mathcal{Q}(\tri;\Z)$ can be written as a linear combination 
\begin{equation}
S=\sum_{i=1}^n \frac12 \left(x_i E_i+ y_i T_i\right) + \sum_{k=1}^r \frac12\left( p_k M_k + q_k L_k \right),
\end{equation}
where $x_i,y_i,p_k,q_k \in \Z$.
We can regard this as a linear combination of half-edge solutions $\frac12 E_i$, half-tetrahedral solutions $\frac12 T_j$ and half-peripheral curve solutions $\frac12 M_k,\frac12 L_k$, with coefficients in $\Z$. We can also require that \textit{the edge coefficients are $0$ on the set of omitted edges $\E'$}; then we say that the linear combination is \textit{admissible}. Note that an admissible linear combination is unique. 

This gives a projection $\beta_0$ taking $S \in \mathcal{Q}(\tri;\Z)$ to the coefficients of half-edge solutions for $S$ along $\E_0$, 
$$\beta_0(S) \colon \E_0 \to  \Z.$$ 

Then the matching equations along edges in $\tri_0$ show that $S_1, S_2$ satisfy the ``compatibility condition'' $\beta_0(S_1)=\beta_0(S_2)$. 
Conversely, any $S_1, S_2$ satisfying this compatibility condition gives a unique class $S \in \mathcal{Q}(\tri;\Z)$ with $\beta_0(S)=\beta_0(S_1)=\beta_0(S_2)$.
 
The class $S$ has formal Euler characteristic 
$$\chi(S)=\chi_1(S_1)+\chi_2(S_2)+\chi_0(\beta_0(S)),$$
where the contributions $\chi_1, \chi_2$ and $\chi_0$ are given as follows:   
\begin{equation}\label{eqn:chi_S_i}
\begin{aligned}
\chi_i(S_i) = & -2( \text{sum of coefficients for internal edge solutions in } \tri_i)\\  &- (\text{sum of coefficients for all tetrahedral solutions in }\tri_i),
\end{aligned}
\end{equation}
and
\begin{align}\label{eqn:chi_b_0}
\chi_0(\beta_0(S)) = -2(\text{sum of coefficients for edge solutions in }\tri_0).
\end{align}

Let $\mathcal{Q}^\omega(\tri;\Z)$ denote the set of $S\in \mathcal{Q}(\tri;\Z)$ with $\partial S = \omega$, and let $\mathcal{Q}_{\beta}^{\omega_i}(\tri_i;\Z)$ be the set consisting of $S_i \in \mathcal{Q}(\tri_i;\Z)$ with boundary $\omega_i$ and half-edge coefficients $\beta_0(S_i)=\beta$ on the edges of $\tri_0$. 
Then there are decompositions 
$$\mathcal{Q}^\omega(\tri;\Z) = \bigcup_{\substack{ \text{admissible}\\\beta\colon \E_0 \to \Z }}\left(\mathcal{Q}^{\omega_1}_{\beta}(\tri_1;\Z)\times \mathcal{Q}^{\omega_2}_{\beta}(\tri_2;\Z) \right)$$
and $$\mathbb{T} = \mathbb{T}_1 \oplus \mathbb{T}_2,$$ where $\mathbb{T}_i$ is the subgroup of $\mathcal{Q}(\tri_i;\Z)$ spanned by tetrahedral solutions in $\tri_i$.

\begin{defn}\label{def:relative_index}
We define the relative index for $\tri_i$ with half-edge coefficients $\beta\colon \E_0 \to \Z$ on $\tri_0$ as
$$\ind^{rel}_{\tri_i}(\omega;\beta)= \sum_{[S_i] \in \mathcal{Q}^\omega_{\beta}(\tri_i;\Z)/\mathbb{T}_i} \left(-q^{\frac12}\right)^{-\chi_i(S_i)} \jtet(S_i),$$
summed over coset representatives $S_i$  for $\mathcal{Q}^\omega_{\beta}(\tri_i;\Z)$ modulo the subgroup $\mathbb{T}_i$. 
Define the relative index to be zero if $\mathcal{Q}^\omega_{\beta}(\tri_i;\Z) = \emptyset$. 
\end{defn}

\begin{thm}[Gluing Theorem]\label{thm:gluing}
Assume that $\tri$ is $1$-efficient, and choose a set of omitted edges $\E' \subset \E$. Then there is a gluing formula:
\begin{equation}\label{eqn:gluing_formula}
\ind_\tri^\omega 
= \sum_{\substack{ \text{admissible}\\\beta\colon \E_0 \to \Z }}  \left(-q^{\frac12}\right)^{-\chi_0(\beta)} \ind^{rel}_{\tri_1}(\omega_1;\beta) \ind^{rel}_{\tri_2}(\omega_2;\beta).
\end{equation}
\end{thm}
\begin{proof}
Using Definition~\ref{def:relative_index}, the right-hand side of equation~\eqref{eqn:gluing_formula} is the multiple sum
$$\sum_{\text{admissible } \beta} \, \sum_{[S_1] \in \mathcal{Q}^{\omega_1}_{\beta}(\tri_1;\Z)/\mathbb{T}_1} \,  \sum_{[S_2] \in \mathcal{Q}^{\omega_2}_{\beta}(\tri_2;\Z)/\mathbb{T}_2}\left(-q^{\frac12}\right)^{-\chi_0(\beta)-\chi_1(S_1)-\chi_2(S_2)} \jtet(S_1)  \jtet(S_2).$$

Now the result follows from the expression for the index of $\tri$ in terms of $Q$-normal surfaces(see \eqref{main_index_def} and \eqref{Isurface_def}). Indeed we can write 
 $$\ind_\tri^\omega = \sum_{[S] \in \mathcal{Q}^\omega(\tri;\Z)/\mathbb{T}}  \left(-q^{\frac12}\right)^{-\chi(S)} \jtet(S)$$
summed over a set of coset representatives $S$ for integer $Q$-normal classes in $\tri$ modulo the subgroup $\mathbb{T}$ spanned by tetrahedral solutions.

Each term with $S = S_1 + S_2$, $\omega=\omega_1+\omega_2$ and $\beta_0(S_1)=\beta_0(S_2)=\beta$ can be factored as
$$\left(-q^{\frac12}\right)^{-\chi(S)} \jtet(S) = \left(-q^{\frac12}\right)^{-\chi_0(\beta)} \left(-q^{\frac12}\right)^{-\chi_1(S_1)} \jtet(S_1) \left(-q^{\frac12}\right)^{-\chi_2(S_2)} \jtet(S_2), $$
and the sum is absolutely convergent by $1$-efficiency of $\tri$, so can be rearranged as the above multiple sum.
\end{proof}

\begin{rmk}
It was proved in~\cite[Ex.~11.1]{garoufalidis20163d} that all $1$-efficient triangulations of the open solid torus with $\le 6$ tetrahedra have trivial $3$D index. On the other hand, we note here that  -- crucially for our results -- the relative index of layered solid tori can be non-trivial.
\end{rmk}

%&&&&&&&&&&&&&&&&&&&&&&&&&&&&&&&&&&&&&&&&&&&&&&&&&&&&&&&&&&&&&&&&&&&&&&&&&&&&&
%&&&&&&&&&&&&&&&&&&&&&&&&&&&&&&&&&&&&&&&&&&&&&&&&&&&&&&&&&&&&&&&&&&&&&&&&&&&&&

\section{Attaching layered solid tori}\label{sec:attaching_solid_tori}

Recall from Section~\ref{sec:LSTs} the setting we are using to implement Dehn fillings; given a triangulation $\tri$ of a cusped manifold $M$ with at least two cusps, and a surgery curve $\alpha$ in a standard cusp $\cusp$ of $\tri$, let $\tri(\alpha)$ be the ideal triangulation obtained from $\tri$ by replacing $\cusp$ by the layered solid torus LST$(\alpha)$.
\begin{figure}[!htbp]
\centering
\includegraphics[width=0.75\linewidth]{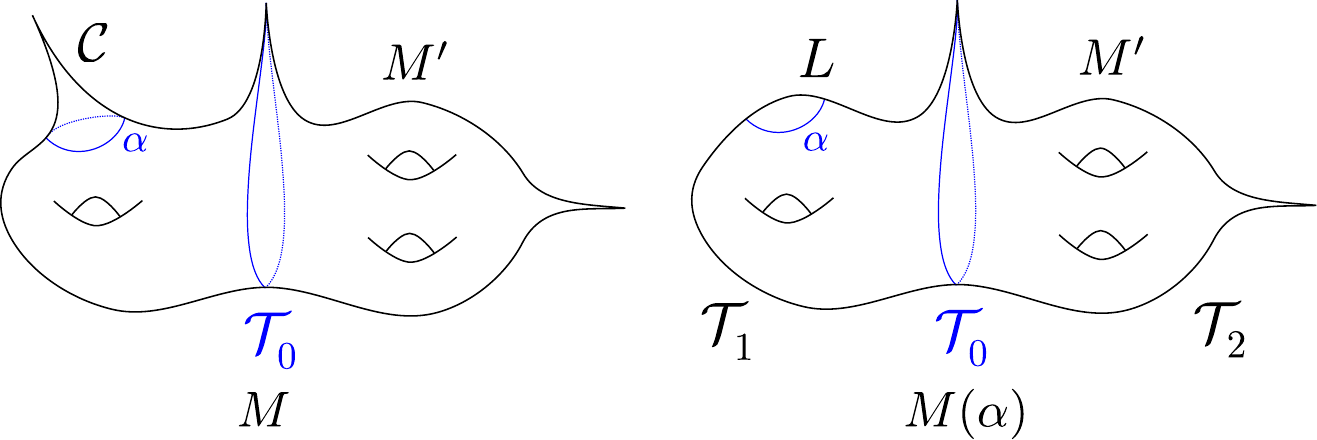}
\caption{A schematic for the effect of replacing a standard cusp $\cusp$ with a layered solid torus $L$ along the filling curve $\alpha$.}
\label{fig:before_after}
\end{figure} 

We compute the index of $\tri(\alpha)$ by applying the gluing formula in Theorem~\ref{thm:gluing}, where now $\tri_2$ is obtained from $\tri$ by removing the standard cusp, $\tri_1$ is a suitable layered solid torus, and the intersection $\tri_0$ is a once-punctured torus (as in Figure~\ref{fig:before_after}). 
Applying the gluing formula, with boundary data $\gamma$ on $\cusp$ and zero on all other cusps, gives
\begin{equation} \label{original_gluing}
\ind_\tri^{\gamma,0} 
= \sum_{\substack{ \text{admissible}\\\beta\colon \E_0 \to \Z }}  \left(-q^{\frac12}\right)^{-\chi_0(\beta)} \ind^{rel}_{\cusp}(\gamma;\beta) \ind^{rel}_{\tri_2}(0;\beta)
\end{equation}
and
\begin{equation} \label{filled_gluing}
\ind_{\tri(\alpha)}^{0} 
= \sum_{\substack{ \text{admissible}\\\beta\colon \E_0 \to \Z }}  \left(-q^{\frac12}\right)^{-\chi_0(\beta)} \ind^{rel}_{\text{LST}(\alpha)}(\beta) \ind^{rel}_{\tri_2}(0;\beta).
\end{equation}
On the other hand, applying~\eqref{original_gluing} to each term in the Gang-Yonekura formula gives
\begin{equation}\label{GY_gluing}
\sum_{\substack{ \text{admissible}\\\beta\colon \E_0 \to \Z }}  \left(-q^{\frac12}\right)^{-\chi_0(\beta)} GY(\alpha,\beta) \ind^{rel}_{\tri_2}(0;\beta),
\end{equation}
where $GY(\alpha,\beta)$ is obtained by applying the Gang-Yonekura formula~\eqref{eqn:GYformula} to the relative index for the standard cusp $\cusp$ (see Proposition~\ref{prop:gang_cusp} below).

Thus, comparing \eqref{filled_gluing} and \eqref{GY_gluing}, to prove the Gang-Yonekura formula for $\tri(\alpha)$ it suffices to show that the relative index of a layered solid torus agrees with the Gang-Yonekura formula applied to the relative index of a standard cusp:
$$GY(\alpha,\beta)=\ind^{rel}_{\text{LST}(\alpha)},$$
for all filling curves $\alpha$ in $\cusp$ and all admissible $\beta\colon \E_0 \to \Z$.  This is the main result of the section and will be proved in Theorem~\ref{thm:main_relative} below.

%&&&&&&&&&&&&&&&&&&&&&&&&&&&&&&&&&&&&&&&&&&&&&&&&&&&&&&&&&&&&&&&&&&&&&&&&&&&&&
%&&&&&&&&&&&&&&&&&&&&&&&&&&&&&&&&&&&&&&&&&&&&&&&&&&&&&&&&&&&&&&&&&&&&&&&&&&&&&

\subsection{Gang-Yonekura formula for the  standard cusp}\label{ssec:GY_for_cusp}~\\

In this section we first compute the relative index for the standard cusp $\cusp$ introduced in Section~\ref{sec:standard_cusp}, then use the resulting expression to study a relative version of the Gang-Yonekura formula for $\cusp$. We use the triangulation for  $\cusp$ with peripheral curves $\lambda, \mu$ and edges labelled as in Figure~\ref{fig:cusp_conventions}. 

\begin{prop}\label{relcuspindex}
The relative index for the standard cusp $\cusp$ with homological boundary $\alpha = x\lambda +y \mu$
and \textit{boundary half-edge coefficients} for $e_1, e_2,e_3$  given by $\underline b = (b_1,b_2,b_3)$ is defined for $b_1,b_2,b_3,x,y \in \Z$ and satisfies 
\begin{equation}
\label{eqn:cusp_index2}
\ind^{rel}_\cusp (\alpha;\underline{b}) = \left(-q^{\frac12}\right)^{-b_3}\itet\left( \frac{ x- b_1+b_3}{2}, \frac{ -y+b_2-b_3}{2}\right)\itet\left( \frac{-x-b_1+b_3}{2}, \frac{ y+ b_2-b_3}{2}\right),
\end{equation}
when $b_1 -b_3 \equiv x\mod{2}$,  $b_2-b_3 \equiv y  \mod{2}$, and is zero otherwise.
\end{prop}
We will also sometimes write this relative index as $\ind^{rel}_\cusp (x,y;b_1,b_2,b_3)$.

\begin{proof}
To compute the relative index of the standard cusp $\cusp$, we omit the internal edge $e_0$ shown in Figure~\ref{fig:cusp_conventions} and label the tetrahedral, edge and peripheral solutions as consistently with Table~\ref{table:cusp_gluing}. 
As in~\cite[Rmk~7.6]{garoufalidis20163d}, each integer normal class $S \in \mathcal{Q}(\cusp;\Z)$ can be written as an \emph{integer} linear combination of half-tetrahedron, half-edge and half-peripheral solutions: 
\begin{equation}\label{eqn:linear_combinations_cusp}
\begin{aligned}
S & =  \frac12 \left(t_1 T_1 + t_2 T_2 + b_1 E_1 + b_2 E_2 + b_3 E_3 + x L + y M \right)  \\ &= 
\frac12\left(b_1+t_1-x,b_2+t_1-y,b_3+t_1,b_1+t_2+x,b_2+t_2+y,b_3+t_2\right),
\end{aligned}
\end{equation}
where $t_1,t_2,b_1,b_2,b_3,x,y \in \Z$.
Note this has integer entries if and only if $t_1\equiv t_2\equiv b_3 \mod{2}, \,x\equiv b_1+b_3 \mod{2}$, and  $y\equiv b_2+b_3 \mod{2}$.

Conversely, given any $S=(q_1,q_2,q_3,q_4,q_5,q_6) \in \Z^6$, we 
can write $S$ as a linear combination as in equation~\eqref{eqn:linear_combinations_cusp}, with   $t_1 = 2 q_3$, $t_2 = 2 q_6$, $b_1 = q_1 - q_3 + q_4 - q_6$, $b_2 = 
 q_2 - q_3 + q_5 - q_6$, $b_3=0$, $x = -q_1 + q_3 + q_4 - q_6$, $y = -q_2 + q_3 + q_5 - q_6$.\\

To compute the relative index we work with integer normal classes \textit{modulo integer linear combinations of tetrahedral solutions}; this means half-tetrahedral solutions are needed as coset representatives if $b_3 \equiv 1 \mod{2}$.

The homological boundary of $S$ is equal to $\alpha = \bd S= x\lambda +y \mu $,
the boundary half-edge coefficients for $e_1, e_2,e_3$ are given by $\underline b = (b_1,b_2,b_3),$ and the contribution to Euler characteristic in \eqref{eqn:chi_S_i} is given by
$-\frac12(t_1+t_2)$. \\

From~\eqref{eqn:linear_combinations_cusp}, we see that the contribution of $S$ to the relative index of $\cusp$ 
is given by 
\begin{equation}\label{eqn:contribution_of_S}
\left(-q^{\frac12}\right)^{\frac12(t_1+t_2)}\jtet\left(\frac{b_1 +t_1 - x}{2}, \frac{b_2 +t_1 - y}{2}, \frac{b_3+t_1}{2}\right)\jtet\left( \frac{b_1 + t_2+ x}{2}, \frac{b_2 + t_2+y}{2}, \frac{b_3+t_2}{2}\right),
\end{equation}
for any choice of $t_1,t_2 \in \Z$ such that $t_1\equiv t_2\equiv b_3 \mod{2}$. \\
Taking $t_1=t_2=-b_3$ gives
\begin{equation}
\label{eqn:cusp_index}
\begin{aligned}
& \ind^{rel}_\cusp (\alpha;\underline{b}) = \ind^{rel}_\cusp (x, y ;b_1,b_2,b_3) \\ &=
\left(-q^{\frac12}\right)^{b_3}\jtet\left(\frac{b_1 -b_3 - x}{2}, \frac{b_2 -b_3 - y}{2}, 0\right)\jtet\left( \frac{b_1 - b_3 + x}{2}, \frac{b_2 -b_3+y}{2}, 0\right).
\end{aligned}
\end{equation}
Equivalently, we can rewrite this in terms of $\itet$ using equation~\eqref{eqn:simmetries_Itet} to give the result.
\end{proof}

\begin{rmk}\label{rmk:cusp_index_symmetries}
Equation~\eqref{eqn_add_tets} implies that we can reduce to the case where $b_3=0$: 
\begin{equation}\label{eq:cusp_index_b3=0}
\ind^{rel}_\cusp (x, y ;b_1,b_2,b_3)= \left(-q^{\frac12}\right)^{-b_3} \ind^{rel}_\cusp (x, y ;b_1-b_3,b_2-b_3,0).
\end{equation}
Furthermore, the relative cusp index exhibits several symmetries; it follows from~\eqref{eqn:cusp_index}  that 
\begin{equation}\label{eqn:cusp_symmetry1}
\ind^{rel}_\cusp (x,y;\underline{b}) = \ind^{rel}_\cusp (-x,-y;\underline{b}),
\end{equation}
while using  equation~\eqref{eqn:tetrahedral_index_symmetries} we get
\begin{equation}\label{eqn:cusp_symmetry2}
\ind^{rel}_\cusp (x,y;b_1,b_2,b_3) = \ind^{rel}_\cusp (y,x;b_2,b_1,b_3)  
\end{equation}
and 
\begin{equation}\label{eqn:cusp_symmetry3}
\ind^{rel}_\cusp (x, y; b_1, b_2,b_3)  =  \ind^{rel}_\cusp(-y, x - y; b_3,b_1,b_2)= \ind^{rel}_\cusp(y-x, -x; b_2,b_3,b_1).
\end{equation}
\end{rmk}

Since the Gang-Yonekura formula in equation~\eqref{eqn:GYformula} involves contributions of the form $q^{-\frac{|\gamma|}{2}}$ (which might introduce terms with diverging negative exponents), it is important to understand the growth of the minimal degree of cusp's relative index. 
\begin{lem}\label{lem:growth_cusp_index}
The minimal $q^{\frac{1}{2}}$-degree of $\ind^{rel}_\cusp (x, y ;b_1,b_2,b_3)$ grows quadratically as a function of $x$ and $y$. More precisely, for any fixed $b_1,b_2,b_3 \in \Z$, and $(x,y) \in \Z^2$ satisfying $x \equiv b_1 +b_3  \mod{2}$ and $y  \equiv b_2+b_3 \mod{2}$, we have 
\begin{equation}\label{cusp_deg_asympt}\lim_{h(x,y)\to +\infty} \frac{\deg \ind^{rel}_\cusp (x, y ;b_1,b_2,b_3)}{h(x,y)^2} = \frac14,\end{equation}
where $h$ denotes the ``hexagonal norm'' on $\Z^2$ defined by
$$h(x,y)= \max(|x|,|y|,|x-y|).$$
\end{lem}
\begin{proof}
From~\cite[Sec~4.2]{garoufalidis20163d},  the minimal $q^{\frac12}$ degree of $\itet(m,e)$ is given by the piecewise quadratic function 
\begin{equation}\deg \itet(m,e) = \begin{cases} m(e+m)+m, &\text{if } m\ge0 \text{ and } e+m\ge0\\
 -e m, &\text{if } m\le 0 \text{ and } e\ge 0\\
 e(e+m)-e, &\text{if } e\le 0 \text{ and } e+m \le 0.
 \end{cases}\end{equation}

It follows, using \eqref{eqn:cusp_index2}, that
\begin{equation}\label{exact_cusp_index_degree}
\deg \ind^{rel}_\cusp (x, y ;0,0,0)= \frac14 h(x,y)^2+ \frac12 h(x,y) \text{ for all } x,y\in 2\Z.\end{equation}
To prove this, we consider $$f(a,b) = q^{\frac12}\text{-degree  of } \itet(a,-b)\itet(-a,b).$$
If $a,b$ are integers with $0\le b \le a$, then  
\begin{align*} f(a,b) &= \deg \itet(a,-b) + \deg \itet(-a,b) \\ 
&=a(a-b)+a + ab 
 = a^2 +a \\
&=h(a,b)^2+h(a,b),\end{align*} since $h(a,b)=a$. Using the symmetries of the relative cusp index from Remark~\ref{rmk:cusp_index_symmetries} and the symmetries of the hexagonal norm, it follows that $f(a,b)=h(a,b)^2+h(a,b)$ for all $a,b \in \Z$. Putting $a=\frac{x}{2}, b=\frac{y}{2}$ gives equation~\eqref{exact_cusp_index_degree}. In particular, this shows that equation~\eqref{cusp_deg_asympt} holds when $b_1=b_2=b_3=0$.

In general, $\deg \ind^{rel}_\cusp (x, y ;b_1,b_2,b_3)$ differs from $\deg \ind^{rel}_\cusp(x, y ;0,0,0)$ by at most a piecewise affine function of $|x|,|y|$ with coefficients depending on $b_1,b_2,b_3$. This gives the result. 
\end{proof}

Next we study the relative Gang-Yonekura formula for the standard cusp. 

\begin{prop}\label{prop:gang_cusp}
Let $\alpha\in \hh_1(\partial \cusp ;\Z)$ denote the homology class of a filling curve, and $\beta$ a dual class such that $\alpha\cdot\beta=\pm1$. 
Then the relative version 
of the Gang-Yonekura formula~\eqref{eqn:GYformula} for the standard cusp $\cusp$ with boundary half-edge coefficients $\underline{b} = (b_1,b_2,b_3)$ is given by 
\begin{equation}\label{eqn:GY_for_cusp}
GY(\alpha;\underline{b}) = \frac{1}{2} \left[\sum_{k \in \Z} (-1)^k\left(\left(q^\frac{k}{2} + q^{-\frac{k}{2}}\right)\ind^{rel}_\cusp (k \alpha;\underline{b}) - \ind^{rel}_\cusp (k \alpha + 2\beta;\underline{b}) - \ind^{rel}_\cusp (k \alpha - 2\beta;\underline{b})\right) \right] .
\end{equation}
\end{prop}
If $\alpha = x \lambda + y \mu$, we will also write $GY(x,y;b_1,b_2,b_3)$ to denote $GY(\alpha;\underline{b})$.
\begin{proof}
We apply the Gang-Yonekura formula~\eqref{eqn:GYformula} where $\gamma \in \hh_1(\partial \cusp ;\Z)$. Note that if $\alpha\cdot\gamma=0$ then $\gamma=k \alpha$ where $k \in \Z$ and $(-1)^{|\gamma|} = (-1)^k$. 
Further, since $\alpha\cdot\beta=\pm1$, if $\alpha\cdot\gamma=\pm 2$ then $\gamma = k \alpha \pm 2 \beta$ where $k \in \Z$ and $(-1)^{|\gamma|} = (-1)^k$.
\end{proof}

\begin{rmk} \label{parity_of_k}
By Proposition~\ref{relcuspindex}, we note that the only non-zero terms in the sum~\eqref{eqn:GY_for_cusp}  occur when $(b_1-b_3,b_2-b_3)\equiv(0,0) \mod{2}$ and $k$ is even, or when  $(b_1-b_3,b_2-b_3)\equiv(x,y) \mod{2}$ and $k$ is odd. 
\end{rmk}

We can again reduce to the case where $b_3=0$ using \eqref{eq:cusp_index_b3=0}, which gives 
\begin{equation}\label{eq:GY_b3=0}
GY(x, y; b_1, b_2, b_3) = \left(-q^{\frac12}\right)^{-b_3} GY(x, y; b_1 - b_3, b_2 - b_3,0).
\end{equation} 
Further, the symmetries of the relative cusp index given in Remark~\ref{rmk:cusp_index_symmetries} also give corresponding symmetries in the standard cusp's Gang-Yonekura formula; for example 
\begin{equation}\label{eq:GYsymmetry1}
GY(x, y; b_1, b_2, b_3)=GY(-x, -y; b_1, b_2, b_3),
\end{equation}
and 
\begin{equation}\label{eq:Usymmetry}
GY(x, y; b_1, b_2, b_3) 
= GY(-y, x - y; b_3, b_1, b_2)
= GY(y - x, -x; b_2, b_3, b_1) .
\end{equation}

In particular, by the symmetry~\eqref{eqn:cusp_symmetry1}, equation~\eqref{eqn:GY_for_cusp} is equivalent to the compact version 
\begin{equation}\label{eqn:compact_version}
GY(\alpha;\underline{b}) = \sum_{k\in\Z} (-1)^k \left(q^{\frac{k}{2}} \ind_\cusp^{rel}(k \alpha;\underline{b}) - \ind^{rel}_\cusp(k \alpha +2\beta;\underline{b}) \right),
\end{equation}
where as before $\alpha\cdot\beta=\pm1$. 

\medskip 
The starting point for our inductive argument to prove our main result Theorem~\ref{thm:main}, will involve Dehn filling along the curve $\alpha = -\lambda + \mu$. As previously discussed in Section~\ref{sec:LSTs}, this is a degenerate case, corresponding to the layered solid torus LST$(1,1,2)$. We need to choose this layered solid torus as our base case as a consequence of the results in Section~\ref{sec:collar_eff} (see also Appendix~\ref{sec:gang011}); starting from the more `natural' LST$(0,1,1)$ yields either a trivial or non-converging index. 

We use the notation from the start of Section~\ref{sec:attaching_solid_tori}, where $\tri_2$ is obtained from $\tri$ by removing the standard cusp. We label the three edges $e_1,e_2,e_3$ in the boundary of $\tri_2$ to match the notation for the standard cusp $\cusp$ in Figure~\ref{fig:cusp_conventions}.
Then the triangulation $\tri(\alpha)$ is obtained from $\tri_2$ by folding the two triangles in the torus boundary across the edge $e_3$, as in Case 1 of Section~\ref{sec:LSTs}. 
To compute the index of the filled triangulation $\tri(\alpha)$, we take $Q$-normal classes in $\tri_2$ which are compatible with the folding operation; 
this means that the boundary half-edge coefficients $b_1$ and $b_2$ must satisfy $b_1 = b_2$. Then the contribution from the formal Euler characteristic of the edge $e_1 = e_2$ is given by $\left(-q^{\frac12}\right)^{-b_1} =\left(-q^{\frac12}\right)^{-b_2} $. 
We also choose $e_3$ as an excluded edge. 
This leads to the following formula for the index of the filled triangulation $\tri(\alpha)$:
\begin{equation}\label{eqn:tria_alpha_index}
I_{\tri(\alpha)} = \sum_{b_1,b_2 \in \Z} \left(-q^{\frac12}\right)^{-b_1} \delta_{b_1,b_2} \ind^{rel}_{\tri_2}(b_1,b_2,0) .
\end{equation}
Comparing this with the Gluing Theorem~\ref{thm:gluing}, we think of the the factor 
$\left(-q^{\frac12}\right)^{-b_1} \delta_{b_1,b_2}$ as the relative index of the solid torus $\ind^{rel}_{\text{\rm LST}(1,1,2)}(b_1,b_2,0)$. Using equation~\eqref{eq:cusp_index_b3=0}, this also gives 
$$\ind^{rel}_{\text{\rm LST}(1,1,2)}(b_1,b_2,b_3)=\left(-q^{\frac12}\right)^{-b_1} \delta_{b_1,b_2} \text{ for all } (b_1,b_2,b_3)\in \Z^3.$$ (This agrees with the computation in Appendix~\ref{sec:relative_degenerate} using Regina's two tetrahedron triangulation for LST$(1,1,2$.)

\medskip
Therefore to start our induction, we want to show the following 
\begin{thm}\label{thm:new_thm_112}
For all $\underline b = (b_1,b_2,b_3)\in \Z^3$ we have
\begin{equation}\label{eqn:GY_conjecture}
 GY(-\lambda + \mu;\underline{b}) = 
\left(-q^{\frac12}\right)^{-b_1} \delta_{b_1,b_2} = \left(-q^{\frac12}\right)^{-b_2} \delta_{b_1,b_2}. 
\end{equation}
\end{thm}

\begin{rmk} Exploiting the symmetries given in equation~\eqref{eq:Usymmetry} we can deduce the following equivalent results: 
\begin{equation}\label{eqn:GY_conjecture2}
 GY(2\lambda + \mu;\underline{b}) = 
\left(-q^{\frac12}\right)^{-b_3} \delta_{b_1,b_3} = \left(-q^{\frac12}\right)^{-b_1} \delta_{b_1,b_3}  ,
\end{equation}
\begin{equation}\label{eqn:GY_conjecture3}
 GY(\lambda + 2\mu;\underline{b}) = 
\left(-q^{\frac12}\right)^{-b_3} \delta_{b_2,b_2} = \left(-q^{\frac12}\right)^{-b_2} \delta_{b_2,b_3} ,
\end{equation}
\end{rmk}

The proof of Theorem~\ref{thm:new_thm_112} turns out to be the main technical hurdle of the paper, and it will occupy Section~\ref{ssec:initial_cusp_formula} and most of Section~\ref{sec:proof}.\\

%&&&&&&&&&&&&&&&&&&&&&&&&&&&&&&&&&&&&&&&&&&&&&&&&&&&&&&&&&&&&&&&&&&&&&&&&&&&&&
%&&&&&&&&&&&&&&&&&&&&&&&&&&&&&&&&&&&&&&&&&&&&&&&&&&&&&&&&&&&&&&&&&&&&&&&&&&&&&

\subsection{Proof of initial cusp index formula}\label{ssec:initial_cusp_formula}~\\

The initial case of the Gang–Yonekura formula for the standard cusp given in Theorem~\ref{thm:new_thm_112} is equivalent, using \eqref{eq:cusp_index_b3=0} and \eqref{eqn:compact_version} with $\alpha = -\lambda + \mu$ and $\beta = -\lambda$, to the following statement. Its proof will occupy the rest of this section, as well as Section~\ref{sec:proof}.

\begin{thm}\label{thm:gang112}
For every $\underline{b} = (b_1,b_2,0) \in \Z^3$
\begin{equation}\label{eqn:compact_for_112}
\sum_{k \in \Z} (-1)^k q^{\frac{k}{2}} \ind^{rel}_\cusp (-k,k;\underline{b}) - (-1)^k\ind^{rel}_\cusp (-k,k-2;\underline{b})
= \left(-q^{\frac12}\right)^{-b_2}\delta_{b_1,b_2}.
\end{equation}
\end{thm}

Rather than tackling the theorem directly, our approach begins with constructing a generating function in the variable $z$, where the coefficient of a specific power of $z$ yields the left-hand side of equation~\eqref{eqn:compact_for_112}.\\

More precisely, for every $r \in \Z$, we  define the generating function
\begin{equation}\label{eqn:seriesindexshifted}
\varphi_r(z,q) = \sum_{e \in \Z}  \itet (e-r,e) z^e. 
\end{equation}
In other words, $\varphi_r(z,q)$ is keeping track of tetrahedral indices arranged along a shifted diagonal in the $\Z^2$ lattice.\\ We will prove functional relations involving $\varphi_r(z,q)$, and prove they imply further relations for $\ind^{rel}_\cusp $, which will ultimately lead to a proof of 
Theorem~\ref{thm:gang112} and Theorem~\ref{thm:new_thm_112}.

\begin{rmk}\label{rmk:symmetry_and_compact_version}
The symmetry $\itet (m,e) = \itet(-e,-m)$ implies that $$\varphi_r (z,q) = z^r \varphi_r (z^{-1},q).$$
More generally, relations for $\itet$ from Section~\ref{sec:tretrahedral_index} imply relations for $\varphi_r (z,q)$.
\end{rmk}

With this in place, we can showcase the main functional equation involving $\varphi_r(z,q)$ we will need in what follows:
\begin{prop}\label{prop:generating_function_noncompact}
The left-hand side of equation~\eqref{eqn:compact_for_112} is the coefficient of $(-z)^{b_2}$ in the following generating function: 
\begin{equation}\label{eqn:generating_function_noncompact}
\varphi_r (z q^{\frac12},q) \varphi_r (z q^{-\frac12},q) - \varphi_{r+1} (z,q) \varphi_{r-1} (z,q),
\end{equation}
where $r = \frac{b_1 + b_2}2$ and $b_1 \equiv b_2\mod{2}$.
\end{prop}
\begin{proof}

Consider the left-hand side of equation~\eqref{eqn:compact_for_112}.
By Remark~\ref{parity_of_k}, the only non-zero terms occur when $(b_1,b_2) \mod{2} $ is a multiple of $(-1,1) \mod{2}$, so there are two cases to consider:
\begin{itemize}
\item For $b_1\equiv b_2 \equiv 0 \mod{2}$ the non-zero terms have $k$ even and the sign $(-1)^k$ is always $+1$;
\item For $b_1 \equiv b_2\equiv 1 \mod{2}$ the non-zero terms have $k$ odd and the sign $(-1)^k$ is always $-1$. 
\end{itemize}
So we can write the left-hand side of equation~\eqref{eqn:compact_for_112} explicitly as 
\begin{align*}\label{eqn:explicit_compact_for_112}
\,\,& 
\varepsilon(\underline{b})\left[\, \sum_{\substack{k \in \Z\\ k\equiv b_1 \equiv b_2 \mod{2}}} q^{\frac{k}{2}} \itet \left(\frac{k-b_1}{2},\frac{k+b_2}{2}\right)\itet \left(\frac{-k-b_1}{2},\frac{-k+b_2}{2}\right) \right.\\ & \hspace{1cm} \left.-   \itet \left(\frac{k-b_1+2}{2},\frac{k+b_2}{2}\right) \itet \left(\frac{-k-b_1-2}{2},\frac{-k+b_2}{2}\right) \right],
\end{align*}
where $\varepsilon(\underline{b})$ is $1$ when $b_1\equiv b_2 \equiv 0 \mod{2}$,  $-1$ if $b_1 \equiv b_2\equiv 1 \mod{2}$, and $0$ otherwise.\\
Observe that $\varepsilon(\underline{b})=(-1)^{b_2}$ if $b_1\equiv b_2 \mod{2}$, and is $0$ otherwise.\\

Now we note that
$$\itet \left(\frac{k-b_1}{2},\frac{k+b_2}{2}\right) =\itet(e-r,e) \,\,\text{ and } \,\,
\itet \left(\frac{-k-b_1}{2},\frac{-k+b_2}{2}\right) =\itet(e'-r,e'),$$ 
where $e=\frac{k+b_2}{2}$, $e'=\frac{-k+b_2}{2}$ and $r=\frac{b_1+b_2}{2}$ are integers. \\
Further, $e$ and $e'$ both vary over all of $\Z$ when $k$ varies over the integers satisfying $k \equiv b_2 \mod{2}$.

It follows from~\eqref{eqn:seriesindexshifted} that the double sum
\begin{gather*}
\sum_{\substack{k_1,k_2 \in \Z\\ k_1\equiv k_2 \equiv b_2 \mod{2}}}
\itet \left(\frac{k_1-b_1}{2},\frac{k_1+b_2}{2}\right) \itet \left(\frac{-k_2-b_1}{2},\frac{-k_2+b_2}{2}\right) z_1^{\frac{k_1+b_2}{2}} z_2^{\frac{-k_2+b_2}{2}}
\end{gather*}
coincides with the product of generating functions  $$\varphi_r (z_1,q) \varphi_r (z_2,q).$$
Therefore the evaluation $z_1 = q^{\frac12} z$ and $z_2 = q^{-\frac12} z$ yields
\begin{gather}\label{eqn:double_sum1}
z^{b_2}
\sum_{\substack{k_1,k_2 \in \Z\\ k_1\equiv k_2 \equiv b_2 \mod{2}}}
q^{\frac{k_1+k_2}{4}} \itet \left(\frac{k_1-b_1}{2},\frac{k_1+b_2}{2}\right) \itet \left(\frac{-k_2-b_1}{2},\frac{-k_2+b_2}{2}\right) z^{\frac{k_1-k_2}{2}}.
\end{gather}
Similarly, the product of generating functions  
$$\varphi_{r-1}(z_1,q) \varphi_{r+1}(z_2,q)$$
gives the double sum
\begin{gather*}
\sum_{\substack{k_1,k_2 \in \Z\\ k_1\equiv k_2 \equiv b_2 \mod{2}}}
\itet \left(\frac{k_1-b_1+2}{2},\frac{k_1+b_2}{2}\right) \itet \left(\frac{-k_2-b_1-2}{2},\frac{-k_2+b_2}{2}\right) z_1^{\frac{k_1+b_2}{2}} z_2^{\frac{-k_2+b_2}{2}}.
\end{gather*}
This time, using the specialisation $z_1 = z_2 = z$ in the double sum gives
\begin{gather}\label{eqn:double_sum2}
z^{b_2}
\sum_{\substack{k_1,k_2 \in \Z\\ k_1\equiv k_2 \equiv b_2 \mod{2}}}
\itet \left(\frac{k_1-b_1+2}{2},\frac{k_1+b_2}{2}\right) \itet \left(\frac{-k_2-b_1-2}{2},\frac{-k_2+b_2}{2}\right) z^{\frac{k_1 -k_2}{2}}.
\end{gather}
Putting everything together, the left-hand side of equation~\eqref{eqn:compact_for_112} is $\varepsilon(\underline{b})$ times the coefficient of $z^{b_2}$ in the generating function $$\varphi_r (z q^{\frac12},q) \varphi_r (z q^{-\frac12},q) - \varphi_{r+1} (z,q) \varphi_{r-1} (z,q).$$
This coefficient occurs when we set $k_1 = k_2$ in the double sums~\eqref{eqn:double_sum1} and~~\eqref{eqn:double_sum2}. 
\end{proof}

We conclude the proof of Theorem~\ref{thm:gang112} by using the following result, showing that the whole generating function in Proposition~\ref{prop:generating_function_noncompact} is somewhat trivial. In particular this implies our main result.  The proof of this result will occupy most of Section~\ref{sec:proof}.

\begin{thm}\label{thm:trigonometry}
For every $r \in \Z$
\begin{equation}\label{eqn:generating_function_noncompact_equality}
\varphi_r (z q^{\frac12},q) \varphi_r (z q^{-\frac12},q) - \varphi_{r+1} (z,q) \varphi_{r-1} (z,q) \equiv \left(z q^{-\frac12}\right)^{r}.
\end{equation}
\end{thm}

Using Theorem~\ref{thm:trigonometry}, we can conclude that \eqref{eqn:generating_function_noncompact} is $\varepsilon(\underline{b})$ times the coefficient of $z^{b_2}$ in  $\left(z q^{-\frac12}\right)^{r}$, which is $\varepsilon(\underline{b})  \left(q^{-\frac12}\right)^{r} = (-1)^{b_2} \left(q^{-\frac12}\right)^{b_2}$ if $r=\frac{b_1+b_2}{2}=b_2$, \textit{i.e.}~$b_1=b_2$, and $0$ if $b_1\ne b_2$. This completes the proofs of Theorem~\ref{thm:gang112} and Theorem~\ref{thm:new_thm_112}.\qed
\bigskip

%&&&&&&&&&&&&&&&&&&&&&&&&&&&&&&&&&&&&&&&&&&&&&&&&&&&&&&&&&&&&&&&&&&&&&&&&&&&&&
%&&&&&&&&&&&&&&&&&&&&&&&&&&&&&&&&&&&&&&&&&&&&&&&&&&&&&&&&&&&&&&&&&&&&&&&&&&&&&

\subsection{Relative cusp index and boundary curves}\label{ssec:}~\\

The aim of this section is to describe how the relative cusp index changes as peripheral curves are varied. This will be used to establish the inductive step in the proof of Theorem~\ref{thm:main}.\\~\\ 
Consider the two matrices in $\mathrm{SL}_2(\Z)$ given by $$R = \left(\begin{array}{cc}1 & 0 \\
1 & 1\end{array}\right),\,\,\,L = \left(\begin{array}{cc}1 & 1 \\
0 & 1\end{array}\right).$$ 
A word $w$ in the alphabet $\{R,L\}$ gives a matrix $A_w = \left(\begin{array}{cc}a & b \\
c & d\end{array}\right) \in \mathrm{SL}_2(\Z)$, as in Section~\ref{sec:LSTs}.

We have a right action of $\mathrm{SL}_2(\Z)$ on the peripheral curves: $\alpha\mapsto \alpha\cdot w=w^{-1}\alpha$ where we represent the homology class $x\lambda+y\mu$ by its coordinate vector $\begin{pmatrix} x\\ y\end{pmatrix}$ with respect to the basis $\lambda,\mu$. This gives a corresponding action on the relative cusp index, where the effect of multiplying by $R$ or $L$ is given by attaching a single tetrahedron via a $R$ or $L$ move. Note that multiplication by the -Id matrix has no effect on the index, given the symmetry~\eqref{eqn:cusp_symmetry1}.

We can now write down the inductive step in the proof of Gang's relative formula for the standard cusp; again, we need to distinguish the right\slash left cases.

\begin{prop}\label{prop:inductive_step_cusp} 
Performing an $R$ move, so that $w \mapsto w\cdot R$, changes the curves and  the relative cusp index as follows: $$\alpha = \begin{pmatrix} x\\ y\end{pmatrix} \mapsto R^{-1} \alpha = \begin{pmatrix}x\\ -x+y\end{pmatrix} = \alpha',$$ 
\begin{equation}\label{enq:index_cusp_R_new}
\begin{gathered}
I_{\cusp}^{rel} (\alpha';b_1,b_2,b_3) = I_{\cusp}^{rel}  (R^{-1}\alpha;b_1,b_2,b_3) \\=  \sum_{\substack{k \in \Z\\k\equiv b_3 \mod{2}}} \left(-q^{\frac12}\right)^k \jtet\left(b_1,b_2,\frac{b_3+k}{2}\right) I_{\cusp}^{rel}(\alpha;k,b_2,b_1).
\end{gathered}
\end{equation}

Performing an $L$ move changes the curves and  the relative cusp index as follows: $w \mapsto w\cdot L$, $$\alpha = \begin{pmatrix} x\\ y\end{pmatrix}\mapsto L^{-1} \alpha = \begin{pmatrix} x-y\\y \end{pmatrix} = \alpha',$$
\begin{equation}\label{enq:index_cusp_L_new}
\begin{gathered}
I_{\cusp}^{rel} (\alpha';b_1,b_2,b_3) = I_{\cusp}^{rel}  (L^{-1}\alpha;b_1,b_2,b_3) \\=  \sum_{\substack{k \in \Z\\k\equiv b_3 \mod{2}}} \left(-q^{\frac12}\right)^k \jtet\left(b_1,b_2,\frac{b_3+k}{2}\right) I_{\cusp}^{rel}(\alpha;b_1,k,b_2).
\end{gathered}
\end{equation}
\end{prop}
\begin{proof}

\begin{figure}[!htbp]
\centering
\includegraphics[width=0.7\linewidth]{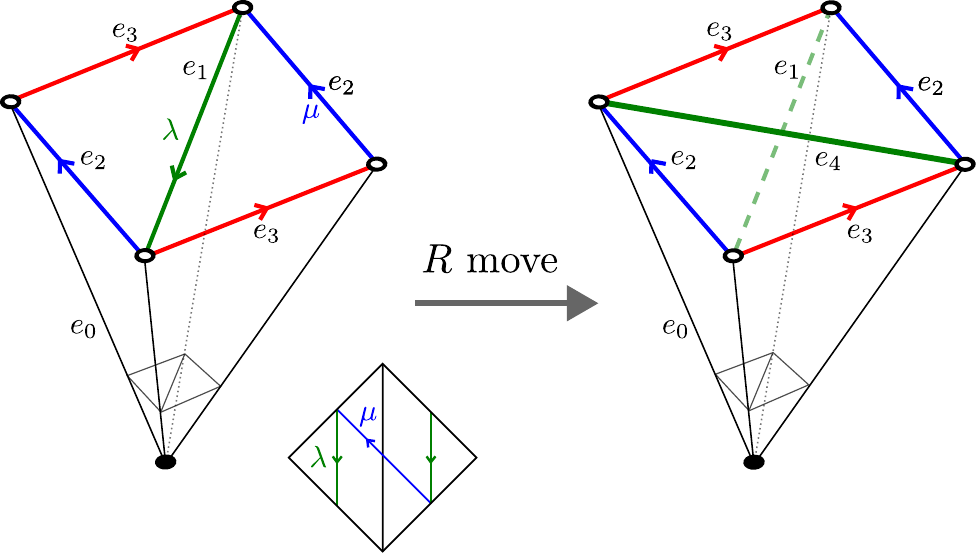}
\caption{Layering a tetrahedron on top of the standard cusp $\cusp$ to form a new triangulation $\cusp'$.}
\label{fig:cusp_standard_layer}
\end{figure}

We will prove equation~\eqref{enq:index_cusp_R_new}; the proof of equation~\eqref{enq:index_cusp_L_new} is similar. 
The intuitive idea of the proof is to first ``expand'' the standard cusp $\cusp$, by adding a new tetrahedron via an $R$ move to get a triangulation $\cusp'$, as shown in Figure~\ref{fig:cusp_standard_layer}. Computing the index for $\cusp'$ gives the right-hand side of equation~\eqref{enq:index_cusp_R_new}. We then apply the pentagon identity to $\cusp'$ to get the left-hand side of the equation.\\

We first study the right-hand side of equation~\eqref{enq:index_cusp_R_new}. From equation~\eqref{eqn:contribution_of_S} with $t_1 = t_2 = \epsilon \equiv b_3 \mod{2}$, we have 
\begin{equation}
\label{eqn:index_cusp_v0}
\begin{gathered}
\ind^{rel}_{\cusp} (x,y;b_1,b_2,b_3) \\= \left(-q^{\frac12}\right)^{\epsilon}\jtet\left( \frac{ b_1 + \epsilon-x}{2}, \frac{b_2 + \epsilon  - y}{2}, \frac{b_3 + \epsilon}{2}\right) \jtet\left( \frac{ b_1 + \epsilon+x}{2}, \frac{b_2 + \epsilon +y}{2}, \frac{b_3 + \epsilon}{2}\right).
\end{gathered}
\end{equation}

For integers $k \equiv b_3 \mod{2}$, we can write $k = 2k' + \epsilon$, where $k' \in \Z$ and $\epsilon \equiv b_3 \mod{2}$. Using this substitution and equation~\eqref{eqn:index_cusp_v0}, the right-hand side of equation~\eqref{enq:index_cusp_R_new}
can be written as
\begin{equation*}
\begin{gathered}
\sum_{k' \in \Z} \left(-q^{\frac12}\right)^{\epsilon} q^{k'} \jtet\left(b_1,b_2,\frac{b_3+\epsilon}{2} + k'\right) \ind_{\cusp}^{rel}(x,y;2k'+\epsilon,b_2,b_1) \\=
\sum_{k' \in \Z} q^{\epsilon} q^{k'} \jtet\left(b_1,b_2,\frac{b_3+\epsilon}{2} + k'\right) \jtet\left(k' + \epsilon + \frac{x}2, \frac{b_2 + \epsilon + y}2, \frac{b_1 + \epsilon}{2}\right) \\\jtet\left(k' + \epsilon - \frac{x}2, \frac{b_2 + \epsilon - y}2, \frac{b_1 + \epsilon}{2}\right)\\ = \sum_{k' \in \Z} q^{\epsilon} q^{k'}  \jtet\left(k' + \epsilon + \frac{x}2, \frac{b_2 + \epsilon + y}2, \frac{b_1 + \epsilon}{2}\right) \\\jtet\left(k' + \epsilon - \frac{x}2, \frac{b_2 + \epsilon - y}2, \frac{b_1 + \epsilon}{2}\right) \jtet\left(\frac{b_3+\epsilon}{2} + k', b_2,b_1\right).
\end{gathered}
\end{equation*}

By applying the first part of equation~\eqref{eqn:simmetries_Itet} to each term in the right-hand side, this becomes
\begin{equation*}
\begin{gathered}
\left(-q^{\frac12}\right)^{\epsilon}\sum_{k' \in \Z} q^{k'}  q^{-b_2}\itet\left(\frac{b_2 -b_1 + y}{2},k' + \frac{x - y -b_2 + \epsilon}{2}\right) \\\itet\left(\frac{b_2 -b_1 - y}{2},k' + \frac{-x + y -b_2 + \epsilon}{2}\right) \itet\left(b_2 -b_1,k' + \frac{b_3 -2b_2+ \epsilon}{2}\right).\\ 
\end{gathered}
\end{equation*}
Let us tidy up the variables by putting 
\begin{gather*}
e = k',\; m_1 = \frac{b_2 -b_1+y}{2},\; m_2 = \frac{b_2 - b_1 -y}{2},\;m_1 + m_2 = b_2 - b_1,\\\; x_1 = \frac{x-y-b_2 + \epsilon}{2},\; x_2 = \frac{-x + y -b_2 +\epsilon}{2},\; x_3 = \frac{b_3 -2b_2 + \epsilon}{2}.
\end{gather*} 
We can then use the pentagon relation in equation~\eqref{eqn:pentagon_identity} to obtain 
\begin{equation*}
\begin{aligned}
q^{-b_2}\left(-q^{\frac12}\right)^{\epsilon} \sum_{e \in \Z} q^{e} \itet (m_1, e +x_1) \itet (m_2, e +x_2) \itet (m_1+m_2, e+x_3)  \\=
q^{-b_2 } \left(-q^{\frac12}\right)^{\epsilon} q^{-x_3}\itet (m_1 - x_2 +x_3, x_1 -x_3) \itet (m_2 - x_1 + x_3, x_2 -x_3), \end{aligned}
\end{equation*}
Switching back the variables shows the right-hand side of equation~\eqref{enq:index_cusp_R_new} can be written as
\begin{equation}\label{eqn:proof_eqn}
\left(-q^{\frac12}\right)^{-b_3}  \itet \left(\frac{b_3 - b_1 + x}2 , \frac{b_2 -b_3 +x - y}2\right) \itet \left(\frac{b_3 -b_1 -x_1}2, \frac{b_2 - b_3 -x +y}2\right).
\end{equation}

Meanwhile, the left-hand side of 
equation~\eqref{enq:index_cusp_R_new} is 
\begin{equation}
\label{eqn:index_cusp_R_2}
\begin{gathered}
\ind^{rel}_{\cusp} (x,-x+y;b_1,b_2,b_3) \\= \left(-q^{\frac12}\right)^{\epsilon}\jtet\left( \frac{ b_1 + \epsilon-x}{2}, \frac{b_2 + \epsilon +x - y}{2}, \frac{b_3 + \epsilon}{2}\right) \jtet\left( \frac{ b_1 + \epsilon+x}{2}, \frac{b_2 + \epsilon -x + y}{2}, \frac{b_3 + \epsilon}{2}\right).
\end{gathered}
\end{equation}
Using the second part of equation~\eqref{eqn:simmetries_Itet}, this coincides with~\eqref{eqn:proof_eqn}. With this we are done, after noting that the same proof \textit{mutatis mutandis} works for the $L$-move. 
\end{proof}

Since the effect of an $R$- or $L$-move on the relative Gang-Yonekura formula for the cusp is obtained by applying the action in Proposition~\ref{prop:inductive_step_cusp} to each term, we obtain the following:
\begin{cor}
The Gang-Yonekura formula for the relative index of the standard cusp $GY(\alpha;\underline{b})$ transforms in the same way under $R$ and $L$ moves as the relative cusp index.
\end{cor}~\\

%&&&&&&&&&&&&&&&&&&&&&&&&&&&&&&&&&&&&&&&&&&&&&&&&&&&&&&&&&&&&&&&&&&&&&&&&&&&&&
%&&&&&&&&&&&&&&&&&&&&&&&&&&&&&&&&&&&&&&&&&&&&&&&&&&&&&&&&&&&&&&&&&&&&&&&&&&&&&

\subsection{Layered solid torus index and boundary curves}\label{ssec:LST_and_boundary}~\\

This section outlines how the relative index of a layered solid torus changes as the triangulation varies by right and left layering moves, as described in Section~\ref{sec:LSTs}.

Recall that each triangle in the Farey tessellation corresponds to a triangulation of a solid torus with three edges, whose slopes are given by the three ideal vertices of the triangle. We order the edges as described in Figure~\ref{fig:RL_edge_order}.
Note that the slopes $s_1, s_2,s_3$ are in anti-clockwise order on the circle at infinity. 

We start with a layered solid torus LST$(w)$ determined by a word $w$ in $R,L$, as in Case (1) of Section~\ref{sec:LSTs}. We obtain a new solid torus corresponding to the word $wR$ or $wL$ by attaching a tetrahedron
by a $R$ or $L$ move to the triangulation of the torus boundary as shown in Figure~\ref{fig:RL_layering}. 
This corresponds to moving to an adjacent triangle in the Farey tessellation, and the new slopes are labelled as shown in Figure~\ref{fig:RL_layering}.

\begin{figure}[!htbp]
\centering
\includegraphics[width=0.3\linewidth]{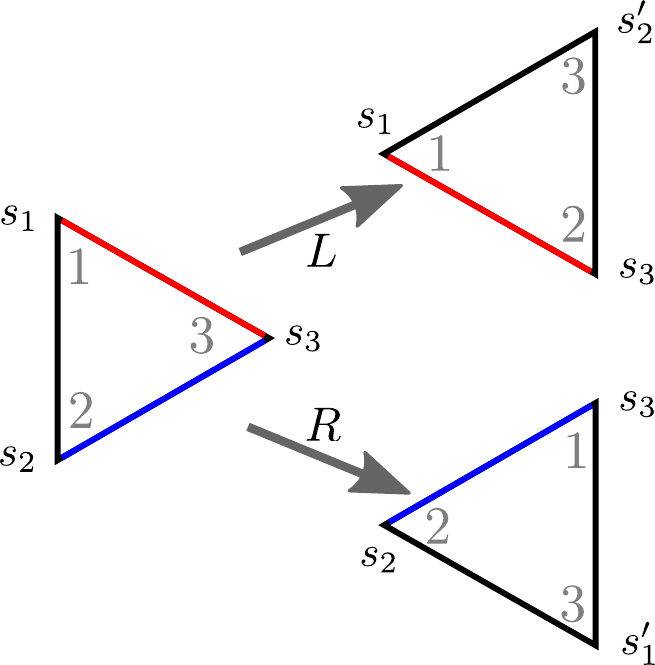}
\caption{Ordering of edge slopes for the triangulations of the torus.}
\label{fig:RL_edge_order}
\end{figure}

\begin{figure}[!htbp]
\centering
\includegraphics[width=0.65\linewidth]{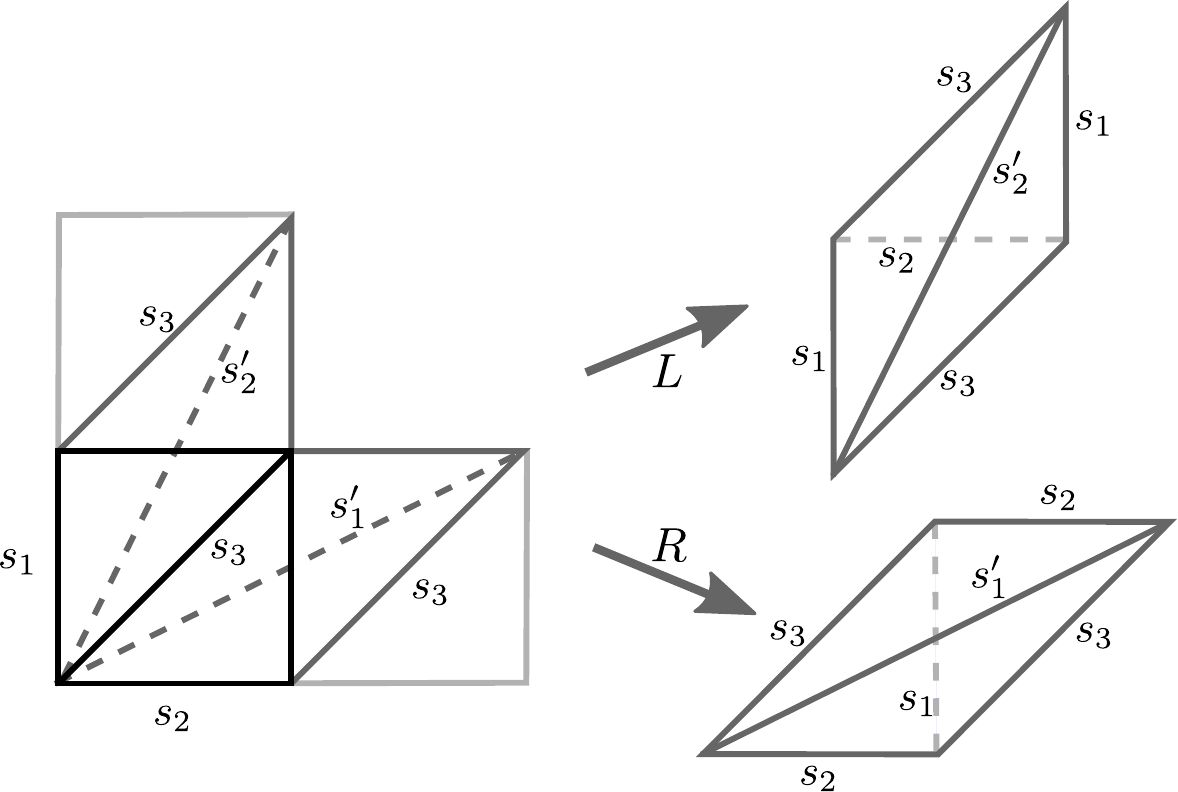}
\caption{Attaching a tetrahedron to a two-triangle torus via a right or left move.}
\label{fig:RL_layering}
\end{figure}

The following proposition shows how the relative index for a layered solid torus changes when a $R$ or $L$ move is performed.
\begin{prop}\label{prop:LST_change}
Consider the relative index of a layered solid torus {\rm LST}$(w)$ with boundary half-edge coefficients $\underline{b} = (b_1,b_2,b_3) \in \Z^3$. 
\begin{enumerate}
\item Performing an $R$ move changes the relative index as follows 
\begin{equation}\label{enq:index_cusp_R}
\ind_{\text{\rm LST}(wR)}^{rel} (\underline{b}) =  \sum_{\substack{k \in \Z\\k\equiv b_3 \mod{2}}} \left(-q^{\frac12}\right)^k \jtet\left(b_1,b_2,\frac{b_3+k}{2}\right) \ind_{\text{\rm LST}(w)}^{rel}(k,b_2,b_1).
\end{equation}
\item Similarly, performing an $L$ move has the following effect
\begin{equation}\label{enq:index_cusp_L}
\ind_{\text{\rm LST}(wL)}^{rel} (\underline{b}) =  \sum_{\substack{k \in \Z\\k\equiv b_3 \mod{2}}} \left(-q^{\frac12}\right)^k \jtet\left(b_1,b_2,\frac{b_3+k}{2}\right) \ind_{\text{\rm LST}(w)}^{rel}(b_1,k,b_2).
\end{equation}
\end{enumerate}
\end{prop}
\begin{proof}
We compute the relative index by taking all admissible integer linear combinations of half-edge and half-tetrahedral solutions
such that the corresponding quad coordinates are integers, modulo integer linear combinations of tetrahedral solutions.\\
Each quad coordinate is the sum of the edge coefficients on the two edges facing the quad plus the tetrahedral coordinate.
From Figure~\ref{fig:RL_layering} we see that in each ideal tetrahedron there exist two pairs of opposite edges that are identified to give a punctured torus. This implies that in each tetrahedron the tetrahedral coefficient must be an integer, and the top and bottom edge coefficients must be congruent modulo 2. In particular, to compute the relative index we can assume all tetrahedral coefficients are zero. 

First note that the boundary half-edge coefficients are related as follows. We start with the triple
$$(b_1,b_2,b_3) =(b(s_1),b(s_2),b(s_3)), $$
and after a right move we get
$$(b_1',b_2',b_3') = (b(s_3),b(s_2),b(s_1')).$$ 
After a left move we get instead 
$$(b_1',b_2',b_3') = (b(s_1),b(s_3),b(s_2')).$$ 
From Figure~\ref{fig:RL_layering}, the new relative index after an $R$ move can be seen to be given by 
\begin{equation}
\begin{gathered}
\ind^{rel}_{\text{\rm LST}(wR)}(b_1',b_2',b_3') = \ind^{rel}_{\text{\rm LST}(wR)}(b(s_3),b(s_2),b(s_1'))  \\=\sum_{\substack{b(s_1) \in \Z\\b(s_1)\equiv b(s_1')\mod{2}}} \left(-q^{\frac12}\right)^{b(s_1)} \jtet \left( \frac{b(s_1) + b(s_1')}2, b(s_2),b(s_3)\right) \ind^{rel}_{\text{\rm LST}(w)}(b(s_1),b(s_2),b(s_3))  \\=
\sum_{\substack{k \in \Z\\k\equiv b_3'\mod{2}}} \left(-q^{\frac12}\right)^k \jtet \left( \frac{k+ b_3'}2, b_2',b_1'\right) \ind^{rel}_{\text{\rm LST}(w)}(k,b_2',b_1'),
\end{gathered}
\end{equation}
where we sum over the new internal edge with weight $k = b(s_1)$. The invariance of $\jtet$ under permutations of its arguments gives the stated result.\\
The result for the $L$ move follows similarly.
\end{proof}
We can now state the main theorem relating the outcome of the relative Gang-Yonekura formula for the standard cusp to the appropriate layered solid torus.
\begin{prop}\label{prop:inductive_step}
Let $w$ be a word in $R,L$,  and $\alpha(w)$ be the corresponding filling slope. 
Then for all boundary half-edge coefficients $\underline{b} = (b_1,b_2,b_3)$, 
\begin{equation*}
GY(\alpha(w);\underline{b}) \equiv \ind_{\text{\rm LST}(w)}^{rel} (\underline{b}).
\end{equation*}
\end{prop}
\begin{proof}
The result of Propositions~\ref{prop:inductive_step_cusp} and~\ref{prop:LST_change} show that the effect of an $R$-move or $L$-move on the relative cusp index is the same as effect on the relative index for layered solid tori.

Furthermore the result holds for  $w = I$  by Theorem~\ref{thm:new_thm_112}. Thus the result holds for all $w$ by induction on the word's length. 
\end{proof}

If we let LST$(\alpha)$ denote the layered solid torus with meridian curve $\alpha$, as explained in Section~\ref{sec:LSTs}, then our main result for the   version of the Gang-Yonekura formula can be stated as follows. 
\begin{thm}\label{thm:main_relative}
For all slopes $\alpha \in \hh_1(\partial \cusp; \Z)$, and all $\underline{b} \in \Z^3$ we have
\begin{equation}
GY(\alpha;\underline{b}) \equiv \ind_{\text{\rm LST}(\alpha)}^{rel} (\underline{b}).
\end{equation}
\end{thm}
\begin{proof}
Proposition~\ref{prop:inductive_step} covers case (1) of Section~\ref{sec:LSTs}, so proves the result for all curves $\alpha$ with slopes in the interval $(-\infty,0)$. 
By using the symmetry $U$, this result also holds for all $\alpha$ with slopes in the intervals $(0,1)$ and $(1,\infty)$. The remaining cases, with slopes $0$, $1$ and $\infty$, correspond to fillings using LST$(0,1,1)$, and the theorem follows from the results in Appendix~\ref{sec:gang011} and Appendix~\ref{sec:relative_degenerate}. 
\end{proof}

With all necessary tools in place, we finally proceed to the proof of our main theorem:
\begin{proof}[Proof of the Main Theorem~\ref{thm:main}]
Theorem~\ref{thm:main_relative} establishes that the local form of the Gang-Yonekura formula holds for all slopes. As explained at the start of Section~\ref{sec:attaching_solid_tori}, combining this with the Gluing Theorem~\ref{thm:gluing}, which allows us to assemble local contributions into a `global' index computation, completes the proof of Theorem~\ref{thm:main}. 
\end{proof}

%&&&&&&&&&&&&&&&&&&&&&&&&&&&&&&&&&&&&&&&&&&&&&&&&&&&&&&&&&&&&&&&&&&&&&&&&&&&&&
%&&&&&&&&&&&&&&&&&&&&&&&&&&&&&&&&&&&&&&&&&&&&&&&&&&&&&&&&&&&&&&&&&&&&&&&&&&&&&

\section{Basic hypergeometric functions}\label{sec:hypergeometric}

The proof of our main Theorem~\ref{thm:main} relies crucially on the theory of basic hypergeometric functions. This section is devoted to recall, in a self-contained manner, all the definitions and results needed in what follows. The best reference for these results is Gasper and Rahman's book~\cite{GasperRahman}.\\

The basic hypergeometric series (also referred to as $q$-hypergeometric series) are a $q$-analogue of hypergeometric series. For $r,s \ge 0$, and $a_i,b_j \in \C$ they are defined as
\begin{equation}\label{eqn:definition_hypergeometric}
{_r\phi_s} \left[\begin{matrix} 
a_1 & a_2 & \ldots & a_{r} \\ 
b_1 & b_2 & \ldots & b_{s} \end{matrix} 
; q,z \right] = \sum_{n=0}^\infty  
\frac {(a_1, a_2, \ldots, a_{r};q)_n} {(b_1, b_2, \ldots, b_s,q;q)_n} \left( (-1)^n q^{\binom{n}{2} }  \right)^{1+s-r}z^n,
\end{equation}
where $(c_1, c_2, \ldots, c_{r};q)_n$ is a shorthand for $\displaystyle \prod_{i = 1}^{r} (c_i;q)_n$. \\We are moreover assuming that the coefficients $b_j$ are such that the denominator is never $0$. As we will be working under the assumption that $0<|q|<1$, the series in~\eqref{eqn:definition_hypergeometric} converge for all values of $z$ whenever $r \le s$, and for $|z|<1$ if $r = s+1$.\\

Let us describe some specific instances of $q$-analogues taking the form of $q$-hypergeometric functions; the first is a $1$-parameter family of $q$-analogues of the exponential function, first introduced in~\cite{Atakishiyev}. \\
For $\alpha \ge 0$ define
$$E_q^{(\alpha)} (z) = \sum_{n \ge 0} \frac{q^{\frac{\alpha n^2}{2}}}{(q)_n} z^n.$$
Two well-studied~\cite{GasperRahman} special cases are $E^{(0)}_q (z) = e_q(z)$ and $E^{(1)}_q (z) = E_q\left(q^{\frac12}z\right)$, which can be written as
$$e_q(z) = \sum_{n \ge 0} \frac{z^n}{(q)_n} ,\,\,\,\,\, E_q(z) = \sum_{n \ge 0} \frac{q^{\binom{n}{2}}}{(q)_n}z^n .$$

Note that by \cite[Sec.~1.3]{GasperRahman}, we have $$e_q(z) = {_1\phi_0}\left[  \begin{matrix} 0\\ -
\end{matrix}; q, z \right] = \frac{1}{(z)_\infty},$$ and $$E_q(z) = {_0\phi_0}\left[  \begin{matrix} -\\ -
\end{matrix}; q, - z \right] = (-z)_\infty.$$
Moreover, recall that we encountered the product of these exponentials before in equation~\eqref{enq:generating_function_tetra}, as the generating function for the tetrahedral index; in this notation we have
$$e_q \left(q^{-\frac{m}{2}} z\right) E_q \left(-q^{1-\frac{m}{2}} z^{-1}\right) = \sum_{e \in \Z} \itet (m,e) z^e.$$

As in~\cite{Nevanlinna}, denote by $$\qex (z) = E^{\left(\frac{1}{2}\right)}_q (z) = {_1\phi_1}\left[  \begin{matrix} 0\\ -q^{\frac12}
\end{matrix}; q^{\frac12}, -q^{\frac14} z \right] = \sum_{n \ge 0} \frac{q^{\frac{n^2}{4}}}{(q)_n} z^n, \,\, z \in \C.$$

This function satisfies the relation 
$$\qex (z) - \qex (qz) = z q^{\frac14} \qex \left(q^{\frac12} z\right).$$

The associated $q$-analogues of the sine and cosine functions are defined to obey a $q$-version of Euler's formula:
 $$\qex (\textit{i} z) = \qcos (z) + \textit{i} q^{\frac14} \qsin(z).$$
Explicitly, we can write
$$\qsin (z) = \sum_{n\ge 0} \frac{(-1)^n q^{n(n+1)}}{(q)_{2n+1}} z^{2n+1}, \,\,\,\, \qcos (z) = \sum_{n\ge 0} \frac{(-1)^n q^{n^2}}{(q)_{2n}} z^{2n}.$$
These formulae can be re-written more concisely as 
$$\qsin(z) = \frac{z}{1-q}
{_1\phi_1}\left[  \begin{matrix} 0\\ q^3
\end{matrix}; q^2, q^2 z^2 \right]
\,\,\,\,\, \qcos(z) = 
{_1\phi_1}\left[  \begin{matrix} 0\\ q
\end{matrix}; q^2, q z^2 \right].$$

The following relations are an immediate consequence of the definitions above:
$$\qsin(z) - \qsin(qz) = z \qcos\left(q^{\frac12} z\right)\,\,\,\,\,\,\,\,\,\qcos(z) - \qcos(qz) = -z q^{\frac12} \qsin\left(q^{\frac12} z\right)$$
$$\mathcal{E}_{q^{-1}} \left(q^{-\frac14}z\right) = \qex \left(-q^{\frac14} z\right)$$

Crucially for our purposes, certain products of these exponentials can be re-written in terms of $_3\phi_3$ functions:
\begin{prop}[\cite{Nevanlinna}]
For $u,v \in \C$:
\begin{gather*}
\begin{gathered}
\qex (u) \qex(-v) = \\{_3\phi_3} \left[ \begin{matrix} 0 & u^{-1} v q^{\frac12} & u v^{-1}q^{\frac12}\\ -q & q^{\frac12} & -q^{\frac12} \end{matrix}; q, uv q^{\frac12}\right] +   q^{\frac14} \frac{u-v}{1-q}  {_3\phi_3} \left[ \begin{matrix} 0 & u^{-1} v q & u v^{-1}q \\ -q & q^{\frac32} & -q^{\frac32} \end{matrix}; q, uv q\right]
\end{gathered}
\end{gather*}
\end{prop}

\begin{cor}[\cite{Nevanlinna}]\label{cor:3phi3_and_trig}
For $u,v \in \C$:
\begin{equation}\label{eqn:product_coscos}
\qcos (u) \qcos(v) + q^{\frac12} \qsin(u)\qsin(v) = {_3\phi_3} \left[ \begin{matrix} 0 & u^{-1} v q^{\frac12} & u v^{-1}q^{\frac12}\\ -q & q^{\frac12} & -q^{\frac12} \end{matrix}; q, -uv q^{\frac12}\right] ,  
\end{equation}
and 
\begin{equation}\label{eqn:product_cossin}
\qsin(u) \qcos(v) - \qcos(u) \qsin(v) = \frac{u-v}{1-q} {_3\phi_3} \left[ \begin{matrix} 0 & u^{-1} v q & u v^{-1}q \\ -q & q^{\frac32} & -q^{\frac32} \end{matrix}; q, -uv q\right].
\end{equation}
\end{cor}
\begin{lem}[$q$-Pythagoras, \cite{Nevanlinna}]\label{lem:q-pythagoras}
If the two variables $u$ and $v$ in equation~\eqref{eqn:product_coscos} are related by $u = q^{\frac12} v$, then
\begin{equation}\label{eqn:q-pythagoras}
\qcos\left(q^{\frac12} v\right) \qcos(v) + q^{\frac12} \qsin\left(q^{\frac12} v\right) \qsin(v) \equiv 1.    
\end{equation}
\end{lem}
\begin{proof}
This follows at once from the previous corollary, noting that the corresponding $q$-hypergeometric function has a single non-zero term, equal to $1$.
\end{proof}

%&&&&&&&&&&&&&&&&&&&&&&&&&&&&&&&&&&&&&&&&&&&&&&&&&&&&&&&&&&&&&&&&&&&&&&&&&&&&&
%&&&&&&&&&&&&&&&&&&&&&&&&&&&&&&&&&&&&&&&&&&&&&&&&&&&&&&&&&&&&&&&&&&&&&&&&&&&&&

\section{Proof of the main result}\label{sec:proof}

We now turn to provide the last step needed in the proofs of Theorem~\ref{thm:trigonometry} and Theorem~\ref{thm:main}. This will be achieved by showcasing some unexpected connections between the generating functions defined in Section~\ref{ssec:initial_cusp_formula} and basic hypergeometric functions.

Our approach begins with a generating function (equation~\eqref{eqn:generating_all} below), which defines the \textit{meromorphic $3$D index} introduced in~\cite{garoufalidis2019meromorphic}. Using standard complex analysis techniques, we will extract from it the generating functions for the ``shifted diagonal'' tetrahedral index defined in equation~\eqref{eqn:seriesindexshifted}. \\

Let us begin by introducing the auxiliary meromorphic function 
$$ G_q(z) = \frac{(-qz^{-1};q)_\infty}{(z;q)_{\infty}},$$
defined for $z \in\C^*, 0<|q|<1$. The only poles of $G_q(z)$ are at $z = q^{-n}$ for all integers $n \ge 0$, and each such pole is simple. It is proved in~\cite[Lem.~2.2]{garoufalidis2019meromorphic} that $G_q(z)$ has a convergent Laurent series expansion in the annulus $0<|z|<1$. 

In order to get a cleaner statement in what follows, we introduce a slight variation on the usual tetrahedral index: $$I^\Delta (m,e)(q) = (-q)^e \itet (m,e)(q^2).$$

According to~\cite{garoufalidis2019meromorphic}, the generating function for the sum of $I^\Delta(m,e)$ over all lattice points $(m,e)\in \Z^2$ is the two-variable meromorphic function defined as
\begin{equation}\label{eqn:generating_all}
\begin{gathered}
\psi^0(z,w) = c(q) G_q\left(-qz\right) G_q\left(w^{-1}\right) G_q\left(wz^{-1}\right)  \\=c(q)\frac{\left(z^{-1};q\right)_\infty}{\left(-qz;q\right)_{\infty}} \frac{\left(-qw;q\right)_\infty}{\left(w^{-1};q\right)_{\infty}} \frac{\left(-qzw^{-1};q\right)_\infty}{\left(wz^{-1};q\right)_{\infty}},
\end{gathered}
\end{equation}
where $c(q) = \frac{(q;q)^2_\infty}{(q^2;q^2)_{\infty}} = \sum_{k \in \Z} (-q)^{k^2}$ is a normalising factor.  \\~\\ 
More precisely, as proved in~\cite[Thm.~2.5]{garoufalidis2019meromorphic}, for $1 < |w|<|z|<|q^{-1}|$ we have the absolutely convergent Laurent series expansion
\begin{equation}\label{eqn:generating_all_implicit}
\psi^0(z,w) = \sum_{e,m\in \Z} I^{\Delta}(m,e)(q) z^e w^m = \sum_{e,m\in \Z} (-q)^e \itet(m,e)(q^2) z^e w^m  .
\end{equation}

As stated above, our current goal is to extract from \eqref{eqn:generating_all_implicit} the generating function for the ``diagonal tetrahedral index'' $\itet(k-r,k)$, where $r \in \Z$ is fixed, as in equation~\eqref{eqn:seriesindexshifted}. This in turn will give us closed expressions we can use to prove equation~\eqref{eqn:generating_function_noncompact_equality}. \\

We rely on a standard technique from Stanley~\cite[Sec.~6.3]{Stanley}; note that the function 
we seek is essentially the coefficient of 
$s^r$ in the series (regarded as a function of $s$):
\begin{equation}\label{enq:find_diagonal}
\psi^0\left(s,\frac{x}{s}\right) = \sum_{e,m\in \Z} (-q^e) \itet (m,e)(q^2) x^m s^{e-m},
\end{equation}
where the expression is valid for $s\in \C$ such that $1<|x|^{\frac{1}{2}} < |s| < \min\{|x|,|q^{-1}|\}$.

Putting $e-m = r$ in equation~\eqref{enq:find_diagonal} and using the residue theorem gives 
\begin{equation}\label{rintegral}
\left[ \psi^0\left(s,\frac{x}{s}\right)\right]_{s^r} = \sum_{e \in \Z} (-q)^e \itet (e-r,e)(q^2) x^{e-r} = \frac{1}{2\pi \emph{i}}\int_{|s|=\rho} \psi^0\left( s,\frac{x}{s}\right) \frac{\mathrm{d}s}{s^{r+1}} \ ,
\end{equation}
where $1 < |x|^{\frac{1}{2}}< \rho < \min\{|x|,|q|^{-1}\}$.

In order to compute the integral on the right-hand side of equation~\eqref{rintegral}, we need to compute the residues of 
$\frac{1}{s^r} \psi^0\left(s,\frac{x}{s}\right)$ in the punctured disc $D_\rho \setminus\{0\}\subset \C^*$ centred at the origin whose boundary is the circle $|s| = \rho$.
Each of the three factors in equation~\eqref{eqn:generating_all} contributes with an infinite set of simple poles, which we divide in three categories below; note that the positions of the poles corresponding to the first two factors lie outside the disc of radius $\rho$, and converge to $\infty$. 
The poles of type $(\emph{iii})$, associated to the third factor, instead come in pairs and converge to the origin. In all cases, the poles are of order $1$, and the numbering below is over $n \ge 0$. 
\begin{itemize}
\item[(\emph{i})] $s = -q^{-n-1}$; 
$$Res\left(\frac{1}{s^r} \psi^0\left(s,\frac{x}{s}\right),-q^{-n-1}\right) = (-1)^{r+1} q^{(n+1)r} \frac{(-q^{-1-2n}x^{-1};q)_n  (q^{n+2} x;q)_n}{(-q;q)_n (q^{-n};q)_n}$$
\item[(\emph{ii})] $s = xq^{-n}$; 
$$Res\left(\frac{1}{s^r} \psi^0\left(s,\frac{x}{s}\right),xq^{-n}\right) = -x^{-r} q^{nr} \frac{(-q^{n}x^{-1};q)_n  (-q^{1-2n} x;q)_n}{(-q;q)_n (q^{-n};q)_n}$$
\item[(\emph{iii})] $s = \pm x^{\frac{1}{2}}q^{\frac{n}{2}}$; 
$$Res\left(\frac{1}{s^r} \psi^0\left(s,\frac{x}{s}\right),\pm x^{\frac{1}{2}} q^{\frac{n}{2}}\right) = \frac{1}{2 (\pm x^{\frac{1}{2}} q^{\frac{n}{2}})^r} \frac{(\pm q^{-\frac{n}{2}} x^{-\frac{1}{2}} ; q)_n ( \mp q^{1-\frac{n}{2}} x^{\frac{1}{2}} ; q)_n}{(-q; q)_n (q^{-n}; q)_n}.$$
\end{itemize}

Note that only the poles of type $(iii)$ lie in the punctured disc $D_\rho \setminus\{0\}$. 

We need the following technical lemma in order to apply Cauchy's residue theorem. 
\begin{lem}\label{lem:residue_to_0}
Let $|q|<1$ and choose $\rho_0>0$ such that the circles $|s|=\rho_0 |q|^n$ avoid the poles of $\psi^0\left(s,\frac{x}{s}\right)$ for all integers $n>0.$ 
Then $\rho_n=\rho_0 |q|^n \to 0$ as $n \to \infty$ and for each $m\in\Z$
$$ \lim_{n\rightarrow \infty} \frac{1}{2\pi \emph{i}} \int_{|s| = \rho_n}\psi^0\left(s,\frac{x}{s}\right) s^{m-1} \, \mathrm{d}s = 0.$$
\end{lem}
\begin{proof} 
This can be proved by a slight modification of the argument in \cite[Sec.~4.10]{GasperRahman}. Let us write $$P(s) = \psi^0 \left(s, \frac{x}{s}\right) = c(q) \frac{(as^2, b_1 s , b_2 s ;q)_\infty }{(c_1 s, c_2 s, d s^{-2};q)_\infty},$$
where $a = -q x^{-1}, \, b_1 = 1,\, b_2 = -qx,\, c_1 = -q,\, c_2 = x^{-1}$ and $d = x$. Then, poles correspond to $s = c_j^{-1} q^{-n}$, for $j = 1,2$ and $n \ge 0$. 

For $m \in \Z$, consider the integral $$I_m = \frac{1}{2\pi \emph{i}} \int_{K} P (s) s^{m-1} \mathrm{d}s,$$
where $K$ is the positively oriented circle $|s| = \rho$, and $1< |x|^\frac12<\rho < \min\{|x|,|q|^{-1}\}$. 

Let $|q|< 1$, and pick $\delta_0 > 0$ such that $\delta_0 \neq |d q^n|$ and $\delta_0 \neq |c_j q^n|$, for $j = 1,2$; then put $\delta = \delta_0 e^{\emph{i} \theta}$, for $\theta \in \R$, and look at $s = \delta q^N$, $N \ge 1$. \\Call $C_N =\{s \in \C\,|\, |s| = \delta_0 |q|^N\}$, for a positive integer $N$. Then, for all $N  \ge 1$, $C_N$ does not intersect the poles of $P(s)$.  
As a consequence we can write 
$$\left| P(\delta q^N) (\delta q^N)^{m-1} \right| = \left| (\delta q^N)^{m-1} \frac{(q^{2N}a \delta, q^{-N} b_1 \delta^{-1}, q^{-N} b_2 \delta^{-1} ;q)_\infty}{(q^N c_1 \delta, q^N c_2 \delta, q^{-2N} d \delta^{-2} ;q)_\infty}\right|.$$

To proceed, we consider the following relations (\textit{cf.}~\cite[App.~1]{GasperRahman}):
\begin{itemize}
\item $ (q^{2N} a \delta^2;q)_\infty = \frac{(a\delta^2;q)_\infty}{(a\delta^2;q)_{2N}}$
\item $(b_j \delta^{-1};q)_\infty = (-1)^N b_j^{-N} \delta^{N} q^{\binom{N+1}{2}}\frac{(q^{-N} b_j \delta^{-1};q)_\infty}{(q\delta b_j^{-1};q)_N}$.
\item $(q^{N}c_j \delta ; q)_\infty = \frac{(c_j \delta ;q)_\infty }{(c_j \delta ; q)_N}$.
\item $(d \delta^{-2} ; q)_\infty = d^{-2N} q^{2N} q^{\binom{2N+1}{2}} \frac{(q^{-2N} d \delta^{-2};q)_\infty}{(q\delta^2 d^{-1};q)_{2N}}$.
\end{itemize}
Putting everything together, we get:
\begin{gather*}
\begin{gathered}
\left| P(\delta q^N) (\delta q^N)^{m-1}\right| = \\ \left| \frac{(a\delta^2, b_1\delta^{-1} , b_2 \delta^{-1};q)_\infty}{(c_1 \delta, c_2 \delta, d \delta^{-2}; q)_\infty}\right| \left| \frac{(c_1 \delta , c_2 \delta , q \delta b_1^{-1}, q\delta b_2^{-1};q)_N}{(a \delta^2;q)_{2N} (q \delta^2 d^{-1};q)_{2N}}\right| \left| \left( \frac{b_1 b_2 q^{m-1}}{d^2} \right)^N \right|\\ \delta^{m-1} \left| \delta^{4N-2N} q^{\binom{2N+1}{2}-2 \binom{N+1}{2}}\right|  = O \left( \left| \frac{b_1 b_2 q^{m-1}}{d^2}\right|^N \delta^{2N} |q|^{N^2}\right),
\end{gathered}
\end{gather*}
where the last equality follows by letting $N \rightarrow \infty$.

The length of the circle $C_N$ asymptotically behaves as $len (C_N) = O(q^N)$, therefore we can deduce 
$$I_m=\int_{K} P (s) s^{m-1} \mathrm{d}s = O \left( \left| \frac{b_1 b_2 q^{m}}{d^2}\right|^N \delta^{2N} |q|^{N^2}\right)\, ,$$ which tends to $0$ as $N$ goes to $\infty$.
\end{proof}

Using Lemma~\ref{lem:residue_to_0} with $m=-r$, we can now apply Cauchy's residue theorem on the annulus $\rho_n\le|s|\le\rho$ and take the limit as $n\to\infty$ to conclude that the integral in 
equation~\eqref{rintegral}
is the
sum of the residues of type $(\emph{iii})$: 
\begin{gather}\label{eqn:diagonal_not_simplified}
\begin{gathered}
\sum_{n =0}^\infty Res\left(\frac{1}{s^r}\psi^0\left(s,\frac{x}{s}\right), x^{\frac12} q^{\frac{n}{2}}\right) + Res\left(\frac{1}{s^r}\psi^0\left(s,\frac{x}{s}\right), -x^{\frac12} q^{\frac{n}{2}}  \right)  \\ 
= \frac{1}{2x^{\frac{r}{2}}} \sum_{n = 0}^\infty \frac{( q^{-\frac{n}{2}} x^{-\frac{1}{2}} ; q)_n ( - q^{1-\frac{n}{2}} x^{\frac{1}{2}} ; q)_n  + (-1)^r (- q^{-\frac{n}{2}} x^{-\frac{1}{2}} ; q)_n ( q^{1-\frac{n}{2}} x^{\frac{1}{2}} ; q)_n}{q^{\frac{nr}{2}}(-q; q)_n (q^{-n}; q)_n} 
\end{gathered}
\end{gather}

Our next step consists in identifying this series of residues with known $q$-hypergeometric series.
First note that the expression in equation~\eqref{eqn:diagonal_not_simplified} 
can be greatly simplified. The terms in the denominator can be paired as follows: 
\begin{equation}\label{simp_denom}
(-q;q)_n (q^{-n};q)_n =(-1)^nq^{-\binom{n+1}{2}} (q^2;q^2)_n,\end{equation}
by regrouping the factors. (Here, both sides are interpreted as $1$ when $n=0$.)

Similarly, the two products appearing in the numerator of~\eqref{eqn:diagonal_not_simplified} can be simplified as well. Indeed, 
\begin{equation}\label{simp_numer1}   
(q^{-\frac{n}{2}} x^{-\frac{1}{2}};q)_n  (-q^{1-\frac{n}{2}} x^{\frac{1}{2}};q)_n = (qx)^{\frac{n}{2}} (q^{-n}x^{-1};q^2)_n ,\end{equation} and 
\begin{equation}\label{simp_numer2}
(-q^{-\frac{n}{2}} x^{-\frac{1}{2}};q)_n  (q^{1-\frac{n}{2}} x^{\frac{1}{2}};q)_n = (qx)^{-\frac{n}{2}} (q^{n}x; q^{-2})_n = (-1)^n(qx)^{\frac{n}{2}} (q^{-n}x^{-1}; q^2)_n.\end{equation}

Altogether, we can rewrite the series in equation~\eqref{eqn:diagonal_not_simplified} as:
\begin{equation}\label{eqn:diagonal_simplified}
 \sum_{n = 0}^\infty \frac{(-1)^n + (-1)^{r}}{2} q^{\frac{n(n+2-r)}{2}}  x^{\frac{n-r}{2}} \frac{(q^{-n}x^{-1}; q^2)_n}{  (q^2;q^2)_n}.\end{equation}

It is possible to further simplify this expression and identify it as a $q$-hypergeometric function.  
We need to distinguish between two cases, depending on the parity of the diagonal shift $r$. 

\medskip
If $r \equiv 0 \, (\text{mod } 2)$, put $n = 2k$ and $r = 2t$ where $k,t \in \Z$. Then the sum in~\eqref{eqn:diagonal_simplified} becomes
\begin{equation}\label{evensum}
\sum_{k \ge 0} q^{2k(k+1-t)} x^{k-t}  \frac{(q^{-2k}x^{-1}; q^2)_{2k}}{(q^2;q^2)_{2k}} \ .
\end{equation}

The initial term is $a_0=x^{-t}$, and the ratio of two consecutive terms is
$$\frac{a_{k+1}}{a_k} = q^{4+4k-2t}x \frac{(q^2;q^2)_{2k}}{(q^2;q^2)_{2k+2}} \frac{(q^{-2k-2}x^{-1}; q^2)_{2k+2}}{(q^{-2k}x^{-1}; q^2)_{2k}} = -q^{2+2k-2t} \frac{(1-q^{2k}x^{-1})(1-q^{2k+2}x)}{(1-q^{2+4k})(1-q^{4+4k})}.$$
After the substitution $u = q^2$ we get
$$\frac{a_{k+1}}{a_k} = -u^{1+k-t} \frac{(1-u^{k}x^{-1})(1-u^{k+1}x)}{(1-u^{k+\frac{1}{2}})(1+u^{k+\frac{1}{2}}) (1-u^{k+1})(1+u^{k+1})},$$
therefore the $n$-th term $a_n$ can be written as 
$$(-1)^n x^{-t} \frac{u^{\binom{n}{2}}(u^{1-t})^n}{(u;u)_n} \frac{(x^{-1};u)_n (ux;u)_n (0;u)_n}{(-u;u)_n (u^{\frac{1}{2}};u)_n (-u^{\frac{1}{2}};u)_n}.$$
Thus we can identify the series~\eqref{evensum} as the $q$-hypergeometric function 
\begin{equation*}
x^{-t}  \,{_3\phi_3} \left[\begin{matrix} x^{-1} &  x q^2   & 0\\
-q^2 & q & -q 
 \end{matrix} 
; q^2, q^{2(1-t)} \right] 
\end{equation*}
and by \eqref{rintegral} this gives $\sum_{e \in \Z} x^{e-2t} (-q)^e  \itet (e-2t,e)(q^2)$.
After the replacement of $q$ by $q^\frac12$ and $t$ by $\frac{r}{2}$, we obtain
\begin{equation}\label{eqn:diagonal_qypergeometric_EVEN}
x^{\frac{r}{2}} {_3\phi_3} \left[\begin{matrix} x^{-1} &  x q  & 0\\
-q & q^\frac12 & -q^\frac12 
 \end{matrix} 
; q, q^{1-\frac{r}{2}} \right] = \sum_{e \in \Z} (-xq^\frac12)^e \itet (e-r,e).
\end{equation}

\medskip
If instead $r \equiv 1\, (\text{mod } 2)$, we apply to equation~\eqref{eqn:diagonal_simplified} the substitutions $n = 2k+1$ and $r = 2t+1$ to get:
\begin{equation}\label{oddsum}
- \sum_{k = 0}^\infty q^{(2k+1)(k-t+1)} x^{k-t} \frac{(q^{-2k-1}x^{-1};q^2)_{2k+1}}{(q^2;q^2)_{2k+1}}.\end{equation}

The initial term in this case is $-q^{1-t} x^{-t} \frac{(1-q^{-1}x^{-1})}{(1-q^2)} = \frac{(1-qx)}{q^t x^{t+1} (1-q^2)}$, and the ratio of two consecutive terms is
$$\frac{a_{k+1}}{a_k} = q^{5+4k-2t}x \frac{(q^2;q^2)_{2k+1}}{(q^2;q^2)_{2k+3}} \frac{(q^{-2k-3}x^{-1}; q^2)_{2k+3}}{(q^{-2k-1}x^{-1}; q^2)_{2k+1}} = -q^{2+2k-2t} \frac{(1-q^{2k+1}x^{-1})(1-q^{2k+3}x)}{(1-q^{4+4k})(1-q^{6+4k})}.$$

Using again the substitution $u = q^2$, this ratio becomes
$$-u^{k-t+1} \frac{(1-u^{k+\frac{1}{2}}x^{-1})(1-u^{k+\frac{3}{2}}x)}{(1-u^{k+1})(1+u^{k+1})(1-u^{k+\frac{3}{2}})(1+u^{k+\frac{3}{2}})}.$$

Therefore, the $n$-th term $a_n$ in this case is 
$$ (-1)^n \frac{(1-u^{\frac{1}{2}} x) (u^{1-t})^n u^{\binom{n}{2}}}{u^{\frac{t}{2}}x^{t+1}(1-u)} \frac{(u^{\frac{1}{2}} x^{-1};u)_n (u^{\frac{3}{2}} x ;u)_n (0;u)_n}{(u;u)_n (-u;u)_n (u^{\frac{3}{2}};u)_n (-u^{\frac{3}{2}}; u)_n}\ ,$$
and the corresponding $q$-hypergeometric function is
\begin{equation*}
\frac{(1-xq)}{x^{\frac{r+1}{2}} q^{\frac{r-1}{2}} (1-q^2)} {_3\phi_3} \left[\begin{matrix} x^{-1} q & xq^3  & 0\\
-q^2 & q^3 & -q^3 
 \end{matrix} 
; q^2, q^{3-r} \right]  = \sum_{e \in \Z} (-q)^e \itet (e-r,e)(q^2) x^{e-r}.  
\end{equation*}

As before, after replacing $q$ by $q^\frac12$ and $t$ by $\frac{r-1}{2}$, we get:
\begin{equation}\label{eqn:diagonal_qypergeometric_ODD}
x^{\frac{r-1}{2}} q^{\frac{1-r}{4}} 
\frac{(1-x q^\frac12)}{1-q} {_3\phi_3} \left[\begin{matrix} x^{-1} q^\frac12 & x q^{\frac{3}{2}}  & 0\\
-q & q^\frac{3}{2} & -q^{\frac{3}{2}} 
 \end{matrix} 
; q, q^{\frac{3-r}{2}} \right] 
 = \sum_{e \in \Z} (-q^\frac12 x)^e \itet (e-r,e) .
\end{equation}

To summarise the results we obtained so far in this section, we collect the sums of shifted diagonal terms from equations~\eqref{eqn:diagonal_qypergeometric_EVEN} and~\eqref{eqn:diagonal_qypergeometric_ODD} in a single function.\\

After applying the substitution $z = -x q^{\frac12}$, we see that 
\begin{equation}\label{eqn:phi_expression}
\varphi_r(z,q) = 
\begin{cases}
(-z q^{-\frac12})^{\frac{r}{2}}\,{_3\phi_3} \left[\begin{matrix} -z^{-1} q^\frac12 &  -zq^\frac12  & 0\\
-q & q^\frac12 & -q^\frac12 
 \end{matrix} 
; q, q^{1-\frac{r}{2}} \right]   & r \text{ even}\\
~\\
(-z)^{\frac{r-1}{2}} q^{\frac{1-r}{2}} 
\frac{(1+z)}{1-q}\, {_3\phi_3} \left[\begin{matrix} -z^{-1} q   & -z q  & 0\\
-q & q^\frac{3}{2} & -q^{\frac{3}{2}} 
 \end{matrix} 
; q, q^{\frac{3-r}{2}} \right]   & r \text{ odd.}
\end{cases}
\end{equation}

We can now provide the proof of Theorem~\ref{thm:trigonometry} from Section~\ref{ssec:GY_for_cusp}. 

\begin{proof}[Proof of Theorem~\ref{thm:trigonometry}]
We begin by identifying the functions $\varphi_r (z, q)$ with suitable combinations of the trigonometric $q$-analogues from~\cite{Nevanlinna} introduced in Section~\ref{sec:hypergeometric}. 

Indeed, using Corollary~\ref{cor:3phi3_and_trig}, it is immediate to realise that, if $r$ is even, 
\begin{equation*}\label{eqn:phi_trigonometric_even}
\varphi_r (z,q) = \left(-z q^{-\frac12}\right)^{\frac{r}{2}} \left[\qcos \left(z^{\frac12} q^{\frac{1-r}{4}}\right) \qcos\left(- z^{-\frac12} q^{\frac{1-r}{4}}\right) + q^{\frac12} \qsin\left(z^{\frac12} q^{\frac{1-r}{4}}\right)\qsin\left(- z^{-\frac12} q^{\frac{1-r}{4}}\right) \right],
\end{equation*}
while for $r$ odd
\begin{equation*}\label{eqn:phi_trigonometric_odd}
\varphi_r (z,q) = (-1)^{\frac{r-1}{2}} q^{\frac{1-r}{4}} z^{\frac{r}{2}}\left[ \qsin\left(z^{\frac12} q^{\frac{1-r}{4}} \right) \qcos\left(-z^{-\frac12} q^{\frac{1-r}{4}}\right) - \qcos\left(z^{\frac12} q^{\frac{1-r}{4}}\right) \qsin\left(-z^{-\frac12} q^{\frac{1-r}{4}} \right)\right].
\end{equation*}

Substituting the above expressions in equation~\eqref{eqn:generating_function_noncompact_equality} for any $r\in \Z$, the resulting expression simplifies to
\begin{multline*}
(z q^{-\frac12})^r \left(\qcos\left(-\frac{q^{-\frac{r}{4}}}{z^\frac12},q\right) \qcos\left(-\frac{q^{\frac{1}{2}-\frac{r}{4}}}{z^\frac12},q\right)+q^\frac12 \qsin\left(-\frac{q^{-\frac{r}{4}}}{z^\frac12},q\right) \qsin\left(-\frac{q^{\frac{1}{2}-\frac{r}{4}}}{z^\frac12},q\right)\right)\\ \times \left(\qcos\left(z^\frac12 q^{-\frac{r}{4}},q\right) \qcos\left(z^\frac12 q^{\frac{1}{2}-\frac{r}{4}},q\right)+q^\frac12 \qsin\left(z^\frac12 q^{-\frac{r}{4}},q\right) \qsin\left(z^\frac12 q^{\frac{1}{2}-\frac{r}{4}},q\right)\right).
\end{multline*}
Then, a simple application of the $q$-Pythagoras Lemma~\ref{lem:q-pythagoras} implies that both factors in the expression above are identically $1$.
\end{proof}

\begin{rmk}
We suspect equation~\eqref{eqn:generating_function_noncompact} is equivalent to this alternative version 
\begin{equation}\label{eqn:alternative}
\varphi_r(z q,q) \varphi_r(z^{-1} q,q) -q^r z^{-r}\varphi_{r-2}(z,q)\varphi_{r+2}(z,q) \equiv 1 ,
\end{equation}
but we were unable to prove the equivalence between the two generating functions. 
\end{rmk}

As a direct consequence of the computations that led to equation~\eqref{eqn:phi_expression}, we obtain the following ``linear''  relations (as opposed to the ``quadratic'' ones~\eqref{eqn:quadratic_identity}) between certain values of the tetrahedral index. These relations may be of independent interest, and are not directly implied by the linear relation used in~\cite[Sec.~11.1]{garoufalidis20163d}.
\begin{cor}\label{cor:linear_diagonal_relations}
The following identities hold. \\
If $r \in \Z$ is even, then:
\begin{align}
\sum_{e \in \Z} \left(-q^\frac12\right)^e \itet (e-r,e) & =1 ,\\
\sum_{e \in \Z} \left(-q^\frac12\right)^{-e} \itet (e-r,e) & = q^{-\frac{r}{2}} .
\end{align}
If $r \in \Z$ is odd, then:
\begin{align}
\sum_{e \in \Z} (-q)^e \itet (e-r,e) & =1 ,\\
\sum_{e \in \Z} (-q)^{-e} \itet (e-r,e) & = -q^{-r} \\
\sum_{e \in \Z} (-1)^e \itet (e-r,e) & = 0.
\end{align}
\end{cor}
\begin{proof}
Evaluate the appropriate one of  equations~\eqref{eqn:diagonal_qypergeometric_EVEN} and \eqref{eqn:diagonal_qypergeometric_ODD} at $x = 1, q^{-1}, q^\frac12,q^{-\frac{3}{2}}$ and $q^{-\frac12}$ respectively.
\end{proof}

\begin{rmk}
The values chosen in the proof above are the smallest ones which make the corresponding $q$-hypergeometric series terminate. More generally, if $r$ is even, we get a linear relation for each value of $x = q^c$ for $c \in \Z$. By this, we mean that the infinite shifted and weighted series $$\sum_{e \in \Z}\left(-q^{\frac{1}{2} + c}\right)^e \itet (e-r,e)$$ will be equal to a quotient of two Laurent polynomials in $q$; moreover, choosing $c-1$ or $-c$ will result in the same polynomial. Similarly, for $r$ odd, if $x = q^{\frac{c}{2}}$.
\end{rmk}

%&&&&&&&&&&&&&&&&&&&&&&&&&&&&&&&&&&&&&&&&&&&&&&&&&&&&&&&&&&&&&&&&&&&&&&&&&&&&&
%&&&&&&&&&&&&&&&&&&&&&&&&&&&&&&&&&&&&&&&&&&&&&&&&&&&&&&&&&&&&&&&&&&&&&&&&&&&&&

\section{Large fillings}

One very interesting consequence of the Gang-Yonekura formula is that it 
allows us to understand the limiting behaviour of the $3$D index for manifolds obtained as $(a,b)$ Dehn fillings with $|a|+|b|$ large. 
Surprisingly, in contrast to Thurston's of hyperbolic Dehn filling theorem~\cite{thurston1982three}, we find that the limiting behaviour of the $3$D index generally depends on the \emph{direction} in which $(a,b)$ approaches $\infty$.

We will see that the limiting behaviour of the index $\ind_{\tri(a_n,b_n)}$ can be quite different for sequences $(a_n,b_n)\to \infty$ where the slopes $\frac{a_n}{b_n}$ approach: (i) a rational limit, and (ii) an irrational limit (see Figure~\ref{fig:farey_convergents}).

\medskip
First we look at the limiting behaviour of the index $\ind_{\tri(a_n,b_n)}$ for some sequences where the slopes $\frac{a_n}{b_n}$ approach a rational number. 

\begin{thm}\label{thm:index_asymptotics}
Let $M$ be a $3$-manifold with at least two cusps and $\tri$ a $1$-efficient triangulation of $M$ with a standard cusp $\cusp$ at $T \subset \bd M$. Let $\alpha, \beta$ be a pair of dual curves on $T$, and define $\alpha_n = \alpha + n\beta$ for $n \in \Z$. Then:
\begin{equation}
\lim_{n\rightarrow \infty} \ind_{\tri(\alpha_n)}(q) = \ind^0_\tri(q) - \ind_\tri^{2\beta} (q).
\end{equation}
\end{thm}

\begin{rmk}
Note that if $[\alpha]=(a,b)$ and $[\beta]=(c,d)$ then $\alpha_n$ has slope $\frac{a+n c}{b+n d}$ approaching $\frac{c}{d} \in \Q\cup\infty$ as $n \to \infty$.
It is striking that here the limit depends on the `asymptotic direction' in which $\alpha_n$ approaches infinity. Example~\ref{rational_limit_slope_WH}  below shows this explicitly.
\end{rmk}

\begin{proof}
The strategy of the proof is the following: we start by looking at the effect of taking increasing values of $n$ in the relative Gang-Yonekura formula~\eqref{eqn:compact_version} for the standard cusp filled along $\alpha_n$. We then just apply the gluing construction from Section~\ref{sec:gluing_formula}. 

Since $\beta$ is a dual curve to $\alpha_n$ for all $n$, equation~\eqref{eqn:compact_version} for $\alpha_n$ gives: 
\begin{equation}\label{eqn:compact_version_alpha_n}
GY(\alpha_n;\underline{b}) = \sum_{k\in\Z} (-1)^k \left(q^{\frac{k}{2}} \ind_\cusp^{rel}(k \alpha_n;\underline{b}) - \ind^{rel}_\cusp(k \alpha_n +2\beta;\underline{b}) \right).
\end{equation}

By Lemma~\ref{lem:growth_cusp_index}, each summand, for a fixed $k$, has a minimal degree which is bounded below by a quadratic function of $n$. \\Taking the limit for $n\to \infty$ thus only leaves the terms which do not depend on $n$, which occur when $k=0$: $\ind_\cusp^{rel}(0;\underline{b}) $ and $- \ind^{rel}_\cusp(2\beta;\underline{b})$. In other words, for $n$ large, the relative index for the LST corresponding to the slope $\alpha_n$ will coincide with $\ind_\cusp^{rel}(0;\underline{b}) - \ind^{rel}_\cusp(2\beta;\underline{b})$ up to a certain degree $D_n$. This degree $D_n$ tends to $\infty$ when $n\to \infty$, and therefore in the limit $GY(\alpha_n;\underline{b})$ coincides with $\ind_\cusp^{rel}(0;\underline{b}) - \ind^{rel}_\cusp(2\beta;\underline{b})$.
We can conclude by applying the Gluing Theorem~\ref{thm:gluing}. 
\end{proof}

\begin{ex}\label{rational_limit_slope_WH}
Using our code (see Section~\ref{sec:code}) we can quickly compute the terms up to order $q^{50}$ for the limit of the index of $L5a1(n,1)$ as $n \to \infty$:
\begin{gather*}
1 - q - q^{2} + 4 q^{3} + 6 q^{4} + 5 q^{5} - 11 q^{6} - 32 q^{7} - 51 q^{8} - 50 q^{9} + 10 q^{10} + 121 q^{11} + 299 q^{12}\\ + 467 q^{13} + 561 q^{14} + 394 q^{15} - 131 q^{16} - 1166 q^{17} - 2669 q^{18} - 4453 q^{19} - 6031 q^{20} \\- 6575 q^{21} - 5109 q^{22} - 337 q^{23} + 8629 q^{24} + 22467 q^{25} + 40390 q^{26} + 60556 q^{27} + 78364 q^{28} \\+ 87665 q^{29} + 79188 q^{30} + 42895 q^{31} - 32325 q^{32} - 154775 q^{33} - 328507 q^{34} - 548054 q^{35}\\ - 795333 q^{36} - 1035176 q^{37} - 1211893 q^{38} - 1249293 q^{39} - 1050566 q^{40} - 507845 q^{41}\\ + 490653 q^{42} + 2036211 q^{43} + 4182792 q^{44} + 6906360 q^{45} + 10081189 q^{46} + 13429158 q^{47} \\+ 16501230 q^{48} + 18635166 q^{49} + 18965699 q^{50} + O\left(q^{51}\right),
\end{gather*}
and similarly for $L5a1(1,n)$ the limit of the index as $n \to \infty$ is:
\begin{gather*}
1 - 4 q + q^{2} + 16 q^{3} + 22 q^{4} + q^{5} - 72 q^{6} - 158 q^{7} - 210 q^{8} - 118 q^{9} + 231 q^{10} + 848 q^{11}\\ + 1642 q^{12} + 2278 q^{13} + 2224 q^{14} + 725 q^{15} - 2815 q^{16} - 8778 q^{17} - 16657 q^{18} - 24902 q^{19}\\ - 30337 q^{20} - 28345 q^{21} - 13224 q^{22} + 20877 q^{23} + 78101 q^{24} + 158654 q^{25} + 255784 q^{26}\\ + 353319 q^{27} + 422739 q^{28} + 423649 q^{29} + 304160 q^{30} + 7643 q^{31} - 519030 q^{32} - 1309493 q^{33}\\ - 2360867 q^{34} - 3609277 q^{35} - 4909000 q^{36} - 6010263 q^{37} - 6548215 q^{38} - 6043485 q^{39} \\- 3925823 q^{40} + 415766 q^{41} + 7548590 q^{42} + 17873212 q^{43} + 31466485 q^{44} + 47887782 q^{45} \\+ 65975434 q^{46} + 83634586 q^{47} + 97673663 q^{48} + 103697217 q^{49} + 96141119 q^{50} + O\left(q^{51}\right).
\end{gather*}
In both cases the terms displayed agree with $\ind^0_\tri(q) - \ind_\tri^{2\beta} (q)$ for $\alpha = (0,1)$, $\beta = (1,0)$ and  $\alpha = (1,0)$, $\beta = (0,1)$ respectively.
\end{ex}

\begin{figure}
\centering
\includegraphics[width=0.36\linewidth]{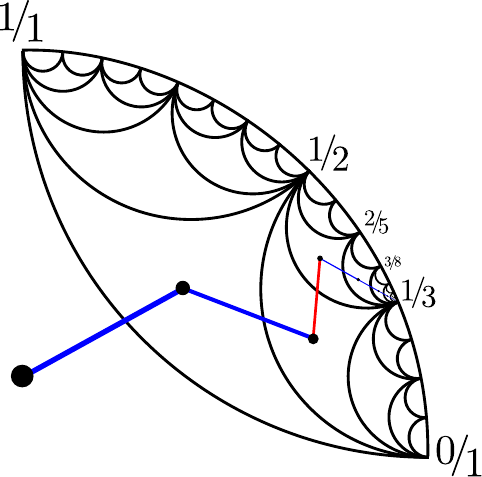}
\caption{The limiting behaviour of the slopes approximating a rational number in Theorem~\ref{thm:index_asymptotics} correspond to infinite paths on the Farey graph which are eventually right- or left-veering. All other paths (moving monotonically away from the centre) correspond instead to irrational numbers.}
\label{fig:farey_convergents}
\end{figure}

Next we look at the limiting behaviour of the index $\ind_{\tri(p_n,q_n)}$ for some special sequences where the slopes $p_n/q_n$ approach an irrational number $r$. 
For simplicity, we assume $r\in[0,1]$ (otherwise just need to add an initial $a_0 \in \Z$). 
Then $r$ has an infinite continued fraction expansion with convergents
\begin{equation}
\frac{p_n}{q_n} = 
    \cfrac{1}{a_1+\cfrac{1}{a_2+\cfrac{1}{a_3+\hspace{0.5cm}\ddots
\hspace{0.5cm}\raisebox{-3ex}{$\cfrac{1}{a_{n-1}+\cfrac{1}{a_n}}$} }}}
\end{equation}
where $a_i, p_i, q_i$ are positive integers. This gives a sequence of rationals $\frac{p_n}{q_n}$ 
with 
\begin{equation} \lim_{n\to\infty} \frac{p_n}{q_n}= r, \end{equation}
and such that 
\begin{equation}\label{dual_convergents}  p_n q_{n-1} - q_n p_{n-1}= \pm 1 \text{ for each } n. \end{equation}

\begin{thm}\label{thm:real_fillings}
Let $M$ be a $3$-manifold with at least two cusps, and $\tri$ a $1$-efficient triangulation of $M$ with a standard cusp $\cusp$ at $T \subset \bd M$.
Let $r \in \R\setminus\Q$ be an irrational number, with
continued fraction convergents $\frac{p_n}{q_n}$ as described above. 
Then 
$$\lim_{n \to \infty} \ind_{\tri(p_n,q_n)} (q) = \ind_\tri^0 (q).$$
\end{thm}

\begin{proof}
Let $\alpha_n=(p_n,q_n)$. Then $\beta_n=(p_{n-1},q_{n-1})$ is dual to $\alpha_n$ by~\eqref{dual_convergents}. Recall the compact relative version of the Gang-Yonekura formula for the standard cusp: 
$$GY(\alpha_n;\underline{b}) = \sum_k (-1)^k \left(q^{\frac{k}{2}} \ind_\cusp^{rel}(k \alpha_n;\underline{b}) - \ind^{rel}_\cusp(k \alpha_n +2\beta_n;\underline{b}) \right).$$
 
If $k \neq 0$, then the entries $k  \alpha_n$ and $k  \alpha_n + 2\beta_n$ are not bounded as $n\to\infty$, since the denominators 
$q_n$ form an increasing sequence of positive integers. Hence the minimal degree for the relative cusp indices $\ind_\cusp^{rel}(k \alpha_n;\underline{b})$ and $\ind^{rel}_\cusp(k \alpha_n +2\beta_n;\underline{b})$
approach $+\infty$ by Lemma~\ref{lem:growth_cusp_index}. 

Therefore we can restrict our analysis to $k=0$. In this case, we are left with 
$$\lim_{n \to \infty} \ind^{rel}_\cusp(0;\underline{b}) - \ind^{rel}_\cusp(2\beta_n;\underline{b}).$$
Now the entries of $\beta_n=(p_{n-1},q_{n-1})$ are not bounded in absolute value since $q_{n-1} \to \infty$.
 We can then conclude by applying Lemma~\ref{lem:growth_cusp_index} again and the gluing formula~\ref{thm:gluing}. 
\end{proof}

\begin{ex}\label{ex:golden_ratio}
It is well known that the ratios $\frac{F_{n+1}}{F_n}$ of consecutive Fibonacci numbers converge to the golden ratio $\phi$. Using our certified code (described in~Section~\ref{sec:code}) we can see Theorem~\ref{thm:real_fillings} in action, \textit{e.g.}~for the Whitehead link and the complement of the figure eight knot (\textit{cf.}~Section~\ref{sec:closed} for an interpretation of the latter indices). We display the results of the computations for the first $10$ coefficients in Tables~\ref{tab:whitehead_fibonacci} and~\ref{tab:fig8_fibonacci} below. Note that the coefficients considered here stabilise after $n =6$. 

\begin{center}
\begin{table}[!htbp]
\begin{tabular}{|c|c|c|}
\hline
$n$ & $\frac{F_{n+1}}{F_n}$ & $\ind_{L5a1(F_{n+1},F_n)}(q)$ \\
\hline
\multirow{2}{*}{$0$} & \multirow{2}{*}{$\frac{1}{0}$} & $O\left(q^{11}\right)$ \\
 & & \\ \hline
\multirow{2}{*}{$1$} & \multirow{2}{*}{$\frac{1}{1}$} & $1 - 2 q - 3 q^2 + 2 q^3 + 8 q^4 + 18 q^5 + 18 q^6 + 14 q^7$ \\
 & & $- 12 q^8 - 52 q^9 - 106 q^{10} + O\left(q^{11}\right)$ \\ \hline
\multirow{3}{*}{$2$} & \multirow{3}{*}{$\frac{2}{1}$} & $1 - q^{\frac12} - q - 2 q^{\frac32} - q^2 + 2 q^{\frac52} + 6 q^3 + 8 q^{\frac72} + 9 q^4 + 11 q^{\frac92}$ \\
 & & $+ 12 q^5 + 6 q^{\frac{11}2} - 5 q^6 - 17 q^{\frac{13}2} - 34 q^7 - 57 q^{\frac{15}2}$ \\
 & & $- 79 q^8 - 100 q^{\frac{17}2} - 118 q^9 - 124 q^{\frac{19}2} - 118 q^{10} + O\left(q^{\frac{21}2}\right)$ \\ \hline
\multirow{2}{*}{$3$} & \multirow{2}{*}{$\frac{3}{2}$} & $1 - 3 q + q^2 + 18 q^3 + 18 q^4 - 10 q^5 - 94 q^6 - 185 q^7$ \\
 & & $- 214 q^8 - 61 q^9 + 413 q^{10} + O\left(q^{11}\right)$ \\ \hline
\multirow{2}{*}{$4$} & \multirow{2}{*}{$\frac{5}{3}$} & $1 - 3 q + q^2 + 16 q^3 + 16 q^4 - 12 q^5 - 92 q^6 - 167 q^7$ \\
 & & $- 172 q^8 + 11 q^9 + 495 q^{10} + O\left(q^{11}\right)$ \\ \hline
\multirow{2}{*}{$5$} & \multirow{2}{*}{$\frac{8}{5}$} & $1 - 3 q + 3 q^2 + 18 q^3 + 14 q^4 - 29 q^5 - 123 q^6 - 198 q^7$ \\
 & & $- 151 q^8 + 160 q^9 + 834 q^{10} + O\left(q^{11}\right)$ \\ \hline
\multirow{2}{*}{$6$} & \multirow{2}{*}{$\frac{13}{8}$} & $1 - 3 q + q^2 + 16 q^3 + 16 q^4 - 11 q^5 - 87 q^6 - 160 q^7$ \\
 & & $- 161 q^8 + 16 q^9 + 482 q^{10} + O\left(q^{11}\right)$ \\ \hline
\multirow{2}{*}{$7$} & \multirow{2}{*}{$\frac{21}{13}$} & $1 - 3 q + q^2 + 16 q^3 + 16 q^4 - 11 q^5 - 87 q^6 - 160 q^7$ \\
 & & $- 161 q^8 + 16 q^9 + 482 q^{10} + O\left(q^{11}\right)$ \\ \hline
\end{tabular}
\caption{The value of the index on one cusp of the Whitehead link (using SnapPy's basis for L5a1) for surgery coefficients approximating the golden ratio.}
\label{tab:whitehead_fibonacci}
\end{table}
\end{center}

\begin{center}
\begin{table}[!htbp]
\begin{tabular}{|c|c|c|}
\hline
$n$ & $\frac{F_{n+1}}{F_n}$ &  $\ind_{m004(F_{n+1},F_n)(q)}$\\
\hline
\multirow{2}{*}{$1$} & \multirow{2}{*}{$\frac{1}{1}$} & $1 + O\left(q^{11}\right)$ \\
   && \\ \hline
\multirow{2}{*}{$2$} & \multirow{2}{*}{$\frac{2}{1}$} & $1 + O\left(q^{11}\right)$ \\
   && \\ \hline
\multirow{2}{*}{$3$} & \multirow{2}{*}{$\frac{3}{2}$} & $1 - 2 q - 3 q^2 - q^3 + 3 q^4 + 13 q^5 + 18 q^6$ \\
   && $+ 26 q^7 + 23 q^8 + 14 q^9 - 10 q^{10} + O\left(q^{11}\right)$ \\ \hline
\multirow{2}{*}{$4$} & \multirow{2}{*}{$\frac{5}{3}$} & $1 - 2 q - 3 q^2 + 6 q^4 + 16 q^5 + 20 q^6 + 21 q^7$ \\
   && $+ 6 q^8 - 23 q^9 - 69 q^{10} + O\left(q^{11}\right)$ \\ \hline
\multirow{2}{*}{$5$} & \multirow{2}{*}{$\frac{8}{5}$} & $1 - 3 q - 4 q^2 + 4 q^3 + 15 q^4 + 29 q^5 + 29 q^6$ \\
   && $+ 17 q^7 - 29 q^8 - 102 q^9 - 197 q^{10} + O\left(q^{11}\right)$ \\ \hline
\multirow{2}{*}{$6$} & \multirow{2}{*}{$\frac{13}{8}$} & $1 - 2 q - 3 q^2 + 2 q^3 + 8 q^4 + 18 q^5 + 18 q^6$ \\
   & & $+ 14 q^7 - 12 q^8 - 52 q^9 - 106 q^{10} + O\left(q^{11}\right)$ \\ \hline
\multirow{2}{*}{$7$} & \multirow{2}{*}{$\frac{21}{13}$} & $1 - 2 q - 3 q^2 + 2 q^3 + 8 q^4 + 18 q^5 + 18 q^6$ \\
   & & $+ 14 q^7 - 12 q^8 - 52 q^9 - 106 q^{10} + O\left(q^{11}\right)$ \\ \hline
\end{tabular}
\caption{The value of the index  the complement of $4_1$ (using SnapPy's basis for m004) for surgery coefficients approximating the golden ratio.}
\label{tab:fig8_fibonacci}
\end{table}
\end{center}

Note that the results displayed in Tables~\ref{tab:whitehead_fibonacci} and~\ref{tab:fig8_fibonacci} do indeed coincide with the first ten coefficients of $\ind_{L5a1}^{0}(q)$ and $\ind_{m004}^{0}(q)$ respectively.

If instead we look at the continued fraction expansion $$\pi = 3 + \cfrac{1}{7 + \cfrac{1}{15 + \cfrac{1}{1 + \cfrac{1}{292 + \cfrac{1}{\dots}}}}}$$
we obtain the sequence converging to 
$\ind_{m004}^0(q)$ shown in Table~\ref{tab:fig8_pi}.

\begin{center}
\begin{table}[!htbp]
\begin{tabular}{|c|c|c|}
\hline
$n$ & $\frac{p_n}{q_n}$ 
&  $\ind_{m004(p_n,q_n)(q)}$\\
\hline
\multirow{2}{*}{$1$} & \multirow{2}{*}{$\frac{3}{1}$} & $1 + O\left(q^{11}\right)$ \\
   && \\ \hline
\multirow{2}{*}{$2$} & \multirow{2}{*}{$\frac{22}{7}$} & $1 - 2 q - 4 q^2 + q^3 + 7 q^4 + 19 q^5 + 22 q^6 + 24 q^7$ \\
   & & $- 2 q^{\frac{15}2} + 4 q^8 - 2 q^{\frac{17}2} - 30 q^9 - 3 q^{\frac{19}2} - 81 q^{10} + O\left(q^{\frac{21}{2}}\right)$ \\ \hline
\multirow{2}{*}{$3$} & \multirow{2}{*}{$\frac{333}{106}$} & $1 - 2 q - 3 q^2 + 2 q^3 + 8 q^4 + 18 q^5 + 18 q^6 + 14 q^7$ \\
   &  & $- 12 q^8 - 52 q^9 - 106 q^{10} + O\left(q^{11}\right)$ \\ \hline
\multirow{2}{*}{$4$} & \multirow{2}{*}{$\frac{355}{113}$} & $1 - 2 q - 3 q^2 + 2 q^3 + 8 q^4 + 18 q^5 + 18 q^6 + 14 q^7$ \\
   &  & $- 12 q^8 - 52 q^9 - 106 q^{10} + O\left(q^{11}\right)$ \\ \hline
\multirow{2}{*}{$5$} & \multirow{2}{*}{$\frac{103993}{33102}$} & $1 - 2 q - 3 q^2 + 2 q^3 + 8 q^4 + 18 q^5 + 18 q^6 + 14 q^7$ \\
   & & $- 12 q^8 - 52 q^9 - 106 q^{10} + O\left(q^{11}\right)$ \\ \hline
\end{tabular}
\caption{Values of the index (for $q$-degree up to $10$) of surgeries on m004, with slopes $\frac{p_n}{q_n}$ determined by successive approximations of $\pi$.}
\end{table}
\label{tab:fig8_pi}
\end{center}
Note that in this case we get an almost immediate convergence for the first ten terms, since the numerators and denominators become quite big very quickly.
\end{ex}

%&&&&&&&&&&&&&&&&&&&&&&&&&&&&&&&&&&&&&&&&&&&&&&&&&&&&&&&&&&&&&&&&&&&&&&&&&&&&&
%&&&&&&&&&&&&&&&&&&&&&&&&&&&&&&&&&&&&&&&&&&&&&&&&&&&&&&&&&&&&&&&&&&&&&&&&&&&&&

\section{Examples}\label{sec:examples}

In this section we are going to collect some basic and interesting examples obtained by applying the Gang-Yonekura formula, either explicitly or using our code. 
We remark that, for all $1$-efficient triangulations considered, the index we obtain by direct computation on the filled manifold coincides with the value we obtain from applying equation~\eqref{eqn:GYformula} with our code, up to a chosen threshold. Further empirical validations of our code are presented in Section~\ref{sec:code}. 
There are many ways in which an equation such as~\eqref{eqn:GYformula} can fail to give a well-defined answer; there could be infinitely many terms with increasingly negative exponents. Otherwise 
if infinitely many summands provide non-trivial contributions to some coefficient, this might blow up, or oscillate. The only situation in which we can get a proper $q$-series is when there are only finitely many summands of degree less or equal of any fixed threshold. From this perspective, many computations in this and the next section do not yield well-defined $q$-series. In several cases however, it is very easy to isolate the `ill behaved' contributions; this process could potentially yield a `normalised' index, defined on not necessarily $1$-efficient triangulations. 

%&&&&&&&&&&&&&&&&&&&&&&&&&&&&&&&&&&&&&&&&&&&&&&&&&&&&&&&&&&&&&&&&&&&&&&&&&&&&&
%&&&&&&&&&&&&&&&&&&&&&&&&&&&&&&&&&&&&&&&&&&&&&&&&&&&&&&&&&&&&&&&&&&&&&&&&&&&&&

\subsection{Alternating torus knots}\label{ssec:alternating_torus}~\\

The following computations prove that the $3$D index of alternating torus knots -- for a specific triangulation -- is somewhat trivial. Note that, unlike with hyperbolic knots~\cite{garoufalidis20163d}, where the Epstein-Penner provides a `canonical' way of computing the index, to the best of the author's knowledge, this is not the case for the complement of torus knots (\textit{cf.}~Question (4) from Section~\ref{sec:questions}). 

The index (again for a specific triangulation) of the trefoil was computed in~\cite[Sec.~4.3]{dimofte20133} and~\cite[Sec.~11]{garoufalidis20163d}. It takes a particularly easy form; if $\tri_{(3,2)}$ denote the $2$-tetrahedra triangulation for the complement of $3_1 = T_{3,2}$, then 
$$\ind_{\tri_{(3,2)}}^{(x,y)}(q) = \delta_{0,x+6y}.
$$
A similar result was proved (and slightly extended) in Max Jolley's thesis~\cite{jolley}. It was conjectured in~\cite{dimofte20133} on physical grounds that a similar pattern should hold more generally:
\begin{con}[Sec.~4.3~\cite{dimofte20133}]
For coprime $p,q >0$ such that $pq \equiv 0 \mod{2}$, let $\tri_{(p,q)}$ be a $1$-efficient ideal triangulation of the complement of the $(p,q)$ torus knot. Then
$$\ind^{(x,y)}_{\tri_{(p,q)}}(q) = \delta_{0,x + pqy}.$$ 
\end{con}

In this section we are going to prove this conjecture.
Recall that the alternating torus knot $T_{2n-1,2}$ can be obtained by the surgery diagram shown in Figure~\ref{fig:torus_complement}.
\begin{figure}[!htbp]
\centering
\includegraphics[width=8.5cm]{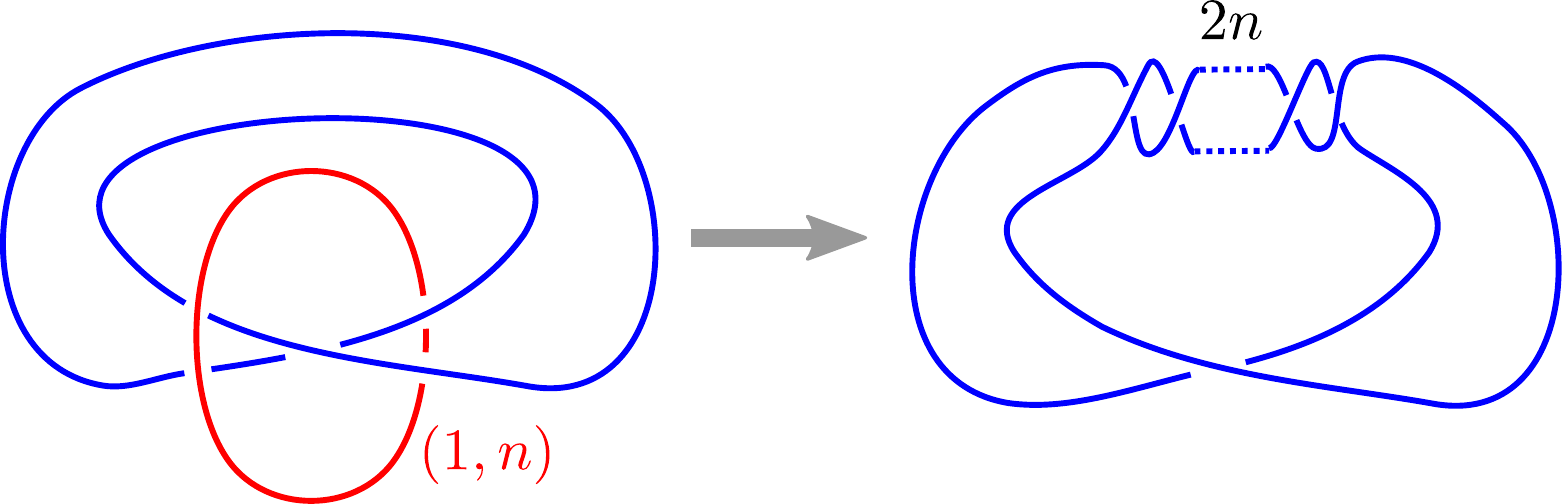}
\caption{Performing $\frac{1}{n}$-sloped Dehn surgery on the red unknotted component yields the knot $T_{2n-1,2}$.}
\label{fig:torus_complement}
\end{figure}

One possible  triangulation $\tri_L$ of the complement of the link $L$ in the left of Figure~\ref{fig:torus_complement} can be found using SnapPy; its gluing matrix $G_L$ (together with edge labellings and weights on its left) is: 
$$G_L = 
\left|
\begin{matrix}
k_1 & e_1\\ * & e_2\\ *  & e_3 \\k_2  & e_4\\ x_1 & 2m_1\\ y_1 & \ell_1'\\x_2 & 2m_2\\y_2 & \ell_2'
\end{matrix}
\right|
\left( \begin{matrix}
1 & 0 & 1 & 1 & 1 & 0 & 0 & 1 & 2 & 0 & 2 & 1\\ 
0 & 1 & 0 & 0 & 0 & 1 & 2 & 0 & 0 & 2 & 0 & 0\\ 
1 & 1 & 1 & 1 & 1 & 1 & 0 & 0 & 0 & 0 & 0 & 0\\ 
0 & 0 & 0 & 0 & 0 & 0 & 0 &  1 & 0 & 0 & 0 & 1\\ 
1 & 0 & 0 & 0 & 1 & 0 & -1 & 0 & 0 & 0 & 1 & 0\\ 
-1 & 0 & 0 & 0 & -1 & 0 & 1 & 0 & -1 & 1 & -1 & 1\\ 
0 & 0 & 1 & -1 & 0 & 0 & 0 & 0 & 0 & 0 & 0 & 0\\ 
0 & 0 & 0 & 0 & 0 & -1 & -1 & 0 & 0 & -1 & 0 & 0
\end{matrix} \right)
$$
Note that one cusp is standard. We can use this surgery description with Theorem~\ref{thm:main} to obtain the following:
\begin{prop}\label{prop:index_of_torus}
Let $\tri_{2n-1,2}$ be the triangulation obtained by replacing the standard cusp in $\tri_L$ with a layered solid torus with slope $\frac{1}{n}$. Then $\tri_{2n-1,2}$ is a triangulation for the complement of $T_{2n-1,2}$, and 
$$\ind_{\tri_{2n-1,2}}^{(x,y)}(q) = \delta_{0,x + 2(2n-1)y}.$$ 
\end{prop}

In order to prove this result, we will need the following lemma. Note that, despite being closely related, this result differs from the quadratic identity~\eqref{eqn:quadratic_identity}; here the first variables of $\itet$ are opposite to one another, rather than being equal.
\begin{lem}\label{lem:opposite_alpha_quadratic}
Let $r, s, t$ be integers; then
$$\sum_{k\in \Z} \itet(r,k+s) \itet(-r, k + t)q^{k} = \left(-q^\frac12\right)^{r }q^{-t} \delta_{t, r+s} $$
\end{lem}
\begin{proof}
From equation~\eqref{enq:generating_function_tetra}, we get that
$$\frac{(q^{1-\frac{r}{2}} z_1^{-1})_\infty (q^{1+\frac{r}{2}}z_2^{-1})_\infty}{ (q^{-\frac{r}{2}} z_1)_\infty (q^{\frac{r}{2}}z_2)_\infty} = z_1^s z_2^t \sum_{k_1,k_2\in \Z}  \itet(r,k+s) \itet(-r, k + t) z_1^{k_1} z_2^{k_2} .$$
Setting $z = z_1 = q z_2^{-1}$ yields
$$z^{s-t} q^{t}\sum_{k_1,k_2\in \Z}  \itet(r,k_1+s) \itet(-r, k_2 + t) z^{k_1-k_2} q^{k_2}= \frac{(q^{1-\frac{r}{2}} z^{-1})_\infty (q^{\frac{r}{2}} z)_\infty}{(q^{-\frac{r}{2}}z)_\infty (q^{1+\frac{r}{2}}z^{-1})_\infty}.$$
The right-hand side simplifies to $$\frac{(q^{1-\frac{r}{2}} z^{-1})_r}{(q^{-\frac{r}{2}}z)_r} = (-q^\frac12 z^{-1})^r.$$

Therefore, the left-hand side in the statement is the coefficient of $z^0$ in $q^{\frac{r}{2} -t} z^{t - r - s}$. Clearly this is equal to $(-1)^r q^{\frac{r}{2} - t} \delta_{t, r+s}$. 
\end{proof}

\begin{proof}[Proof of Proposition~\ref{prop:index_of_torus}]
Consider the gluing matrix $G_L$. The choice of edges $e_1$ and $e_4$ is admissible (in the sense of~\cite{garoufalidis20151}); the corresponding index is
\begin{gather*}
\ind_{\tri_L}^{(2x_1,y_1,2x_2,y_2)}(q) =\\ =\left(-q^\frac12 \right)^{y_1-y_2}\sum_{k_1,k_2 \in \Z} q^{k_1+k_2}
\jtet(k_1 + x_1 -y_1, 0 ,k_1+x_2 ) 
\jtet( k_1  -x_2, k_1+x_1-y_1, -y_2 )\\
\jtet( -x_1+y_1-y_2, k_1+k_2, 2k_1-y_1 )
\jtet(y_1-y_2 ,2k_1+x_1-y_1 , k_1+k_2+y_1) \\
= \left(-q^\frac12 \right)^{y_1}\sum_{k_1\in \Z} q^{-k_1} \itet(-x_1+y_1+x_2, -k_1-x_2)  \itet(-x_1+y_1-x_2,-k_1+x_2-y_2 )  \\ 
\left( \sum_{k_2 \in \Z} q^{k_2} \itet(2k_1 +x_1-2y_1+y_2, -k_1 + k_2+y_1 )  \itet(-2k_1-x_1 +2y_1-y_2, k_1+k_2+y_2 )\right).
\end{gather*}
By Lemma~\ref{lem:opposite_alpha_quadratic}, the inner product coincides with $q^{-y_1}\delta_{x_1-y_1,0}$.\\
So the index becomes (after the substitutions $x_1 =y_1$ and $k_1 \mapsto -k_1$)
$$
\left(-q^\frac12 \right)^{-y_1}\delta_{x_1-y_1,0} \sum_{k_1\in \Z} q^{k_1} \itet(x_2, k_1-x_2)  \itet(-x_2,k_1+x_2-y_2 ).
$$
A further application of Lemma~\ref{lem:opposite_alpha_quadratic} gives the total sum:
\begin{equation}\label{eqn:index_of_link_torus}
\ind_{\tri_L}^{(2x_1,y_1,2x_2,y_2)}(q) = \delta_{x_1-y_1,0} \delta_{x_2-y_2,0}.
\end{equation}

We can now apply the Gang-Yonekura formula to equation~\eqref{eqn:index_of_link_torus} in order to complete the proof. Replace the standard cusp with a LST with slope $\alpha_n = \mu + n\lambda$. The resulting triangulation $\tri_{2n-1,2}$ gives the complement of the knot $T_{2n-1,2} \subset S^3$. The dual of $\alpha_n$ is just $\beta = \lambda$.
Now, 
\begin{gather*}
\ind_{\tri_{2n-1,2}}^{(x_2,y_2)}(q) = \frac12\bigg(\sum_{k \in \Z} \left(q^k + q^{-k}\right) \ind_M (2k\mu+2kn \lambda)(q) \\  - \ind_M (2k\mu+(2kn+2)(q) \lambda) - \ind_M (2k\mu+(2kn-2) \lambda)(q)\bigg).
\end{gather*}

Applying equation~\eqref{eqn:index_of_link_torus} yields
$$\ind_{\tri_{2n-1,2}}^{(x_2,y_2)}(q) = \frac12\left(\sum_{k \in \Z} \left(q^k + q^{-k}\right) \delta_{k(2n-1),0} - \delta_{k(2n-1)+2,0} - \delta_{k(2n-1)-2,0}\right)\delta_{x_2-y_2,0}.$$
The only non-zero terms correspond to the pairs $(n,k) = (n,0)$, and $(n,k) = (0, \pm 1),(1, \pm 1)$. Therefore
$$\ind_{\tri_{2n-1,2}}^{(x_2,y_2)}(q) = \begin{cases}
\delta_{x_2,y_2} & \text{ if } n \neq 0,1\\
0 & \text{ otherwise.}
\end{cases}$$

Since the linking number of the two components in Figure~\ref{fig:torus_complement} is $2$, the image of the coefficients $(x_2,y_2)$ after the surgery is $ -2(2n-1)$.
\end{proof}~\\

%&&&&&&&&&&&&&&&&&&&&&&&&&&&&&&&&&&&&&&&&&&&&&&&&&&&&&&&&&&&&&&&&&&&&&&&&&&&&&
%&&&&&&&&&&&&&&&&&&&&&&&&&&&&&&&&&&&&&&&&&&&&&&&&&&&&&&&&&&&&&&&&&&&&&&&&&&&&&

\subsection{The Whitehead link}\label{ssec:whitehead}~\\

We begin this section by recalling that the $3$D index is conjectured (see \textit{e.g.}~\cite{dimofte20133, hodgson2021asymptotics}) to contain information on the volume of cusped hyperbolic manifolds. We can however provide examples of cusped manifold with many identical geometric invariants, but different $3$D index. These manifolds are those from~\cite{hodgson1992surgeries}, namely $(4,1)$ and $(4,-3)$ fillings on one cusp of the Whitehead's knot complement L5a1.\\ 

For L5a1$(4,1)$ \textit{i.e.}~m039 in the SnapPy census~\cite{SnapPy}, we get the index $$1 - q + q^{2} + 7 q^{3} + 10 q^{4} - q^{5} - 35 q^{6}- 78 q^{7} - 106 q^{8} - 75 q^{9} + 84 q^{10} + O\left(q^{11}\right),$$
while for L5a1$(4,-3)$ = m035 our computations give $$1 - 5 q + q^{2} + 27 q^{3} + 34 q^{4} + 3 q^{5} - 119 q^{6}- 270 q^{7} - 382 q^{8} - 287 q^{9} + 256 q^{10} + O\left(q^{11}\right).$$
Note that both manifolds have volume $ = 3.1772932786\ldots$, same homology, cusp volume and shape, Chern-Simons invariants, but different Alexander polynomials.
\begin{table}[!htbp]
\begin{tabular}{|c|c|c|c|}
\hline
Slope & Manifold & Index & Volume \\
\hline
\multirow{2}{*}{$(1,-5)$} & \multirow{2}{*}{s007} & $1 - 3 q + 11 q^{3} + 17 q^{4} + 5 q^{5}- 42 q^{6}$ & \multirow{2}{*}{3.4712} \\
 & & $ - 104 q^{7} - 150 q^{8} - 109 q^{9} + 95 q^{10}+ O\left(q^{11}\right)$ & \\
\hline
\multirow{2}{*}{$(1,-4)$} & \multirow{2}{*}{m062} & $1 - 3 q + 11 q^{3} + 17 q^{4} + 5 q^{5}- 42 q^{6}$ & \multirow{2}{*}{3.3667} \\
 & & $ - 104 q^{7} - 150 q^{8} - 109 q^{9} + 95 q^{10}+ O\left(q^{11}\right)$ & \\
\hline
\multirow{2}{*}{$(1,-3)$} & \multirow{2}{*}{m029} & $1 - 3 q + 11 q^{3} + 17 q^{4} + 5 q^{5}- 44 q^{6}$ & \multirow{2}{*}{3.1485} \\
 & & $ - 110 q^{7} - 166 q^{8} - 131 q^{9} + 77 q^{10}+ O\left(q^{11}\right)$ & \\
\hline
\multirow{2}{*}{$(1,-2)$} & \multirow{2}{*}{m007} & $1 - 3 q - 2 q^{2} + 9 q^{3} + 19 q^{4} + 25 q^{5}$ & \multirow{2}{*}{2.5689} \\
 & & $- 48 q^{7} - 132 q^{8} - 219 q^{9} - 255 q^{10}+ O\left(q^{11}\right)$ & \\
\hline
\multirow{2}{*}{$(1,-1)$} & \multirow{2}{*}{trefoil complement} & $1$ & \multirow{2}{*}{0} \\
 & & & \\
\hline
\multirow{2}{*}{$(1,0)$} & \multirow{2}{*}{solid torus} & $0$ & \multirow{2}{*}{0} \\
 & & & \\
\hline
\multirow{2}{*}{$(1,1)$} & \multirow{2}{*}{trefoil complement} & $1$ & \multirow{2}{*}{0} \\
 & & & \\
\hline
\multirow{2}{*}{$(1,2)$} & \multirow{2}{*}{m006} & $1 - 3 q - 2 q^{2} + 9 q^{3} + 19 q^{4} + 25 q^{5}$ & \multirow{2}{*}{2.5689} \\
 & & $- 48 q^{7} - 132 q^{8} - 219 q^{9} - 255 q^{10}+ O\left(q^{11}\right)$ & \\
\hline
\multirow{2}{*}{$(1,3)$} & \multirow{2}{*}{m030} & $1 - 3 q + 11 q^{3} + 17 q^{4} + 5 q^{5}- 44 q^{6}$ & \multirow{2}{*}{3.1485} \\
 & & $ - 110 q^{7} - 166 q^{8} - 131 q^{9} + 77 q^{10}+ O\left(q^{11}\right)$ & \\
\hline
\multirow{2}{*}{$(1,4)$} & \multirow{2}{*}{m061} & $1 - 3 q + 11 q^{3} + 17 q^{4} + 5 q^{5}- 42 q^{6}$ & \multirow{2}{*}{3.3667} \\
 & & $ - 104 q^{7} - 150 q^{8} - 109 q^{9} + 95 q^{10}+ O\left(q^{11}\right)$ & \\
\hline
\multirow{2}{*}{$(1,5)$} & \multirow{2}{*}{s006} & $1 - 3 q + 11 q^{3} + 17 q^{4} + 5 q^{5}- 42 q^{6}$ & \multirow{2}{*}{3.4712} \\
 & & $ - 104 q^{7} - 150 q^{8} - 109 q^{9} + 95 q^{10}+ O\left(q^{11}\right)$ & \\
\hline
\end{tabular}
\caption{The value of the index obtained by performing $\frac{1}{n}$ fillings on the first cusp in the Whitehead's link complement, using SnapPy's m129 geometric basis. Note that the entries are symmetric, \textit{cf.}~\cite{hodgson1992surgeries}.}
\label{tab:whitehead}
\end{table}
We further performed certified computations for the $\pm\frac{1}{n}$ fillings of one cusp in $L5a1$; these are well known to be the complements of twist knots. A suitable (non-minimal) $1$-efficient triangulation with a standard cusp for the Whitehead link complement is provided in~\cite[Sec.~4]{howie2020polynomials}. 
We verified that the $3$D index (with trivial boundary conditions) computed using~\eqref{eqn:GYformula} agrees with the index computed directly from a triangulation of the complement. Results are summarised in Table~\ref{tab:white_twist} (see also Section~\ref{sec:code}). 
\begin{table}[!htbp]
\centering
\begin{tabular}{|c|c|c|c|}
\hline
Twist & Knot name & $3$D index & Volume \\
\hline
\multirow{2}{*}{$K_{-5}$} & \multirow{2}{*}{$11_{552}$} & $1-4 q+q^{2}+16 q^{3}+22 q^{4}+$ & \multirow{2}{*}{3.5538} \\
 &   & $q^{5}-72 q^{6}-158 q^{7}-212 q^{8}-118 q^{9}+231 q^{10}+O\left(q^{11}\right)$ & \\
\hline
\multirow{2}{*}{$K_{-4}$} & \multirow{2}{*}{$9_{2}$} & $1-4 q+q^{2}+16 q^{3}+22 q^{4}+$ & \multirow{2}{*}{3.4867} \\
 &   & $q^{5}-74 q^{6}-158 q^{7}-210 q^{8}-116 q^{9}+231 q^{10}+O\left(q^{11}\right)$ & \\
\hline
\multirow{2}{*}{$K_{-3}$} & \multirow{2}{*}{$7_{2}$} & $1-4 q+q^{2}+16 q^{3}+20 q^{4}+$ & \multirow{2}{*}{3.3317} \\
 &   & $q^{5}-72 q^{6}-156 q^{7}-206 q^{8}-98 q^{9}+275 q^{10}+O\left(q^{11}\right)$ & \\
\hline
\multirow{2}{*}{$K_{-2}$} & \multirow{2}{*}{$5_{2}$} & $1-4 q-q^{2}+16 q^{3}+26 q^{4}+23 q^{5}-$ & \multirow{2}{*}{2.8281} \\
 &   & $34 q^{6}-122 q^{7}-239 q^{8}-312 q^{9}-221 q^{10}+O\left(q^{11}\right)$ & \\
\hline
\multirow{2}{*}{$K_{-1}$} & \multirow{2}{*}{$3_{1}$} & $1+O\left(q^{11}\right)$ & \multirow{2}{*}{0} \\
 &   &  & \\
\hline
\multirow{2}{*}{$K_{0}$} & \multirow{2}{*}{$\bigcirc$} & $O\left(q^{11}\right)$ & \multirow{2}{*}{0} \\
 &   &  & \\
\hline
\multirow{2}{*}{$K_{1}$} & \multirow{2}{*}{$4_{1}$} & $1-2 q-3 q^{2}+2 q^{3}+8 q^{4}+18 q^{5}+18 q^{6}+$ & \multirow{2}{*}{2.0299} \\
 &   & $14 q^{7}-12 q^{8}-52 q^{9}-106 q^{10}+O\left(q^{11}\right)$ & \\
\hline
\multirow{2}{*}{$K_{2}$} & \multirow{2}{*}{$6_{1}$} & $1-4 q+q^{2}+18 q^{3}+22 q^{4}+q^{5}-$ & \multirow{2}{*}{3.164} \\
 &   & $78 q^{6}-178 q^{7}-254 q^{8}-188 q^{9}+167 q^{10}+O\left(q^{11}\right)$ & \\
\hline
\multirow{2}{*}{$K_{3}$} & \multirow{2}{*}{$8_{1}$} & $1-4 q+q^{2}+16 q^{3}+22 q^{4}+3 q^{5}-$ & \multirow{2}{*}{3.4272} \\
 &   & $72 q^{6}-158 q^{7}-212 q^{8}-118 q^{9}+227 q^{10}+O\left(q^{11}\right)$ & \\
\hline
\multirow{2}{*}{$K_{4}$} & \multirow{2}{*}{$10_{1}$} & $1-4 q+q^{2}+16 q^{3}+22 q^{4}+q^{5}-$ & \multirow{2}{*}{3.5262} \\
 &   & $72 q^{6}-156 q^{7}-210 q^{8}-118 q^{9}+229 q^{10}+O\left(q^{11}\right)$ & \\
\hline
\multirow{2}{*}{$K_{5}$} & \multirow{2}{*}{$12a_{803}$} & $1-4 q+q^{2}+16 q^{3}+22 q^{4}+q^{5}-$ & \multirow{2}{*}{3.5739} \\
 &   & $72 q^{6}-158 q^{7}-210 q^{8}-116 q^{9}+231 q^{10}+O\left(q^{11}\right)$ & \\
\hline
\end{tabular}
\caption{From left to right:  the twist knots $K_n$, their Rolfsen name, their total $3$D index (computed as $(1,n)$-filling on the first cusp of L5a1), and their volume.}
\label{tab:white_twist}
\end{table}

%&&&&&&&&&&&&&&&&&&&&&&&&&&&&&&&&&&&&&&&&&&&&&&&&&&&&&&&&&&&&&&&&&&&&&&&&&&&&&
%&&&&&&&&&&&&&&&&&&&&&&&&&&&&&&&&&&&&&&&&&&&&&&&&&&&&&&&&&&&&&&&&&&&&&&&&&&&&&

\subsection{3D index of the ``magic'' manifold}\label{ssec:magic}~\\

Let $N$ be s776, a.k.a. the ``magic manifold''~\cite{martelli2006dehn}, \emph{i.e.}~the complement of the $3$-chain link $6^3_1$. 

We performed certified computations (\textit{cf.}~Section~\ref{sec:code}) of the series we obtain as index for surgeries on a cusp in $N$, up to $q$-degree $10$. The results, summarised in Table~\ref{tab:magic1} show that (at least up to the degrees computed), the index given by applying the Gang-Yonekura formula to $N$ coincides with the index (with trivial boundary conditions) computed on the corresponding filled census manifold (see also~\cite{thompson2025triangulations}). 
\begin{table}[!htbp]
\begin{tabular}{|c|c|c|c|}
\hline
Slope & Census name & $3$D Index & Volume\\
\hline
\multirow{2}{*}{$(-1,2)$} & \multirow{2}{*}{m203} & $1-4 q+4 q^{2}+24 q^{3}+11 q^{4}-52 q^{5}-170 q^{6}$ & \multirow{2}{*}{4.0597} \\
 & & $-216 q^{7}-34 q^{8}+516 q^{9}+1466 q^{10}+O\left(q^{11}\right) $ & \\
\hline
\multirow{2}{*}{$(1,1)$} & \multirow{2}{*}{m129} & $1-3 q+q^{2}+16 q^{3}+16 q^{4}-11 q^{5}-87 q^{6}$ & \multirow{2}{*}{3.6638} \\
 & & $-160 q^{7}-161 q^{8}+16 q^{9}+482 q^{10}+O\left(q^{11}\right)$ & \\
\hline
\multirow{2}{*}{$(-4,1)$} & \multirow{2}{*}{m125} & $1-3 q+q^{2}+16 q^{3}+16 q^{4}-11 q^{5}-87 q^{6}$ & \multirow{2}{*}{3.6638} \\
 & & $-160 q^{7}-161 q^{8}+16 q^{9}+482 q^{10}+O\left(q^{11}\right) $ & \\
\hline
\multirow{2}{*}{$(-5,2)$} & \multirow{2}{*}{m202} & $1-4 q+4 q^{2}+24 q^{3}+11 q^{4}-52 q^{5}-170 q^{6}$ & \multirow{2}{*}{4.0597} \\
 & & $-216 q^{7}-34 q^{8}+516 q^{9}+1466 q^{10}+O\left(q^{11}\right)$ & \\
\hline
\multirow{2}{*}{$(-1,3)$} & \multirow{2}{*}{m391} & $1-4 q+7 q^{2}+22 q^{3}-17 q^{4}-102 q^{5}-159 q^{6}$ & \multirow{2}{*}{4.8511} \\
 & & $+33 q^{7}+607 q^{8}+1376 q^{9}+1682 q^{10}+O\left(q^{11}\right) $ & \\
\hline
\multirow{2}{*}{$(2,1)$} & \multirow{2}{*}{m295} & $1-2 q+2 q^{2}+12 q^{3}-2 q^{4}-43 q^{5} -94 q^{6}$ & \multirow{2}{*}{4.4153} \\
 & &$ -63 q^{7}+154 q^{8}+567 q^{9}+1054 q^{10}+O\left(q^{11}\right)$ & \\
\hline
\multirow{2}{*}{$(-2,3)$} & \multirow{2}{*}{m359} &$ 1-3 q+4 q^{2}+17 q^{3}-8 q^{4}-69 q^{5}-121 q^{6}$ & \multirow{2}{*}{4.7254} \\
 & & $-12 q^{7}+380 q^{8}+959 q^{9}+1342 q^{10}+O\left(q^{11}\right)$ & \\
\hline
\multirow{2}{*}{$(1,2)$} & \multirow{2}{*}{m367} &$ 1-4 q+7 q^{2}+23 q^{3}-16 q^{4}-106 q^{5}-176 q^{6}$ & \multirow{2}{*}{4.7494} \\
 & & $+584 q^{8}+1440 q^{9}+1963 q^{10}+O\left(q^{11}\right)$ & \\
\hline
\multirow{2}{*}{$(-1,4)$} & \multirow{2}{*}{s602} & $1-4 q+7 q^{2}+22 q^{3}-14 q^{4}-100 q^{5}-164 q^{6}$ & \multirow{2}{*}{5.0826} \\
 & & $+6 q^{7}+542 q^{8}+1274 q^{9}+1607 q^{10}+O\left(q^{11}\right)$ & \\
\hline
\multirow{2}{*}{$(-2,5)$} & \multirow{2}{*}{s661} & $1-3 q+7 q^{2}+18 q^{3}-28 q^{4}-111 q^{5} -128 q^{6}$ & \multirow{2}{*}{5.1549} \\
 & & $+146 q^{7}+794 q^{8}+1450 q^{9}+1241 q^{10}+O\left(q^{11}\right)$ & \\
\hline
\end{tabular}
\caption{Filling slope for the first cusp of the magic manifold, corresponding census manifolds as verified in~\cite{thompson2025triangulations}, their index (up to $q$-degree $10$) and volume.}
\label{tab:magic1}
\end{table}

We further computed the result of the Gang-Yonekura formula for all exceptional single fillings on $N$ (\textit{cf.}~~\cite{martelli2006dehn}), and listed them Table~\ref{tab:magic2}. 
\begin{table}[!htbp]
\begin{tabular}{|c|c|c|}
\hline
Slope & Manifold & Index \\
\hline
$(1,0)$ & $T \times [0,1]$ & $O\left(q^{11}\right) $ \\
\hline
$(-3,1)$ & $(A,(2,1))\cup_{\begin{pmatrix}
0 & 1\\ 1 & 0
\end{pmatrix}} (A,(2,1))$ & \texttt{undefined} \\
\hline
$(-2,1)$ & $(P,(3,1))$ & $1+O\left(q^{11}\right)$ \\
\hline
$(-1,1)$ & $(P,(2,1))$ & $1 + O\left(q^{11}\right)$ \\
\hline
$(0,1)$ & $(D,(2,1),(3,1)) \cup_{\begin{pmatrix}
1 & 1\\ -1 & 0
\end{pmatrix}} (P\times S^1)$ &\texttt{undefined} \\
\hline
\end{tabular}
\caption{Exceptional surgeries on the magic manifold, corresponding manifolds, and their index (up to $q$-degree $10$). Entries whose index is \texttt{undefined} have non-absolutely convergent sums from the Gang-Yonekura formula.}
\label{tab:magic2}
\end{table}

%&&&&&&&&&&&&&&&&&&&&&&&&&&&&&&&&&&&&&&&&&&&&&&&&&&&&&&&&&&&&&&&&&&&&&&&&&&&&&
%&&&&&&&&&&&&&&&&&&&&&&&&&&&&&&&&&&&&&&&&&&&&&&&&&&&&&&&&&&&&&&&&&&&&&&&&&&&&&

\section{Closed manifolds}\label{sec:closed}

A rigorous definition of the $3$D index of closed manifolds does not currently exists. However, it is still possible to improperly apply the Gang-Yonekura formula to a singly-cusped manifold, and interpret the resulting $q$-series as a definition of the index for the filled manifold. While this approach lacks a formal foundation, our computations suggest that the resulting series exhibit consistent structural patterns.

This extension to closed manifolds was first proposed by Gang, who conjectured in~\cite{gang2018quantum} that the following identities hold:
\begin{equation}\label{conj:gang_closed}
\ind_M(\alpha) = \begin{cases}
1 + P(q) & \text{with } P \in q\Z[\![q]\!] \text{ if $M(\alpha)$ is hyperbolic}\\
0,1, \text{ divergent} & \text{ if $M$ is not hyperbolic}.
\end{cases}
\end{equation}

In the light of the computations presented in this section, we amend this conjecture in a few ways in~\eqref{conj:new_closed} below.\\~\\

%&&&&&&&&&&&&&&&&&&&&&&&&&&&&&&&&&&&&&&&&&&&&&&&&&&&&&&&&&&&&&&&&&&&&&&&&&&&&&
%&&&&&&&&&&&&&&&&&&&&&&&&&&&&&&&&&&&&&&&&&&&&&&&&&&&&&&&&&&&&&&&&&&&&&&&&&&&&&

\subsection{Dehn filling the trefoil knot}\label{sec:filling_trefoil}~\\

Using~Section~\ref{ssec:alternating_torus} we can explicitly compute the value of the Gang-Yonekura formula for certain closed manifolds, namely the fillings on the trefoil knot complement. Recall from~Proposition~\ref{prop:index_of_torus} that the index of $T_{3,2}$, with boundary condition $\gamma = (x,y)$ is the delta function 
\begin{equation}
\label{trefoil_index}
\ind_M^\gamma(q) = \delta_{i(\gamma,\beta),0},
\end{equation}
where $\beta$ is the curve representing the homology class $-6\mu + \lambda$, with respect to the standard basis for $\hh_1(\partial M;\Z)$ given by the meridian $\mu$ and  the Seifert longitude $\lambda$.

Let $\alpha$ a simple closed curve on $\bd M$.
Then we can prove the following: 
\begin{prop}\label{prop:surgery_on_torus}
Let $M_{T_{3,2}}$ denote the complement of the trefoil knot, and  $\tri_{(3,2)}$ be the triangulation of $M_{T_{3,2}}$, as  in Proposition~\ref{prop:index_of_torus}. Then 
\begin{equation}\ind_{M_{T_{3,2}}(\alpha)}(q) = 
\begin{cases} 
\text{undefined, if } \alpha=\pm \beta,\\
  0, \text{ if } i(\alpha,\beta)=\pm1, \\
  2, \text{ if }  i(\alpha,\beta)=\pm2, \\  
  1,   \text{ if }  |i(\alpha,\beta)| \ge 3 .
\end{cases}    
\end{equation}
\end{prop}
This should be compared with the topology of the Dehn filled manifolds, as determined by Moser~\cite{moser1971elementary} for the torus knot $T_{3,2}$:

\begin{itemize}
\item $M_{T_{3,2}}(\alpha)$ is the connected sum of two lens spaces if $\alpha=\pm\beta$,
\item $M_{T_{3,2}}(\alpha)$ is a lens space if $i(\alpha,\beta)=\pm1$,
\item $M_{T_{3,2}}(\alpha)$ is Seifert fibre space over base $S^2(2,3,2)$ if  $i(\alpha,\beta)=\pm2$,
\item $M_{T_{3,2}}(\alpha)$ is Seifert fibre space over base $S^2(2,3,k)$ if $|i(\alpha,\beta)| =k \ge 3$.
\end{itemize}

\begin{proof}
By equation~\eqref{trefoil_index}, the only terms contributing to~\eqref{eqn:GYformula} are of the form $\gamma=k \beta$, where $k \in \Z$ and $i(\alpha,\gamma)=0,+2,-2$. Since 
$ \ind_\tri(\gamma)=1$ for these values of $\gamma$, we have
\begin{equation}
\ind_{M_{T_{3,2}}(\alpha)} =\frac{1}{2}  \sum_{\gamma \in \Z \beta} 
(-1)^{ i(\alpha^*,\gamma)}
\left( \delta_{i(\alpha,\gamma),0}\left(q^{\frac{1}{2} i(\alpha^*,\gamma)}+q^{-\frac{1}{2}i(\alpha^*,\gamma)}\right)  
- \delta_{i(\alpha,\gamma),2}-\delta_{i(\alpha,\gamma),-2}
\right)
\end{equation}
Now  $\gamma=k\beta$ (with $k\in\Z$) contributes to the sum if 
\begin{equation}\label{keqn}
i(\alpha,\gamma)=k \,i(\alpha,\beta)=0 \text{ or } \pm2.
\end{equation}
We consider 4 cases:
\begin{enumerate}
\item[(0)] If $i(\alpha,\beta)=0$ then each $k \in \Z$ contributes. Since $\alpha$ gives a primitive element of $\hh_1(\bd M_{T_{3,2}};\Z)$ and  $\ind_{M_{T_{3,2}}(\alpha)} =\ind_{M_{T_{3,2}}(-\alpha)}$ 
we may assume $\alpha=\beta$. Then $i(\alpha^*,\gamma)=k \,i(\beta^*,\beta)=-k$ and $i(\alpha,\gamma)=k \, i(\alpha,\beta)=0$, so we have 
$$
\ind_{M_{T_{3,2}}(\alpha)} =\frac{1}{2}  \sum_{k \in \Z} 
(-1)^{k} \left(q^{\frac{k}{2}}+q^{-\frac{k}{2}}\right) . 
$$
Hence the sum is not well-defined (since it has infinitely many terms with negative exponents).
\item[(1)] If $i(\alpha,\beta)=\pm 1$, 
then equation~\eqref{keqn} shows that  only $k=0,2,-2$ contribute. Now we may assume that
$i(\alpha,\beta)=-1$ so $\alpha = c \beta + \beta^*$ where $c \in \Z$. Then we can take $\alpha^*=\beta$, so
$i(\alpha^*,\gamma)=k \, i(\beta,\beta)=0$ and $i(\alpha,\gamma)=-k$.
Hence
\begin{align*}
\ind_{M_{T_{3,2}}(\alpha)} &=\frac{1}{2}  \sum_{k=-2,0,2} 
(-1)^{ 0}\left( \delta_{-k,0}\left(q^{\frac{1}{2} 0}+q^{-\frac{1}{2}0}\right)  
- \delta_{-k,2}-\delta_{-k,-2}
\right) \\
&= \frac{1}{2}(2-1-1)=0.
\end{align*}
\item[(2)]  If $i(\alpha,\beta)=\pm 2$, then equation~\eqref{keqn} shows that  only $k=0,1,-1$ contribute.
Now we may assume that
$i(\alpha,\beta)=-2$ so $\alpha = (2c-1) \beta + 2\beta^*$ where $c \in \Z$. 
Then we can take $\alpha^*=c \beta+\beta^*$ so $i(\alpha^*,\gamma)=-k$ and $i(\alpha,\gamma)=-2k$.
Hence
\begin{align*}
\ind_{M_{T_{3,2}}(\alpha)} &=\frac{1}{2}  \sum_{k=-1,0,1} 
(-1)^{-k}\left( \delta_{-2k,0}\left(q^{\frac{1}{2} k}+q^{-\frac{1}{2}k}\right)  
- \delta_{-2k,+2}-\delta_{-2k,-2}
\right) \\
&= \frac{1}{2}(2+1+1)=2.
\end{align*}
\item[(3)]  If $|i(\alpha,\beta)|\ge 3$, then equation~\eqref{keqn} shows that  only $k=0$ contributes to the index sum, hence
$$\ind_{M(\alpha)} =1.$$
\end{enumerate}
\end{proof}

%&&&&&&&&&&&&&&&&&&&&&&&&&&&&&&&&&&&&&&&&&&&&&&&&&&&&&&&&&&&&&&&&&&&&&&&&&&&&&
%&&&&&&&&&&&&&&&&&&&&&&&&&&&&&&&&&&&&&&&&&&&&&&&&&&&&&&&&&&&&&&&&&&&&&&&&&&&&&

\subsection{Figure eight knot}\label{sec:filling_fig8}~\\

The next simplest case is the complement of the figure eight knot $4_1$, with its canonical triangulation $\tri$, \textit{i.e.}~m004 in SnapPy notation. Besides providing extensive computations below, we will prove one exact result, using the $q$-hypergeometric approach from Section~\ref{sec:proof}.

In general, it would be interesting to prove that, for knot complements, the $(1,0)$-filling yields a trivial $3$D index. In this instance, proving that the Gang-Yonekura formula for the index of $S^3 = \text{meridional Dehn filling on }m004$ is zero is equivalent to 
\begin{equation}\label{eqn:surgery_on_4_1}
0 =  \sum_{x \in \Z}(q^x + q^{-x}) \ind_{m004}^{2x \mu}(q) - \ind_{m004}^{2x\mu + \lambda}(q) - \ind_{m004}^{2x\mu  -\lambda}(q ).
\end{equation}

The index of m004 in our notation is given by
$$\ind_{m004}^{2x \mu +y \lambda }(q) = \sum_{k \in \Z} \itet (k-x,k) \itet (k+y,k-x+y),$$
for example, using Example 4.1 and Remark 8.4 of \cite{garoufalidis20163d}.

Now, it is possible to rewrite this in terms of the $\varphi_r$ generating function defined in~
\eqref{eqn:seriesindexshifted}.
The index is the coefficient of $z^0$ in the product $z^y \varphi_x (z) \varphi_{-x} (z)$:
\begin{equation}\label{eqn:fig8_hyper}
\ind_{m004}^{2x \mu +y \lambda} (q) = \left[z^y \varphi_x (z) \varphi_{-x} (z) \right]_{z^0}.
\end{equation}

This implies that $\varphi_x (z) \varphi_{-x} (z)$ is the generating function for the index of m004 for a fixed meridional coefficient, summed over all multiples of the longitude, indexed by powers of $z$. More explicitly:
\begin{equation}\label{eqn:generating_4_1}
\varphi_x (z) \varphi_{-x} (z) = \sum_{k \in \Z} \ind^{2x \mu +k \lambda}_{m004} (q) z^{-k}.
\end{equation}

With this approach, we can give the first example of an \textit{exact} computation 
of the Gang-Yonekura formula for Dehn filling on a cusped hyperbolic manifold:
\begin{thm}\label{thm:4_1_exact}
The $3$D index of the longitudinal surgery on the figure eight knot is $1$. In other words, $$\ind_{m004(0,1)}(q) \equiv 1.$$
\end{thm}
\begin{proof}
The Gang-Yonekura formula applied to longitudinal surgery on $4_1$ can be re-written, using the compact form~\eqref{eqn:compact_version}, as: 
$$\ind_{m004(0,1)}(q) = \sum_{k \in \Z} \left( \left(-q^\frac{1}{2}\right)^{k} \ind_{m004}^{k \lambda}(q) - (-1)^k\ind_{m004}^{2\mu+k \lambda}(q) \right) .$$
In turn, using equation~\eqref{eqn:generating_4_1}, this is equivalent to 
\begin{equation}\label{eqn:functional_m004}
\ind_{m004(0,1)}(q) = \varphi_0 \left(-q^{\frac12}\right)^2  - \varphi_1 (-1) \varphi_{-1} (-1) .
\end{equation}
Then we can use equation~\eqref{eqn:phi_expression} and Corollary~\ref{cor:3phi3_and_trig} to conclude, or note that equation~\eqref{eqn:functional_m004} is just the generating function from Theorem~\ref{thm:trigonometry}, evaluated in $z = -1, \,r=0$.
\end{proof}
Note that this computation is consistent with Conjecture~\ref{conj:new_closed}, and the certified computation from Table~\ref{tab:n1_fig8}, since $m004(0,1)$ is topologically the torus bundle with monodromy $\begin{pmatrix}
2 & 1\\ 1 &1\end{pmatrix}$.\\

We collect here the some computations for surgeries on the figure eight knot, again using SnapPy's triangulation m004 of the knot's complement as a starting point. It is interesting to note that the index takes a very special form for  all exceptional surgeries of $4_1$, as predicted by~\cite{gang2018quantum}.   
\begin{table}[!htbp]
\begin{tabular}{|c|c|}
\hline
Slope & $3$D index\\
\hline
\multirow{2}{*}{$(1,0)$} & $ O\left(q^{11}\right)$ \\
 & \\
\hline
\multirow{2}{*}{$(0,1)$} & $1 + O\left(q^{11}\right)$ \\
 & \\
\hline
\multirow{2}{*}{$(1,1)$} & $1 + O\left(q^{11}\right)$ \\
 & \\
\hline
\multirow{2}{*}{$(2,1)$} & $1 + O\left(q^{11}\right)$ \\
 & \\
\hline
\multirow{2}{*}{$(3,1)$} & $1 + O\left(q^{11}\right)$ \\
 & \\
\hline
\multirow{2}{*}{$(4,1)$} & $\infty - 2q^2 - 2q^3 - 4q^4 - 4q^5  $ \\
 & $- 6q^6 - 4q^7 - 8q^8 - 6q^9 - 8q^{10} + O\left(q^{11}\right)$ \\
\hline
\multirow{2}{*}{$(5,1)$} & $1 - q - 2 q^2 - q^3 - q^4 + q^5 + 2 q^6 + 7 q^7$ \\
 & $+ 8 q^8 + 12 q^9 + 14 q^{10} + O\left(q^{11}\right)$ \\
\hline
\multirow{2}{*}{$(6,1)$} & $1 - t - q^{\frac32} - q^2 - q^\frac52 + q^\frac92 + 3 q^5 + 4 q^{\frac{11}{2}} + 4 q^6+ 6 q^{\frac{13}{2}}$ \\
 & $ + 8 q^7 + 7 q^\frac{15}2 + 7 q^8 + 9 q^\frac{17}2 + 8 q^9 + 6 q^\frac{19}2 + 4 q^{10} + O\left(q^{\frac{21}2}\right)$ \\
\hline
\multirow{2}{*}{$(7,1)$} & $1 - q^2 + q^4 + 3 q^5 + 3 q^6 + 6 q^7$ \\
 & $+ 4 q^8 + 2 q^9 - 4 q^{10} + O\left(q^{11}\right)$ \\
\hline
\multirow{2}{*}{$(8,1)$} & $1 - q^2 + 2 q^4 + 5 q^5 + 6 q^6 + 8 q^7$ \\
 & $+ 4 q^8 - 2 q^9 - 14 q^{10} + O\left(q^{11}\right)$ \\
\hline
\multirow{2}{*}{$(9,1)$} & $1 - q^2 + 2 q^5 + 2 q^6 + 4 q^7 + q^8$ \\
  & $- q^9 - 6 q^{10} + O\left(q^{11}\right)$ \\
\hline
\multirow{2}{*}{$(10,1)$} & $1 - q^2 + q^\frac52 + q^\frac72 + q^\frac92 + 2 q^5 + 2 q^\frac{11}2 + 2 q^6 + q^\frac{13}2 $ \\
  & $+ 4 q^7- q^\frac{15}2 + 2 q^8 - 5 q^\frac{17}2 - 9 q^\frac{19}2 - 4 q^{10} + O\left(q^{\frac{21}2}\right)$ \\
\hline
\end{tabular}
\caption{Integer (and $\infty$) slope surgeries on m004, and their $3$D index.  The first $6$ are the exceptional surgeries for $4_1$ (up to reflection).}
\label{tab:n1_fig8}
\end{table}

There are a few observation to make on the computations in Table~\ref{tab:n1_fig8}; all of them, except for the $(4,1)$ one, were performed with our `certified' code. The slope $(4,1)$ case is special amongst the exceptional surgeries; the filled manifold contains an embedded Klein bottle. Each $n$ multiple of this surface contributes to the index as $1 + P_n(q)$, where $P_n(q) \in \Z[\![q^{\frac12}]\!]$ has a minimal degree monotonically increasing with $n$. In particular this implies that -- with the exception of the constant term -- the $3$D index of the filled manifold stabilises to a well defined $q$-series. 
\begin{rmk}
This `constant term blow-up' phenomena is not limited to this specific case, but appears to be related to the presence of embedded non-peripheral tori or Klein bottles. A similar behaviour will be observed in Section~\ref{sec:toroidal}.
\end{rmk}

The second phenomena we'd like to remark on is that experimentally all $(4k+2)$-sloped surgeries on $4_1$ produce $q$-series with half-integer exponents; note that, by an improper application of Theorem~\ref{thm:index_asymptotics}, we expect the minimal degree where an half-integer appears to increase with $k$.
We include for completeness the first $100$ terms of the $3$D index for the asymptotic integer fillings: 
\begin{align*}
&\lim_{n\to +\infty}\ind_{m004(n,1)}(q)  = 1 - q^2 + 2q^5 + 2q^6 + 4q^7 + 2q^8 - 4q^{10} - 10q^{11} - 17q^{12} - 24q^{13}\\
&- 25q^{14} - 28q^{15} - 15q^{16} + 35q^{18} + 70q^{19} + 129q^{20} + 178q^{21} + 242q^{22} + 278q^{23}\\
&+ 304q^{24} + 274q^{25} + 206q^{26} + 48q^{27} - 186q^{28} - 520q^{29} - 948q^{30} - 1462q^{31} - 2042q^{32}\\
&- 2628q^{33} - 3200q^{34} - 3618q^{35} - 3856q^{36} - 3708q^{37} - 3154q^{38} - 1920q^{39} - 48q^{40}\\
&+ 2780q^{41} + 6424q^{42} + 11152q^{43} 
+ 16663q^{44} + 23122q^{45} + 29957q^{46} + 37142q^{47}\\
&+ 43801q^{48} + 49610q^{49} + 53354q^{50} + 54430q^{51} + 51299q^{52} + 43104q^{53} 
+ 28205q^{54}\\
&+ 5636q^{55} - 26077q^{56} - 67642q^{57} - 119920q^{58} - 182950q^{59} - 256297q^{60} - 338780q^{61}\\
&- 428018q^{62} - 520922q^{63} 
- 612466q^{64} - 697218q^{65} - 766899q^{66} - 813444q^{67}\\
&- 825271q^{68} - 791876q^{69} - 698796q^{70} - 534168q^{71} - 281964q^{72} + 69106q^{73} 
+ 534061q^{74}\\
&+ 1120540q^{75} + 1838929q^{76} + 2688730q^{77} + 3670844q^{78} + 4771390q^{79} + 5975816q^{80}\\
&+ 7250856q^{81} + 8560557q^{82} + 9846340q^{83} + 11045311q^{84} + 12069286q^{85} + 12824816q^{86}\\
&+ 13192476q^{87} + 13048660q^{88} + 12245696q^{89} + 10635543q^{90} 
+ 8052452q^{91} + 4337295q^{92}\\ 
&- 675356q^{93} - 7132498q^{94} - 15170436q^{95} - 24890707q^{96} - 36360340q^{97} - 49591008q^{98} \\
&- 64530416q^{99} - 81046208q^{100} +O\left(q^{101}\right)
\end{align*}
We note that this computation, using our certified code took $\sim 58$s on a ThinkPad with $32$Gb of memory, without any parallelisation.

\begin{table}[!htbp]
\begin{tabular}{|c|c|}
\hline
\multirow{2}{*}{$(1,2)$} & $1 - 2 q - 3 q^{2} + 3 q^{4} + 10 q^5 + 14 q^6$ \\
& $+ 22 q^7 + 20 q^8 + 14 q^9 - 2 q^{10} + O\left(q^{11}\right)$ \\
\hline
\multirow{2}{*}{$(1,3)$} & $1 - 2 q - 3 q^{2} + q^{3} + 6 q^{4} + 13 q^5 + 16 q^6$ \\
& $ + 17 q^7 + 4 q^8 - 20 q^9 - 55 q^{10} + O\left(q^{11}\right)$ \\
\hline
\multirow{2}{*}{$(1,4)$} & $1 - 2 q - 3 q^{2} + q^{3} + 6 q^{4} + 13 q^{5}+16 q^6 $ \\
& $+ 17 q^7 + 4 q^8 - 20 q^9 - 54 q^{10} + O\left(q^{11}\right)$ \\
\hline
\end{tabular}
\caption{The $\frac{1}{n}$-sloped surgeries on m004, and their $3$D index. We included only these three, as the coefficients for the first ten $q$-degrees stabilise already for the $(1,4)$-surgery.}
\label{tab:1n_fig8}
\end{table}

%&&&&&&&&&&&&&&&&&&&&&&&&&&&&&&&&&&&&&&&&&&&&&&&&&&&&&&&&&&&&&&&&&&&&&&&&&&&&&
%&&&&&&&&&&&&&&&&&&&&&&&&&&&&&&&&&&&&&&&&&&&&&&&&&&&&&&&&&&&&&&&&&&&&&&&&&&&&&

\subsection{A toroidal example}\label{sec:toroidal}~\\

It is of course interesting to probe the index's behaviour on non-hyperbolic closed manifolds. Let $\tri$ 
be the triangulation with isosignature \texttt{cPcbbbdei} of the cusped manifold $P$ obtained by gluing two tetrahedra, as explained in~\cite[Ex.~11.4]{garoufalidis20163d}; this manifold was also considered in~\cite[Sec.~4]{garoufalidis2019meromorphic}, where its meromorphic $3$D index was computed. This manifold is toroidal, and the triangulation is \textit{not} $1$-efficient. 

Using Regina, we see that
$$\ind_\tri^{(x,y)} (q) = \sum_{k \in \Z}\itet(x,2y-k) \itet(-x,-k).$$
Therefore, by equation~\eqref{enq:generating_function_tetra}, 
\begin{equation}\label{eqn:index_toroidal}
\ind_\tri^{(x,y)} (q) = \left[ z^{2y}\frac{(q^{1-\frac{x}{2}} z)_\infty}{(q^{-\frac{x}{2}} z^{-1})_\infty} \frac{(q^{1+\frac{x}{2}} z^{-1})_\infty}{(q^{\frac{x}{2}} z)_\infty}\right]_{z^0} = \left[z^{2y} \frac{(q^{1-\frac{x}{2}} z)_{x-1}}{(q^{-\frac{x}{2}} z^{-1})_{x+1}}\right]_{z^0}.
\end{equation}

Since $(a;q^{-1})_n = (a;q)_n (-a)^n q^{-\binom{n}{2}}$, we see that 
$$z^{2y} \frac{(q^{1-\frac{x}{2}} z)_{x-1}}{(q^{-\frac{x}{2}} z^{-1})_{x+1}} = 
\frac{(-1)^{x+1} z^{x+2y+1}}{(1-q^{\frac{x}{2}}z)(1-q^{-\frac{x}{2}}z)}.$$ A simple residue computation provides (for $x \neq 0$) $$ \ind_\tri^{(x,y)} (q) = \frac{(-1)^{x}q^{\frac{|x||x+2y+1|}{2}}}{(1-q^{|x|})}$$ as the coefficient of $z^0$.

If we (improperly) apply equation~\eqref{eqn:GYformula} for non-zero fillings on the cusp in $P$, we observe an interesting phenomenon; the closed index of $(1,0)$-surgery appears to be constantly $0$. We display in Table~\ref{tab:tab1_toric} the values we get for the index on $(n,1)$-surgery on $\tri$, $n = 0, \ldots,10$, computed with the non-certified version of our code (\textit{cf.}~Section~\ref{sec:code}):

\begin{table}[!htbp]
\centering
\begin{tabular}{|c|c|}
\hline
\multirow{2}{*}{$(1,0)$} & $ O\left(q^{11}\right)$ \\
 &  \\
\hline
\multirow{2}{*}{$(0,1)$} & divergent, with infinitely many  \\
 &  terms with negative exponents\\
\hline
\multirow{2}{*}{$(1,1)$} & $1+O\left(q^{11}\right)$ \\
 &  \\
\hline
\multirow{2}{*}{$(2,1)$} & $\infty - 2 q - 4 q^2 - 4 q^3 - 6 q^4 - 4 q^5$ \\
 & $- 8 q^6 -4q^7 - 8q^8 - 6q^9 -8q^{10} +O\left(q^{11}\right)$ \\
\hline
\multirow{2}{*}{$(3,1)$} & $\infty - q - 3 q^2 - 3 q^3 - 5 q^4 - 4 q^5$ \\
 & $- 7 q^6 - 4 q^7 - 8 q^8 - 6 q^9 - 8 q^{10} + O\left(q^{11}\right)$ \\
\hline
\multirow{2}{*}{$(4,1)$} & $\infty - q - q^{\frac{3}2} - 3 q^2 - q^{\frac{5}2} - 3 q^3 - 2 q^{\frac{7}2}- 5 q^4 - 2 q^{\frac{9}2} - 3 q^5 - 2 q^{\frac{11}2} - 7 q^6$ \\
 & $ - 2 q^{\frac{13}2} - 3 q^7 - 3 q^{\frac{15}2} - 7 q^8 - 2 q^{\frac{17}2} - 4 q^9 - 2 q^{\frac{19}2} - 7 q^{10}  + O\left(q^{\frac{21}2}\right)$ \\
\hline
\multirow{2}{*}{$(5,1)$} & $\infty - q - 3 q^2 - 3 q^3 - 5 q^4 - 3 q^5$ \\
 & $- 7 q^6 - 3 q^7 - 7 q^8 - 5 q^9 - 7 q^{10} + O\left(q^{11}\right)$ \\
\hline
\multirow{2}{*}{$(6,1)$} & $\infty - q - 3 q^2 - 3 q^3 - 4 q^4 - 3 q^5$ \\
 & $- 6 q^6 - 2 q^7 - 5 q^8 - 5 q^9 - 5 q^{10} + O\left(q^{11}\right)$ \\
\hline
\multirow{2}{*}{$(7,1)$} & $\infty - q - 3 q^2 - 3 q^3 - 5 q^4 - 3 q^5$ \\
 & $- 7 q^6 - 3 q^7 - 7 q^8 - 5 q^9 - 7 q^{10} + O\left(q^{11}\right)$ \\
\hline
\multirow{2}{*}{$(8,1)$} & $\infty - q - 3 q^2 - 3 q^3 - 5 q^4 - 3 q^5 - 7 q^6 - 3 q^7 - q^{\frac{15}2}$ \\
 & $- 7 q^8 - 5 q^9 - 7 q^{10} + 3 q^{\frac{21}2} + O\left(q^{11}\right)$ \\
\hline
\multirow{2}{*}{$(9,1)$} & $\infty - q - 3 q^2 - 3 q^3 - 5 q^4 - 3 q^5$ \\
 & $- 7 q^6 - 3 q^7 - 7 q^8 - 5 q^9 - 7 q^{10} + O\left(q^{11}\right)$ \\
\hline
\multirow{2}{*}{$(10,1)$} & $\infty - q - 3 q^2 - 3 q^3 - 5 q^4 - 3 q^5$ \\
 & $- 7 q^6 - 3 q^7 - 7 q^8 - 5 q^9 - 7 q^{10} + O\left(q^{11}\right)$ \\
\hline
\end{tabular}
\caption{Integer slope surgeries on the toroidal triangulation \texttt{cPcbbbdei}. $(0,1)$-surgery produces a manifold whose index has infinitely many terms with negative exponents.   }
\label{tab:tab1_toric}
\end{table}~\\
We list the indices for $(-n,1)$-surgery on $\tri$, $n = 1, \ldots,10$ in Table~\ref{tab:tab2_toric}. 
\begin{table}[!htbp]
\centering
\begin{tabular}{|c|c|}
\hline
\multirow{2}{*}{$(-1,1)$} & $ O\left(q^{11}\right)$ \\
 &  \\
\hline
\multirow{2}{*}{$(-2,1)$} & $\infty + O\left(q^{11}\right)$ \\
 &  \\
\hline
\multirow{2}{*}{$(-3,1)$} & $\infty - 2 q - 4 q^2 - 4 q^3 - 6 q^4 - 4 q^5$ \\
 & $- 8 q^6 - 4 q^7 - 8 q^8 - 6 q^9 - 8 q^{10} + O\left(q^{11}\right)$ \\
\hline
\multirow{2}{*}{$(-4,1)$} & $\infty - q^\frac12 - q - 2 q^{\frac{3}2} - 3 q^2 - 2 q^{\frac{5}2} - 2 q^3 - 2 q^{\frac{7}2}
- 5 q^4 - 3 q^{\frac{9}2} - 2 q^5 - 2 q^{\frac{11}2}$ \\ &$  - 6 q^6 - 2 q^{\frac{13}2} - 2 q^7 - 4 q^{\frac{15}2} - 7 q^8 - 2 q^{\frac{17}2} - 3 q^9 - 2 q^{\frac{19}2} - 6 q^{10} - 2 q^{\frac{21}2} + O\left(q^{11}\right)$ \\
\hline
\multirow{2}{*}{$(-5,1)$} & $\infty - q - 3 q^2 - 3 q^3 - 5 q^4 - 3 q^5$ \\
 & $- 8 q^6 - 3 q^7 - 7 q^8 - 6 q^9 - 8 q^{10} + O\left(q^{11}\right)$ \\
\hline
\multirow{2}{*}{$(-6,1)$} & $\infty - q - 2 q^2 - 3 q^3 - 3 q^4 - 2 q^5$ \\
 & $- 6 q^6 - 2 q^7 -4 q^{8} -5q^{9} -4q^{10}+ O\left(q^{11}\right)$ \\
\hline
\multirow{2}{*}{$(-7,1)$} & $\infty - q - 3 q^2 - 3 q^3 - 5 q^4 - 3 q^5$ \\
 & $- 7 q^6 - 3 q^7 - 7 q^8 - 5 q^9 - 7 q^{10} + O\left(q^{11}\right)$ \\
\hline
\multirow{2}{*}{$(-8,1)$} & $\infty - q - 3 q^2 - 3 q^3 - 5 q^4 - q^{9/2} - 3 q^5 - 7 q^6 - 3 q^7$ \\
 & $- 2 q^{15/2} - 7 q^8 - q^{17/2} - 5 q^9 - 7 q^{10} + O\left(q^{11}\right)$ \\
\hline
\end{tabular}
\caption{Negative integer slope fillings on \texttt{cPcbbbdei}; we only included values until $-8$, as the coefficients up to $q$-degree $10$ stabilise for larger values. The entry for $(-2,1)$ means that the computed index consists of an infinite constant term, and the rest of the $q$-series is trivial (up to the chosen threshold).}
\label{tab:tab2_toric}
\end{table}
Finally, we note that the value for $(\pm 1,n)$ surgeries $(n >1)$, appears to give an index uniformly equal to $1$, but we were not able to prove this.

In most of the examples above the constant term blows up to infinity, while the rest of the index stabilises to a well-defined $q$-series. It is possible to give a nice geometric explanation of this phenomenon, hinging on the index's surface interpretation from Section~\ref{ssec:surface_index}. There are only two fundamental surfaces in $\tri$ which are not tetrahedral solutions (or peripheral tori): one embedded torus $S_1 = [0,0,1,0,0,1]$, and one immersed surface $S_2 = [1, 1, 0, 1, 1, 0]$ with $\chi = -2$. Crucially, the sum of these two surfaces is a tetrahedral solution. Therefore, in the computation of the index, only these two surfaces contribute non-trivially; further, any linear combination $n_1 S_1 + n_2 S_2$ with $n_1 \ge n_2$, provides an equivalent contribution as $(n_1-n_2)S_1$ (the case with $n_1 \le n_2$ is symmetric).

This implies that $\ind_\tri^{tot} (q) $ is simply a sum over $k S_1$ and $k S_2$ for $k \ge 0$, without `mixed' terms.
Now, the contribution of $kS_1$ to the index is given by $J_\Delta (0,0,k)^2 = 1 - 2q^k+\ldots$, while the minimal degree contribution for $kS_2$ is $2k(k-1)$. Therefore, the splitting torus contributes with infinitely many $1$s to the constant term, while all other coefficients are given by a finite sum.

\begin{rmk}
Proving that $(1,1)$-surgery on $\tri$ gives an index constantly equal to $1$ is equivalent to proving that
$$\frac12 \sum_{k \in \Z} \itet (k,0) \left( 2\itet (k,0) - \itet (k+2,0) - \itet (k-2,0)\right) \equiv 1,$$
or 
$$ \sum_{k \in \Z} q^{-k}\itet (k,0) \left( q\itet (k-1,0) +  q^{-1}\itet (k+1,0)-\itet (k,0) \right) \equiv 1.$$
We obtain the second expression by repeatedly applying equation~\eqref{eqn:index_3term_adjacent}. Now, by the quadratic identity~\ref{eqn:quadratic_identity}, $$\sum_{k \in \Z} q^{-k}\itet (k,0) \itet (k,0) \equiv 1$$ so it remains to show that the remaining terms cancel out. This final step also follows directly from the quadratic identity.
\end{rmk}~\\

%&&&&&&&&&&&&&&&&&&&&&&&&&&&&&&&&&&&&&&&&&&&&&&&&&&&&&&&&&&&&&&&&&&&&&&&&&&&&&
%&&&&&&&&&&&&&&&&&&&&&&&&&&&&&&&&&&&&&&&&&&&&&&&&&&&&&&&&&&&&&&&&&&&&&&&&&&&&&

\subsection{Exceptional surgeries on m016}~\\

We conclude this section by adding computations of the $3$D index for the exceptional surgeries on m016, a.k.a.~the complement of the $(-2,3,7)$-pretzel knot or $12n_{242}$;  see Table~\ref{tab:m016exceptional}. 
It is interesting to compare these computations with Conjecture~\ref{conj:new_closed} and the discussion in Section~\ref{sec:code} below. 

\begin{table}[!htbp]
\centering
\begin{tabular}{|c|c|c|}
\hline
Slope & $3$D index & Manifold\\
\hline 
\multirow{2}{*}{$(0,1)$} & $ O\left(q^{11}\right)$ & \multirow{2}{*}{$L(18,5)$} \\
 &  & \\
\hline
\multirow{2}{*}{$(1,0)$} & $  O\left(q^{11}\right)$ & \multirow{2}{*}{$S^3$} \\
 &  & \\
\hline
\multirow{2}{*}{$(1,1)$} & $1 + O\left(q^{11}\right) $ & \multirow{2}{*}{SFS} \\
 &  & \\
\hline
\multirow{2}{*}{$(-1,1)$} & $O\left(q^{11}\right) $ & \multirow{2}{*}{$L(19,7)$} \\
 &  & \\
\hline
\multirow{2}{*}{$(-1,2)$} & $\infty - 2 q - 4 q^2 - 4 q^3 - 6 q^4 - 4 q^5 $ & \multirow{2}{*}{graph manifold} \\
 & $- 8 q^6 - 4 q^7 - 
 8 q^8 - 6 q^9 - 8 q^{10} +  O\left(q^{11}\right)$ & \\
\hline
\multirow{2}{*}{$(2,1)$} & $\infty -4 q - 6 q^2 - 4 q^3 - 8 q^4 - 4 q^5$ & \multirow{2}{*}{graph manifold} \\
 & $ -12 q^6 - 4 q^7 - 10 q^8 - 8 q^9 - 8 q^{10} + O\left(q^{11}\right)$ & \\
\hline
\multirow{2}{*}{$(-2,1)$} & $\infty-2 q - 4 q^2 - 4 q^3 - 6 q^4 - 4 q^5  $ & \multirow{2}{*}{graph manifold} \\
 & $- 8 q^6 - 4 q^7 - 8 q^8 - 
 6 q^9 - 8 q^{10} + O\left(q^{11}\right)$ & \\
\hline
\end{tabular}
\caption{Exceptional surgeries for m016. Slopes are determined from SnapPy's geometric basis. The last three entries have non-absolutely converging index. 
Here m016$(1,1)$ is the SFS $\left[S^2: \left(2,\frac32,-\frac53\right) \right]$,  m016$(2,1)$ is the graph manifold $[D: \left(2,2\right)] \cup_m [D: \left(2,3\right)]$, where $m$ is the matrix $
\begin{pmatrix}
1&1\\0&1\end{pmatrix}$, m016$(-1,2)$ is $[D: \left(2,3\right)] \cup_{m'} [D: \left(2,\frac32\right)]$, where $m'$ is the matrix $
\begin{pmatrix}
-1&1\\0&1\end{pmatrix}$, m016$(-2,1)$ is $[D: \left(2,2\right)] \cup_{m''} [D: \left(2,\frac32\right)]$, where $m''$ is the matrix $
\begin{pmatrix}
0&1\\1&1\end{pmatrix}$.
}
\label{tab:m016exceptional}
\end{table}

%&&&&&&&&&&&&&&&&&&&&&&&&&&&&&&&&&&&&&&&&&&&&&&&&&&&&&&&&&&&&&&&&&&&&&&&&&&&&&
%&&&&&&&&&&&&&&&&&&&&&&&&&&&&&&&&&&&&&&&&&&&&&&&&&&&&&&&&&&&&&&&&&&&&&&&&&&&&&

\section{Computational aspects}\label{sec:code}

Before going into the description of how our code works, we showcase some of the computational results obtained with it, complementing those from previous sections. These provide further independent validation of the main results in this paper, and form the basis for many interesting observations and conjectures.
We refer to Section~\ref{sec:availability} for the complete dataset and code.\\

We used our `certified' code to compute the value of the $3$D index for all $1$-cusped manifolds in the census of orientable cusped hyperbolic $3$-manifolds with fewer than $5$ tetrahedra, and boundary curves $(p,q)$ with $|p|,|q| \le 6$, with respect to SnapPy's geometric basis. We further applied the Gang-Yonekura formula for all $(p,q)$ filling on these manifolds, for $|p|,|q|\le 6$. We then cross-referenced the results with the list of equivalent surgery description of manifolds in the Hodgson-Weeks census of closed hyperbolic $3$-manifolds available at~\cite{neil_list}. The indices we computed did coincide in all cases. A sample of the results for the first three equivalent surgery descriptions is given below.
\begin{itemize}
\item m003(-3,1) = m003(2,1) = m006(2,1) = m029(-2,1) \\Volume = $0.942707$, homology $\Z/5 \oplus \Z/5$, length of shortest geodesic	$0.58460$, $3$D index $$ 1 - q -3 q^2 -3 q^3 -2 q^4  +3 q^6 + 9 q^7 + 13 q^8 + 19 q^9 + 23 q^{10}+O\left(q^{11}\right) $$
\item m003(-2,3) = m003(-1,3) = m004(-5,1) = m004(5,1) = m011(1,2) = m015(3,1) = m019(2,1)  \\
Volume = $0.98136$, homology $\Z/5$, length of shortest geodesic	$0.57808$, $3$D index $$ 1 - q -2 q^2 - q^3 - q^4  +q^5 + 2 q^6 + 7 q^7 + 8 q^8 + 12 q^9 + 14 q^{10} +O\left(q^{11}\right)$$
\item m007(3,1) = m010(-1,2) = m207(1,1) \\ Volume = $1.01494$, homology $\Z/3 \oplus \Z/6$, length of shortest geodesic	$0.83144$, $3$D index 
\begin{gather*}
1 - q -2 q^{\frac32} -2 q^2 -2 q^\frac52 -2 q^3 -2 q^\frac72 -2 q^4 + 2 q^5 + 2 q^\frac{11}{2} + 3 q^5\\ + 6 q^\frac{13}{2} + 9 q^7 + 10 q^\frac{15}2 + 11 q^8 + 14 q^\frac{17}2 + 17 q^9 + 18 q^\frac{19}2 + 17 q^{10} + O\left(q^{\frac{21}{2}} \right)
\end{gather*}
\end{itemize}
The analysis provided by these computations provides a strong empirical confirmation that the $3$D index is well-defined -- even for closed manifolds, independent of the choice of triangulations and surgery descriptions. 

We performed a very similar analysis for the $15$ $2$-cusped hyperbolic $3$-manifolds with minimal triangulations having at most $5$ tetrahedra. We looked at all the fillings on either cusp, again with coefficients $|p|,|q|\le 6$. We then compared the resulting manifolds using SnapPy's \texttt{identify()} function. Just as before, equivalent manifolds obtained via different surgery descriptions had the same index 
(up to our chosen threshold of $q^{\frac12}$-degree $20$).\\ 

The next set of data we present consists of the value of the $3$D index on certain closed non-hyperbolic manifolds with \emph{finite fundamental group}. These are obtained via filling on cusped census manifolds, and the database we are using was compiled by C.~Hodgson and J.~Weeks in 1990 using snappea and Magma.

As an example of the data analysed, the $3$D index of $(1,0)$-filling on m003 appears to be constantly $0$; note that m003$(1,0)$ is the lens space $L(10,3)$. If instead we perform $(-2,1)$-surgery on m003, the resulting manifold is a Seifert fibred space with base orbifold the $(2,3,3)$-sphere and fundamental group of order $120$; its $3$D index appears to be constantly $1$. 
For these two examples, our rigorous code implies absolute convergence of the Gang-Yonekura formula~\eqref{eqn:GYformula},  and that the index computations are as indicated above, up to $q^{\frac12}$-degree $20$. 
See Section~\ref{sec:availability} for a complete list of the $3$D indices we computed, and the code used to generate them.\\

We also computed further examples where applying equation~\eqref{eqn:GYformula} does not yield an absolutely convergent sum. The resulting index was computed using a version of our code which is based on the `edge coefficients' approach (as explained in Section~\ref{ssec:edge_index}) rather than the `surface' approach.  The first such example with finite $\pi_1$ is $(1,1)$ filling on m006, which topologically is the Seifert fibred space with base the $(2,2,n)$-sphere. Its index computed using the non-rigorous edge coefficients approach seems to be identically $2$.  \\

All of the examples of fillings with non-absolutely convergent  Gang-Yonekura sums that we have computed so far hint at a general `unreasonable effectiveness' of the Gang-Yonekura formula. Namely, even in cases where there is no absolute convergence, the result of applying equation~\eqref{eqn:GYformula} verbatim always appears to yield a consistent result. As partial support for this naive claim, we noticed that a pattern for \textit{closed} manifolds emerges quickly, which seems to partially align with and refine Gang's conjecture given in~\eqref{conj:gang_closed}. 
\begin{con}\label{conj:new_closed} Let $M$ be a closed $3$-manifold obtained by Dehn filling on a 1-cusped $3$-manifold $N$ with a $1$-efficient triangulation $\tri$. Then applying the Gang-Yonekura formula gives a well-defined valued for the index $\ind_M(q)$ of $M$, and 
\begin{equation}
\ind_{M}(q) = \begin{cases}
1 +  P(q), \,\,\, P(q) \in q^\frac12\Z[\![q^{\frac12}]\!] & \parbox[t]{.5\textwidth}{if $M$ is a hyperbolic manifold }\\
0 & \parbox[t]{.5\textwidth}{ if $M$ is a lens space }\\
1 & \parbox[t]{.5\textwidth}{ if $M$ is a SFS with base orbifold  the $(2,3,5)$, $(2,3,4)$ or the $(2,3,3)$-sphere, or  a SFS with hyperbolic base or a manifold with Solv geometry }\\
2 & \parbox[t]{.5\textwidth}{ if $M$ is a SFS with base orbifold the $(2,2,n)$-sphere}\\
\end{cases}
\end{equation}
Above we are including $S^3$ and $S^2 \times S^1$ as lens spaces.
\end{con}

\begin{rmk}
Despite looking quite daunting, equation~\eqref{eqn:GYformula} can be effectively used in many cases to cut the computational time of the total index. Indeed, if a high-complexity cusped manifold is obtained by filling a low-complexity one along a ``long'' surgery curve, then it is likely that the Gang-Yonekura formula will provide an answer faster than a direct computation. 
A striking example is provided by the computations in Section~\ref{ssec:whitehead}.
Furthermore, if the surgery curve is ``long enough'', then it is possible to use an effective version of Theorem~\ref{thm:index_asymptotics} to further cut down the time needed to compute low-degree terms. 
\end{rmk}

\begin{ex}
In this section and throughout the following ones we stress the necessity of considering half-integer multiples of edge, tetrahedral and peripheral solutions. These are needed to obtain consistent results: for example, $(-2,1)$-surgery on m412 gives m009. The presence of non-peripheral $\Z_2$ homology implies the possible presence of half-integer exponents of $q$. 
Indeed, applying our implementation of the Gang-Yonekura formula from Theorem~\ref{thm:main} yields (\textit{cf.}~Section~\ref{sec:code}) 
\begin{equation*}
\begin{gathered}
\ind_{m009}^0(q) = \\1 - q^{\frac12} - q - 2 q^{3/2} - q^2 + 2 q^{5/2} + 6 q^3 + 8 q^{7/2} +  9 q^4 + 11 q^{9/2} + 12 q^5 +\\ 6 q^{11/2} - 5 q^6 - 17 q^{13/2} -  34 q^7 - 57 q^{15/2} - 79 q^8 - 100 q^{17/2} - 118 q^9 -  124 q^{19/2} - 118 q^{10}
+\ldots
\end{gathered}
\end{equation*}
Without half integer multiples of solutions we would have recovered only half of $\ind_{m009}^0(q)$ after the filling (compare with the `even' and `odd' $3$D index from~\cite[Sec.~11.5]{garoufalidis20163d}).
\end{ex}~\\

%&&&&&&&&&&&&&&&&&&&&&&&&&&&&&&&&&&&&&&&&&&&&&&&&&&&&&&&&&&&&&&&&&&&&&&&&&&&&&
%&&&&&&&&&&&&&&&&&&&&&&&&&&&&&&&&&&&&&&&&&&&&&&&&&&&&&&&&&&&&&&&&&&&&&&&&&&&&&

\subsection{Computational methods}\label{ssec:computation}~\\

Generally, computing the $3$D index is a fairly challenging computational task, and explicit, exact computations can't usually be performed. 
So we aim to compute the $3$D index to some prescribed accuracy $q^{\frac12}$-degree $\le D$. See Appendix~\ref{computation:fig8} for some explicit examples of index computations for the figure eight knot complement.
The computational cost is roughly exponential in the number of tetrahedra; depending on the type of approach used (surfaces or edge weights), there are further complications that make a naive implementation unfeasible. \\

To carry out the surface approach described in Section~\ref{ssec:surface_index}, we proceed as in~\cite[Sec~8.1]{garoufalidis20163d} and enumerate the integer solutions to the $Q$-matching equations modulo integer linear combinations of tetrahedral solutions by choosing the unique \emph{minimal non-negative coset representative} $S$ with minimum quad coordinate zero in each tetrahedron. 

Let $\tri$ be a $1$-efficient ideal triangulation with $n$ tetrahedra of a cusped $3$-manifold $M$ with $r\ge 1$ cusps. 
From $\tri$, we first use \texttt{gluing\_equations()} command in SnapPy to compute the $(3n+2r)\times 3n$ matrix $G_\tri$ representing its gluing equations (containing SnapPy's choices for meridian and longitude for each cusp). Multiplying $G_\tri$ on the right by the  
diagonal matrix $C = \oplus^n C_0$, with 
\[C_0 = \begin{pmatrix}
0 & 1 & -1\\
-1 & 0 & 1\\
1 & -1 & 0
\end{pmatrix},\]
yields the matrix $L_\tri$ 
representing the $Q$-matching equations associated to the edges of $\tri$ 
and the leading trailing deformations corresponding to the meridians and longitudes, see~\cite{garoufalidis20163d} and equation~\eqref{Q-matching-eqns}. 

Next, we use the \texttt{HilbertBasis} command in Normaliz~\cite{Normaliz} to 
find a (finite) set of fundamental solutions 
$\{F_1, \ldots, F_m\}$ for the polyhedral cone of non-negative integer solutions to the linear $Q$-matching equations. 
For each solution $F_i$, we can easily compute its formal Euler characteristic as explained in 
Section~\ref{ssec:Q-normal}: we solve the linear system $G_M^T\cdot y = F_i$, and assign a weight $-2$ to all coefficients in $y$ corresponding to edge solutions, and weight $-1$ to the coefficient of each tetrahedral solution. Then the formal Euler characteristic $\chi(F_i)$ is just the sum of these weights.

We then enumerate all non-negative integer linear combinations of these fundamental solutions giving terms in the index sum of $q^{\frac12}$ degree at most $D$. 
If $$S=\sum_{i=1}^m n_i F_i \text{ with } n_i \in \Z \text{ and all } n_i \ge 0$$
is a minimal non-negative coset representative, then, by~\cite[Eqn.~(33)]{garoufalidis20163d}, its index contribution $I(S)$ has $q^{\frac12}$-degree 
$$\deg(S) = -\chi(S) + \delta(S),$$
where $\chi(S)$ is the formal Euler characteristic and $\delta(S)$ is the double curve count defined in Section~\ref{ssec:Q-normal}.

Now $$\chi(S)= \sum_{i=1}^m \chi(F_i) n_i$$ 
is a linear function of the coefficients $n_i$, and
$$\delta(S) = \sum_{i,j=1}^m \delta(F_i,F_j) n_i n_j $$
is a quadratic function of the coefficients $n_i$.

Note that all $d_{ij} =\delta(F_i,F_j)\ge 0$ for each $i$. Further, if $\tri$ is a \emph{$1$-efficient triangulation}  then for each $i$ we have either: $d_{ii}>0$ or $d_{ii}=0$ and $c_i=-\chi(F_i)> 0$. 
This gives a \emph{finite} set of linear combinations $\sum_{i=1}^m n_i F_i$ with 
$$ \sum_{i=1}^m c_i n_i + \sum_{i,j=1}^m d_{ij}  n_i n_j  \le D.$$ 
From these we remove any tetrahedral solutions to obtain minimal non-negative coset of representatives, remove any duplicates, then add up the index contributions to obtain the index up to $q^{\frac12}$-degree $D$.\\

In fact,  the process just described gives the \emph{total index} $\ind_{\tri}^{\text{tot}} (q) =\sum_\omega \ind_{\tri}^{\omega} (q)$, summed over all boundaries $\omega \in \hh_1(\bd M;\Z)$. There are also several refinements of interest, where we sum over normal surfaces $S$ with additional boundary conditions, such as:
\begin{enumerate}
\item $ \bd S = 0$, 
 \item $\bd S = \alpha$,  
 \item $\bd S \cdot \alpha =0$, 
\item $\bd S \cdot \alpha =\pm 2$. 
\end{enumerate}
The last two cases arise when we compute the index of Dehn filled manifolds using the Gang-Yonekura formula~\eqref{eqn:GYformula}.\\ 

By using the Neumann-Zagier symplectic relations, described in Section~\ref{sec:$3$-man}, we see that these conditions give extra linear equations to be satisfied.  In cases (2) and (4), we input the inhomogeneous linear system using the command \texttt{inhom\_equations} in Normaliz. The output consists of a Hilbert basis for the corresponding homogeneous cone together with an additional list of ``generators''.  The desired inhomogeneous cone is then the union of the translates of the homogeneous cone by these generators.  Then the process for index computation proceeds as described above. 
 
In the Gang-Yonekura formula, there is an additional complication that  increases computation time:  we also need to keep track of the \emph{decrease} in $q^{\frac12}$-degree by $|\gamma|$ for the terms in ~\eqref{eqn:GYformula} with factors $q^{-\frac{|\gamma|}{2}}$, which arise from curves $\gamma$ with $\gamma\cdot\alpha =0$.

The size of the computations needed to enumerate the low degree terms in the index sums can be reduced considerably using a built-in function in Normaliz. If we are only considering solutions with trivial boundary conditions we have a homogeneous polyhedral cone, which can be decomposed as a union of simplicial cones.  This procedure is implemented in Normaliz by the \texttt{UnimodularTriangulation} command. Then in enumerating low degree terms we can restrict to only computing linear combinations of vertices of each simplicial cone. If the boundary conditions are not homogeneous, then the set of solutions has the structure of an affine polyhedral cone. This can be decomposed into a  union of copies of the previous simplicial cones, each shifted by a single fundamental solution whose boundary is the desired one.

While there is efficient software built for this task~\cite{Normaliz, regina} the size of the resulting outputs is often too big to be of practical use.  To reduce the size of computations and memory required we can also divide up the problem into subproblems, where we specify a choice of quad in each tetrahedron whose corresponding quad coefficient is zero. This also has the advantage of reducing the number of tetrahedral solutions to be removed. However, for a triangulation with $n$ tetrahedra there are $3^n$ quad choices, so there is trade-off between the number of sub-problems and size of computation for each sub-problem.\\

We also remark that a naive approach using edge weights following Section~\ref{ssec:edge_index} has a big problem as well; what one could do in practice is to sum all possible weight functions up to a certain threshold (by considering a suitably large cube centred in the lattice $\Z^{n-r}$). 
This process is however not guaranteed to output the exact sum for the desired $3$D index up to a given degree unless we have rigorous degree bounds.
This is difficult due to the different growth behaviours exhibited along different \textit{rays} in the lattice $\Z^{n-r}$. Rays corresponding to multiples of embedded $Q$-normal surfaces will only provide contributions whose minimal exponent is linear in the distance from the origin (\textit{cf.}~Section~\ref{ssec:surface_index}). Another issue is choosing the best centre for the large cube; the best choice depends greatly on the boundary conditions imposed.  
The surface approach is more successful since the minimal degree of the contribution from each solution of the $Q$-matching equations is easily estimated. Thus, it allows for \textit{exact} computations of the index, up to a certain degree.\\ 

On the other hand, the edge weight approach is particular useful for obtaining explicit multi-sums representing the $3$D index, even when convergence is not guaranteed (as in some examples in Sections~\ref{sec:examples} and~\ref{sec:closed}).
We conclude this section by describing some of the computational details.~\\

The index with zero boundary conditions can be written 
\begin{equation}
\label{index0} \ind_{\tri}^{0} (q) = \sum_{[S] \in \mathcal{N}(\tri;\Z)/\mathbb{T}} \left(-q^{\frac12}\right)^{-\chi(S)} \jtet(S).
\end{equation}

We can express this more explicitly as follows. 
Denote by $\mathcal{N}(\tri;\Z)= \{ S \in \mathcal{Q}(\tri;\Z) : \bd S = 0\}$ the set of closed integer $Q$-normal classes. 
Recall that  $\mathcal{N}(\tri;\Z) \supseteq \Z E + \Z T$ is a finite index extension of the integer span of the edge solutions and tetrahedral solutions, and is the intersection of $\Q E + \Q T$ with $\Z^{3n}$. The projection map $\Z^3 \to \Z^2$ given by $(a,b,c) \mapsto (a-c,b-c)$ extends to a projection map
$$\pi :   \Z^{3n} \to \Z^{2n}$$
with kernel $\mathbb{T}$ and the image of $\Z E +\Z T$ is a lattice $\latt$ of rank $n-r$ in $\Z^{2n}$. 
Then $\latt$ is contained in a maximal rank $n-r$ lattice $\overline{\latt} \subset \Z^{2n}$ such that $\pi^{-1}(\overline{\latt})= \mathcal{N}(\tri;\Z)$.
We can use the Hermite decomposition to obtain a basis $\bar b_1, \ldots, \bar b_{n-r}$ for $\overline{\latt}$ and lift this to $b_1, \ldots, b_{n-r}\in \mathcal{N}(\tri;\Z)$, whose cosets $[b_i]$ form a basis for  $\mathcal{N}(\tri;\Z)/\mathbb{T}$. \\If $e_i=\chi(b_i)$ then the index formula (\ref{index0}) can be re-written:
\begin{equation}
\label{eqn:index1} 
\ind_{\tri}^{0} (q) = \sum_{c_1,\ldots,c_{n-r} \in\Z} \left(-q^{\frac12}\right)^{-\sum_i c_i e_i } \,\jtet \left(\sum_i c_i b_i\right),
\end{equation}
where the sums range over $i=1, \ldots, n-r$.~\\

In general, given $\gamma \in \Ker\left( \hh_1(\partial M;\Z) \rightarrow \hh_1(M;\Z_2)\right)$ we compute $\ind_{\tri}^{\gamma} (q)$ as follows. Start with an integer linear combination of peripheral solutions $\Gamma_2$ corresponding to $\gamma$; this gives a $Q$-normal class with $\bd \Gamma_2 =2\gamma$. By solving a system of linear equations with $\Z_2$ coefficients, we can adjust $\Gamma_2$ by adding an integer linear combination of edge and tetrahedral solutions to obtain a $Q$-normal class $\Gamma_2'$ with all quad coefficients {\em even}. Then $\Gamma = \frac{1}{2}\Gamma_2' \in \mathcal{Q}(\tri;\Z)$ is a $Q$-normal class with $\bd \Gamma = \gamma$. If $e_\Gamma=\chi(\Gamma)$ then we have
\begin{equation}
\label{index_gamma}
\ind_{\tri}^{\gamma} (q)= \sum_{c_1,\ldots,c_{n-r} \in\Z} \left(-q^{\frac12}\right)^{-e_\Gamma-\sum_i c_i e_i } \,\jtet \left(\sum_i c_i b_i+\Gamma\right),\end{equation}
using the notation from $(\ref{eqn:index1})$.\\

%&&&&&&&&&&&&&&&&&&&&&&&&&&&&&&&&&&&&&&&&&&&&&&&&&&&&&&&&&&&&&&&&&&&&&&&&&&&&&
%&&&&&&&&&&&&&&&&&&&&&&&&&&&&&&&&&&&&&&&&&&&&&&&&&&&&&&&&&&&&&&&&&&&&&&&&&&&&&

\subsection{Code and data availability}\label{sec:availability}~\\

All the data presented in this paper is completely reproducible. The code used to create the datasets is available at the GitHub repository~\cite{github}.
On the same repository we further provide the following datasets in \texttt{csv} format:
\begin{itemize}[leftmargin=70pt]
\item[\texttt{data\_1\_cusp}] A list of one-cusped census manifolds, together with the value of the index for boundary conditions $(p,q)$ for $|p|,|q| \le 10$
\item[\texttt{closed\_index}] A list of equivalent surgery descriptions for closed hyperbolic manifolds with $\le 4$ tetrahedra and surgery coefficients $(p,q)$ for $|p|,|q| \le 6$, their indices and other hyperbolic invariants
\item[\texttt{cusped\_index}] A list of single Dehn fillings on two cusped manifolds with $\le 5$ tetrahedra and surgery coefficients $(p,q)$ for $|p|,|q| \le 6$, their indices and other hyperbolic invariants
\item[\texttt{finite\_pi1}] A list of values of the index for $100$ non-hyperbolic closed manifolds with finite fundamental group
\end{itemize}

While writing this paper, we developed and tested several versions of code to assist with different aspects of index computation. Most of the code is in Mathematica~\cite{Mathematica}, but we also made large use of Sage~\cite{sage}, Normaliz~\cite{Normaliz}, Regina~\cite{regina} and SnapPy~\cite{SnapPy}. With the exception of the first, all these resources are freely available.

%&&&&&&&&&&&&&&&&&&&&&&&&&&&&&&&&&&&&&&&&&&&&&&&&&&&&&&&&&&&&&&&&&&&&&&&&&&&&&
%&&&&&&&&&&&&&&&&&&&&&&&&&&&&&&&&&&&&&&&&&&&&&&&&&&&&&&&&&&&&&&&&&&&&&&&&&&&&&

\section{Open questions}\label{sec:questions}

\begin{enumerate}
\item Prove that the Gang-Yonekura formula gives well-defined $3$D-index for closed $3$-manifolds, independent of all choices, \textit{e.g.}~surgery descriptions and triangulations used. Also give a direct definition of this index, not involving Dehn filling on a cusped manifold.

\item Extend the main results of this paper to Dehn filling on a non-standard cusp, and to Dehn filling on multiple cusps. 

\item Is there a purely topological interpretation of the coefficients of the $3$D index? 
(Work of Dunfield-Garoufalidis-Hodgson-Rubinstein~\cite{DGHR} gives an interpretation of the coefficient of $q^1$ in the $3$D index as a count of closed genus 2 surfaces in the manifold.) 

\item In~\cite{garoufalidis20151}, the index is promoted to a topological invariant of hyperbolic cusped manifolds by considering a specific class of canonical triangulations. Is there a `preferred triangulation' for small Seifert fibred spaces? 

\item Are there examples of closed $3$-manifolds whose index takes constant integer values other than $0,1,2$? 

\item Note that the ``closed index'' from Section~\ref{sec:closed} appears to vanish on lens spaces, but diverges on the connected sum of two lens spaces (\textit{cf.}~Proposition~\ref{prop:surgery_on_torus}). Is this part of a more general phenomena? That is, can the index somehow pick out reducible surgeries? 
Similarly for toroidal manifolds, as in Section~\ref{sec:toroidal}.

\item Geography and botany problems for the index: which $q$-series appear as index of some manifold, and how many manifolds share the same index? 

\item Is there a topological or physical interpretation of the relative index? 
\end{enumerate}

\appendix

%&&&&&&&&&&&&&&&&&&&&&&&&&&&&&&&&&&&&&&&&&&&&&&&&&&&&&&&&&&&&&&&&&&&&&&&&&&&&&
%&&&&&&&&&&&&&&&&&&&&&&&&&&&&&&&&&&&&&&&&&&&&&&&&&&&&&&&&&&&&&&&&&&&&&&&&&&&&&

\section{Warm-up: the (0,1,1) case}\label{sec:gang011}

After establishing the inductive step required for the proof of Theorem~\ref{thm:main}, we turned to identifying a suitable base case. The natural candidate would be the triangle in the Farey tessellation determined by $\left(\frac{0}{1},\frac{1}{0}, \frac{1}{1}\right)$However, in light of the results from Section~\ref{sec:collar_eff}, we expect this case to yield a trivial outcome.  Indeed the index we get after performing this filling, corresponding to the attachment of LST$(0,1,1)$, on a $1$-efficient triangulation is $0$. This is because, by Theorem~\ref{thm:collar_eff}, the manifold we obtain is either not $1$-efficient or a solid torus. The latter case is expected to produce an index equal to $0$ (\textit{cf.}~\cite[Sec.~11.1]{garoufalidis20163d}). 

This suggests that the triangle $\left(\frac{0}{1},\frac{1}{0}, \frac{1}{1}\right)$ is not a ``good'' starting point for an inductive argument; if the relative index for this slope is trivial, this would imply the triviality of all further ones. 
We can nonetheless prove that our local formulation (as given in equations~\eqref{eqn:GY_for_cusp} or \eqref{eqn:compact_version}) of the Gang-Yonekura formula holds for LST$(0,1,1)$ as well. In other words, equation~\eqref{eqn:GY_for_cusp} gives a relative index $GY(\mu;\underline{b})$ which is constantly $0$.
\begin{thm}\label{thm:trivial_slope}
For all $\underline{b} = (b_1,b_2,b_3) \in \Z^3$, 
\begin{equation}
GY(\mu;\underline{b}) = 0.
\end{equation}
\end{thm}

By equation~\eqref{eq:GY_b3=0}, it suffices to prove the result when $b_3=0$, and by Proposition~\ref{prop:gang_cusp} and Remark~\ref{parity_of_k} we just need to show that $GY(\mu;b_1,b_2,0) =0$
for all $b_1,b_2 \in\Z$ satisfying  $(b_1,b_2) \equiv (0,0) \mod{2}$ or $(b_1,b_2)\equiv (0,1) \mod{2}$. \\
More explicitly, using the expression for $GY(\mu;\underline{b})$ given by Proposition~\ref{prop:gang_cusp} and the formula for relative cusp index in~\eqref{eqn:cusp_index2} we will show the following: 
\begin{prop}\label{prop:gang011}
For all integers $b_1, b_2 \in \Z$ with $(b_1,b_2) \equiv (0,0) \mod{2}$ or $(b_1,b_2)\equiv (0,1) \mod{2}$
we have 
\begin{equation}\label{eqn:simple_relation_generalised}
\begin{gathered}
0 = \sum_{\substack{k \in \Z\\ k\equiv b_2 \mod{2}}}
(-1)^k\left[ \left( q^\frac{k}{2} + q^{-\frac{k}{2}}\right) \itet\left(-\frac{b_1}{2},\frac{k+b_2}{2}\right) \itet\left(-\frac{b_1}{2},\frac{-k+b_2}{2} \right) \right. -\\
 \itet\left(\frac{-2-b_1}{2},\frac{k+b_2}{2} \right) \itet\left(\frac{2-b_1}{2},\frac{-k+b_2}{2} \right) -\\\left. \itet\left(\frac{2-b_1}{2},\frac{k+b_2}{2} \right) \itet\left(\frac{-2-b_1}{2},\frac{-k+b_2}{2} \right)   \right].
\end{gathered}
\end{equation}
\end{prop}
\begin{proof}
We start by changing~\eqref{eqn:simple_relation_generalised} into its compact form. From~\eqref{eqn:compact_version} we have  
\begin{equation}\label{eq:GYmu_compact}
GY(\mu;\underline{b}) 
= \sum_{k\in\Z} (-1)^k \left(q^{\frac{k}{2}} \ind_\cusp^{rel}(0,k;\underline{b}) - \ind^{rel}_\cusp(2,k;\underline{b}) \right),
\end{equation}
and the non-zero terms only occur when $k \equiv b_2 \mod 2$.

Using equation~\eqref{eqn:cusp_index2} and the symmetry~\eqref{eqn:tetrahedral_index_symmetries} gets us to the equivalent compact form:
\begin{equation}\label{eq:compact_011_identity}
\begin{gathered}
0 = \sum_{\substack{k \in \Z\\ k\equiv b_2 \mod{2}}}
(-1)^k \left[ q^{\frac{k}{2}} \itet \left(\frac{k-b_2}{2},\frac{b_1}{2} \right)\itet \left(\frac{-k-b_2}{2}, \frac{b_1}{2}\right)   \right.\\- \left. \itet \left(\frac{k-b_2}{2}, \frac{b_1}{2}-1\right) \itet \left(\frac{-k-b_2}{2},\frac{b_1}{2} +1 \right)  \right].
\end{gathered}
\end{equation}

Using equation~\eqref{eqn:index_3term_adjacent}, with $m=\frac{k-b_2}{2}, e= \frac{b_1}{2}$, 
we can write
$$ q^{\frac{-k+b_2}{4}} \itet\left(\frac{k-b_2}{2},\frac{b_1}{2}-1\right) = \itet\left(\frac{k-b_2}{2},\frac{b_1}{2}\right) - q^{\frac{b_1}{4}}\itet\left(\frac{k-b_2}{2} -1,\frac{b_1}{2}\right),$$
and similarly
$$q^{\frac{k+b_2}{4}} \itet\left(\frac{-k-b_2}{2},\frac{b_1}{2}+1\right) = \itet\left(\frac{-k-b_2}{2},\frac{b_1}{2}\right) - q^{\frac{b_1}{4}}\itet\left(\frac{-k-b_2}{2} +1,\frac{b_1}{2}\right).$$
Multiplying these expressions together gives
\begin{equation}\label{eqn:intermediate_step_011}
\begin{gathered}
q^{\frac{b_2}{2}} \ind^{rel}_\cusp (2,k;\underline{b}) =  \ind^{rel}_\cusp (0,k;\underline{b}) + q^{\frac{b_1}{2}} \ind^{rel}_\cusp (0,k-2;\underline{b}) \\
-q^{\frac{b_1}{4}} \left[ \itet\left(\frac{k-b_2}{2}-1,\frac{b_1}{2}\right) \itet\left(\frac{-k-b_2}{2},\frac{b_1}{2}\right) + \itet\left(\frac{k-b_2}{2},\frac{b_1}{2}\right) \itet\left(\frac{-k-b_2}{2}+1,\frac{b_1}{2}\right)\right] .
\end{gathered}
\end{equation}

We can then take equation~\eqref{eqn:intermediate_step_011} and add it to a copy of itself, with the $k$ variable shifted by $-2$ to give:
\begin{equation}\label{eqn:proof011}
\begin{gathered}
q^{\frac{b_2}{2}}\left( \ind^{rel}_\cusp (2,k;\underline{b}) + \ind^{rel}_\cusp (2,k-2;\underline{b})\right) \\ 
-\ind^{rel}_\cusp (0,k;\underline{b}) - \left(1 + q^{\frac{b_1}{2}} \right) \ind^{rel}_\cusp (0,k-2;\underline{b}) - q^{\frac{b_1}{2}} \ind^{rel}_\cusp (0,k-4;\underline{b}) \\= - q^{\frac{b_1}{4}} \left[ \itet\left(\frac{k-b_2}{2}-1,\frac{b_1}{2}\right) \left( \itet\left(\frac{-k-b_2}{2},\frac{b_1}{2}\right) + \itet\left(\frac{-k-b_2}{2}+2,\frac{b_1}{2}\right)\right) \right.\\ 
+\left.
\itet\left(\frac{-k-b_2}{2}+1,\frac{b_1}{2}\right) \left( \itet\left(\frac{k-b_2}{2},\frac{b_1}{2}\right) + \itet\left(\frac{k-b_2}{2}-2,\frac{b_1}{2}\right)\right)\right] . 
\end{gathered}
\end{equation}

We then apply equation~\eqref{eqn:index_3term_consecutive} to the two sums within the round parentheses on the right-hand side of equation~\eqref{eqn:proof011}, to obtain 
\begin{equation*}
\begin{gathered}
\text{RHS~of~\eqref{eqn:proof011}}\\ = 
- q^{\frac{b_1}{4}}  \left[ \left(q^{\frac{b_1}{4}} + q^{-\frac{b_1}{4}} -q^{\frac{k+b_2}{2} - \frac{b_1}{4} -1}\right)\itet\left(\frac{k-b_2}{2}-1,\frac{b_1}{2}\right)\itet\left(\frac{-k-b_2}{2}+1,\frac{b_1}{2}\right)\right.\\
+\left. \left(q^{\frac{b_1}{4}} + q^{-\frac{b_1}{4}} -q^{\frac{-k+b_2}{2} - \frac{b_1}{4} +1}\right)\itet\left(\frac{-k-b_2}{2}+1,\frac{b_1}{2}\right)\itet\left(\frac{k-b_2}{2}-1,\frac{b_1}{2}\right) \right]  \\=
 \ind_\cusp^{rel} (0,k-2;\underline{b})\left[q^{\frac{b_2}{2}}\left( q^{\frac{k-2}{2}} + q^{\frac{-k+2}{2}}\right) - 2q^{\frac{b_1}{2}}-2\right] .    
\end{gathered}
\end{equation*}

Putting everything together, this implies 

\begin{gather*}
\left( q^{\frac{k-2}{2}} + q^{\frac{-k+2}{2}} \right)\ind^{rel}_\cusp (0,k-2;\underline{b}) - \ind^{rel}_\cusp (2,k;\underline{b}) - \ind^{rel}_\cusp (2,k-2;\underline{b}) \\=
-q^{-\frac{b_2}{2}} \left[ \ind^{rel}_\cusp (0,k;\underline{b}) + q^{\frac{b_1}{2}} \ind^{rel}_\cusp (0,k-4;\underline{b}) - \left( 1 + q^{\frac{b_1}{2}}\right)\ind^{rel}_\cusp (0,k-2;\underline{b})\right].
\end{gather*}
The sum over all $k \in \Z$ with $k \equiv b_2 \mod 2$ of the right-hand side in this last equation is $0$, by an elementary telescoping argument. 
The corresponding sum of the left-hand side gives, after re-indexing and using the symmetry in~\eqref{eq:GYsymmetry1}, $\pm 2$ times the compact form of $GY(\mu;\underline{b})$ in~\eqref{eq:GYmu_compact}. Therefore, we have proved equation~\eqref{eq:compact_011_identity} and the proposition. 
\end{proof}

%&&&&&&&&&&&&&&&&&&&&&&&&&&&&&&&&&&&&&&&&&&&&&&&&&&&&&&&&&&&&&&&&&&&&&&&&&&&&&
%&&&&&&&&&&&&&&&&&&&&&&&&&&&&&&&&&&&&&&&&&&&&&&&&&&&&&&&&&&&&&&&&&&&&&&&&&&&&&

\section{Sanity check: the relative index of degenerate LSTs}\label{sec:relative_degenerate}

As noted in Sections~\ref{sec:$3$-man} and~\ref{sec:attaching_solid_tori}, the minimal layered solid tori corresponding to the triples $(0,1,1)$ and $(1,1,2)$ are particularly degenerate. We thus compute the relative index of these LSTs by choosing two non-minimal representatives, as described in~\cite{jaco2006layered}. We note that these triangulations are the ones provided by Regina~\cite{regina}.

We start from LST$(1,1,2)$, which is presented by Regina with a triangulation $\tri(1,1,2)$ whose gluing equations (after orienting the triangulation and renumbering the edges) are:
$$ 
\begin{matrix}
\frac12 b_1 & E_1\\ \frac12 b_2 & E_2\\ \frac12 b_3 & E_3 \\\frac12 k & E_4
\end{matrix}
\left( \begin{array}{ccc|ccc}
0 &0 &2 &1 &2 &0 \\
1 &0 &0 &0 &0 &0 \\
0 &2 &0 &0 &0 &2 \\
1 &0 &0 &1 &0 &0 
\end{array}
\right)$$

The relative index can be written
\begin{equation*}
\ind^{rel}_{\tri(1,1,2)}(b_1,b_2,b_3) = \sum_{\substack{k \in \Z\\k \equiv b_1 \equiv b_2}} \left( -q^\frac12\right)^k \jtet \left( \frac{k + b_2}{2}, b_3,b_1\right) \jtet \left( \frac{k + b_1}{2}, b_1,b_3\right).
\end{equation*}
A simple application of the quadratic identity~\eqref{eqn:quadratic_identity} yields
$$\ind^{rel}_{\tri(1,1,2)} (b_1,b_2,b_3) = \left( -q^\frac12\right)^{-b_1} \delta_{b_1,b_2} .$$
And of course this is consistent with Theorem~\ref{thm:new_thm_112}.\\
~\\

For LST$(0,1,1)$, Regina gives a non-degenerate triangulation $\tri(0,1,1)$ with $3$ tetrahedra and two internal edges. Its gluing equations (after orienting the triangulation) are: 
$$ 
\begin{matrix}
\frac12 b_1 & E_1\\ \frac12 b_2 & E_2\\ \frac12 b_3 & E_3 \\\frac12 k_1 & E_4\\\frac12 k_2 & E_5
\end{matrix}
\left( \begin{array}{ccc|ccc|ccc}
1 &0 &0 &0 &0 &0 & 0 & 0 & 0 \\
0 &2 &0 &1 &0 &0 & 0 & 0 & 0 \\
0 &0 &2 &0 &0 &2 & 1 & 2 & 0 \\
1 &0 &0 &0 &2 &0 & 0 & 0 & 2 \\
0 &0 &0 &1 &0 &0 & 1 & 0 & 0 
\end{array}
\right)$$
In order to get integer entries in the $\jtet$ functions for this triangulation, we need the following conditions: $k_1 \equiv b_1 \mod{2}$, and $k_2 \equiv b_2 \equiv b_3 \mod{2}$.
We can then write $\ind^{rel}_{\tri(0,1,1)}(b_1,b_2,b_3)$ as
\begin{gather*}
\sum_{\substack{k_1,k_2 \in \Z\\k_1 \equiv b_1 \mod{2}\\k_2 \equiv b_2 \equiv b_3 \mod{2}}} \left( -q^\frac12\right)^{k_1 + k_2} \jtet \left( \frac{k_1 + b_2}{2}, b_2,b_3\right) \jtet \left( \frac{k_2 + b_2}{2}, k_1,b_3\right) \jtet \left( \frac{k_2 + b_3}{2},b_3,k_1\right) \\=
\sum_{\substack{k_1 \in \Z\\k_1 \equiv b_1 \mod{2}}} \left( -q^\frac12\right)^{k_1 } \jtet \left( \frac{k_1 + b_2}{2}, b_2,b_3\right) \\\sum_{\substack{k_2 \in \Z\\k_2 \equiv b_2 \equiv b_3 \mod{2}}} \left( -q^\frac12\right)^{k_2}
\jtet \left( \frac{k_2 + b_2}{2}, k_1,b_3\right) \jtet \left( \frac{k_2 + b_3}{2},b_3,k_1\right) \\ 
= \sum_{\substack{k_1 \in \Z\\k_1 \equiv b_1 \mod{2}}} \left( -q^\frac12\right)^{k_1 }  \left( -q^\frac12\right)^{-b_2 } \jtet \left( \frac{k_1 + b_2}{2}, b_2,b_3\right) \delta_{b_2,b_3},
\end{gather*}
where the last equality follows from an application of the quadratic identity~\eqref{eqn:quadratic_identity}. This later quantity is manifestly $0$ if $b_2 \neq b_3$, and equal to $$\sum_{\substack{k_1 \in \Z\\k_1 \equiv b_1 \mod{2}}}  \left( -q^\frac12\right)^{k_1 -b_2} \jtet \left( \frac{k_1 + b_2}{2}, b_2,b_3\right)$$
otherwise.
Note that we can conclude that the total sum is $0$ in this case as well, using \textit{e.g.}~\cite[Ex.~11.1]{garoufalidis20163d}.

This gives
$$\ind^{rel}_{\tri(0,1,1)}(b_1,b_2,b_3)=0, \text{ for all } (b_1,b_2,b_3) \in \Z^3.$$

Combining this with result from Appendix~\ref{sec:gang011} shows that the relative version of the Gang-Yonekura formula for filling the standard cusp~\ref{thm:main_relative} does indeed hold for these degenerate cases as well.

%&&&&&&&&&&&&&&&&&&&&&&&&&&&&&&&&&&&&&&&&&&&&&&&&&&&&&&&&&&&&&&&&&&&&&&&&&&&&&
%&&&&&&&&&&&&&&&&&&&&&&&&&&&&&&&&&&&&&&&&&&&&&&&&&&&&&&&&&&&&&&&&&&&&&&&&&&&&&

\section{Computational example: figure eight knot complement}\label{computation:fig8}

We illustrate the rigorous computation methods for the figure eight knot complement using its canonical triangulation $\tri$ with two tetrahedra, with edges and peripheral curves labelled as in ~\cite[Ex.~4.1]{garoufalidis20163d}. 

Here the gluing matrix is given by
\begin{equation}
G_\tri = \begin{pmatrix}
2& 1 & 0 & 2 & 1 & 0 \\
0 & 1 & 2 & 0 & 1 & 2 \\
0 & 0 & 1 & -1 & 0 & 0 \\
0 & 0 & 0 & 2 & 0 & -2
\end{pmatrix} 
\end{equation}
where the rows correspond to edge solutions $E_1,E_2$ and peripheral solutions $M,L$.

Recalling that the homological boundaries of the $Q$-normal classes $M$ and $L$ are given by $2\mu$ and $2\lambda$, 
it follows that the  $3$D-index in our notation is given by
$$\ind_{\tri}^{2x \mu +y \lambda }(q) = \sum_{k\in\Z} q^k \jtet(2k-x+y,k,-y)
= \sum_{k \in \Z} \itet (k-x,k) \itet (k+y,k-x+y).$$

To compute the index up to $q$-degree $D$ using Normaliz~\cite{Normaliz}, we proceed as explained in Section~\ref{ssec:computation}: right-multiplication of $G_\tri$ by the skew-symmetric,  block diagonal matrix 
$$C =  \begin{pmatrix}
0 & 1 & -1 & 0 & 0 &0\\
-1 & 0 & 1& 0 & 0 &0\\
1 & -1 & 0& 0 & 0 &0 \\
0 & 0 &0 & 0 & 1 & -1\\
0 & 0 &0 &-1 & 0 & 1\\
0 & 0 &0 &1 & -1 & 0
\end{pmatrix} 
$$
gives the matrix of \emph{leading-trailing deformations}
\begin{equation}
L_\tri = \begin{pmatrix}
-1& 2 & -1 & -1 & 2 & -1 \\
1 & -2 & 1 & 1 & -2 & 1 \\
1 & -1 & 0 & 0 & -1 & 1 \\
0 & 0 & 0 & -2 & 4 & -2
\end{pmatrix} .
\end{equation}
Here the first two rows give the $Q$-matching equations $\L(E_1), \L(E_2)$ for the two edges, while the last two rows give the leading trailing deformations $\L(M), \L(L)$ for the meridian $\mu$ and longitude $\lambda$. 
 
\begin{example} To compute the index $\ind_{\tri}^{(0,0)}(q)$, we need to find the integer $Q$-normal classes $S$ with boundary $\bd S = 0$. Using the Neumann-Zagier symplectic relations, the condition $\bd S = 0$ gives two additional linear equations
$$  \L(M) \cdot S = 0 \text{ and }  \L(L) \cdot S = 0,$$
where $\cdot$ is the dot product.

We first use Normaliz to decompose the polyhedral cone
$$\{ x \in \Z^6 \mid L_\tri  x=0, x\ge 0 \},$$
into a union of simplicial cones, using the input file:

\begin{verbatim}
amb_space	6
equations	4
-1	2	-1	-1	2	-1
1	-2	1	1	-2	1
1	-1	0	0	-1	1
0	0	0	-2	4	-2
UnimodularTriangulation
\end{verbatim}

This gives output consisting of $4$ vertices
\begin{verbatim}
{{0,0,0,1,1,1},{0,1,2,0,1,2},{1,1,1,0,0,0},{2,1,0,2,1,0}}
\end{verbatim}
and cones over the two $2$-simplices with vertex sets
\begin{verbatim}
{{1,2,3},{1,3,4}}
\end{verbatim}

Discarding the tetrahedral solutions corresponding to vertices $1$ and $3$, gives us $2$ rays from the origin spanned by vertices $2$ and $4$.

Then to get all terms in the index sum up to $q^\frac12$-degree $D$ we look at multiples of vertices $2$ and $4$ with non-negative integer coefficients, estimate degrees, eliminate any tetrahedral solutions that arise,  discard duplicates, then add up the contributions to the index. For example, taking $D=20$ gives
$$\ind_{\tri}^{0}(q) =1-2 q-3 q^2+2 q^3+8 q^4+18 q^5+18 q^6+14 q^7-12 q^8-52 q^9-106
   q^{10}+O\left(q^{\frac{21}{2}}\right)$$
      
\end{example}

\begin{example} 
To compute the index $\ind_{\tri}^{(x,y)}(q)$, we need to find the integer $Q$-normal classes $S$ with boundary 
$\alpha = x \mu + y \lambda$. Using the Neumann-Zagier symplectic relations gives additional linear equations
$$  \L(M) \cdot S = - y \text{ and }  \L(L) \cdot S = + x.$$

For example, for $(x,y)=(4,1)$  
we first compute the corresponding homogeneous cone as in the previous example,
and then compute additional translates (module generators) giving the inhomogeneous cone using the Normaliz input file:
\begin{verbatim}
amb_space	6
inhom_equations	4
-1	2	-1	-1	2	-1	0
1	-2	1	1	-2	1	0
1	-1	0	0	-1	1	1
0	0	0	-2	4	-2	-4
nonnegative
\end{verbatim}

The output gives a Hilbert basis for the homogeneous cone (recession cone) as  in the previous example above, and additional translations (module generators):
\begin{verbatim}
{{0,0,2,0,1,0},{1,0,1,2,2,0},{2,0,0,4,3,0}}
\end{verbatim}

For example, taking $D=20$ gives
$$
\ind_{\tri}^{(4,1)}(q) =-q^{\frac{1}2}+q^{\frac{5}2}+4 q^{\frac{7}2}+7 q^{\frac{9}2}+7 q^{\frac{11}2}+3 q^{\frac{13}2}-12 q^{\frac{15}2}-31 q^{\frac{17}2}-62
   q^{\frac{19}2}+O\left(q^{\frac{21}2}\right)
   $$
\end{example}

\begin{example} 
To compute the Gang-Yonekura formula for Dehn filling along $\alpha = x \mu + y \lambda$
we need to find integer $Q$-normal classes $S$ of two types:
\begin{enumerate}
\item  
the intersection number $\bd S \cdot \alpha =0$, and 
\item  
the intersection number $\bd S \cdot \alpha =\pm 2$.
\end{enumerate}
From the Neumann-Zagier symplectic relations, these give additional linear equations: 
\begin{enumerate}
\item $\L(A) \cdot S =0$, and
\item $\L(A) \cdot S = \pm 2$,
\end{enumerate}
where $A = x M + y L$ is the holonomy vector for $\alpha$.\\
~\\
For example, when $\alpha=0\mu+1\lambda$:
case $1$ can be studied using the Normaliz input file

\begin{verbatim}
amb_space	6
equations	3
-1	2	-1	-1	2	-1
1	-2	1	1	-2	1
0	0	0	-2	4	-2
UnimodularTriangulation
\end{verbatim}
and case $2$ can be studied using the two input files\\ 

\begin{minipage}{0.45\textwidth}
\begin{verbatim}
amb_space	6
inhom_equations	3
-1	2	-1	-1	2	-1	0
1	-2	1	1	-2	1	0
0	0	0	-2	4	-2	2
nonnegative
\end{verbatim}
\end{minipage}
\hfill
\begin{minipage}{0.45\textwidth}
\begin{verbatim}
amb_space	6
inhom_equations	3
-1	2	-1	-1	2	-1	0
1	-2	1	1	-2	1	0
0	0	0	-2	4	-2	-2
nonnegative
\end{verbatim}
\end{minipage}

\medskip\noindent
First we obtain a simplicial cone decomposition for the homogeneous cone given by $(1)$:
\begin{verbatim}
vertices
{{0,0,0,0,1,2},{0,0,0,2,1,0},{0,1,2,0,0,0},{2,1,0,0,0,0},{0,0,0,1,1,1},
{1,1,1,0,0,0}}
simplices
{{2,4,5,6},{2,3,5,6},{1,4,5,6},{1,3,5,6}}
\end{verbatim}
After removing tetrahedral solutions this gives $4$ cones over $2$-dimensional simplices.

\noindent
We then find translates of this cone giving the inhomogeneous cones arising in case $(2)$:
\begin{verbatim}
translations
{{0,1,1,0,0,1},{0,1,1,1,0,0},{1,1,0,0,0,1},{1,1,0,1,0,0}}
{{0,0,1,0,1,1},{0,0,1,1,1,0},{1,0,0,0,1,1},{1,0,0,1,1,0}}
\end{verbatim}

Then to get all terms in the index sum up to $q^\frac12$-degree $D$, we look at linear combinations of these vertices with non-negative integer coefficients, estimate degrees, eliminate any tetrahedral solutions that arise, discard duplicates, then add up the contributions to the index.

For example, taking $D=20$ gives
$$\ind_{\tri(0,1)}(q) =1+O\left(q^{\frac{11}{2}}\right).$$
(In fact, we proved in Theorem~\ref{thm:4_1_exact} that the result is exactly $1$.)

\end{example}


\newpage


\begin{thebibliography}{GHHR16}

\bibitem[Ata96]{Atakishiyev}
NM~Atakishiyev.
\newblock On a one-parameter family of $q$-exponential functions.
\newblock {\em Journal of Physics A: Mathematical and General}, 29(10):L223,
  1996.

\bibitem[BBP23]{regina}
Benjamin~A. Burton, Ryan Budney, William Pettersson, et al.
\newblock Regina: Software for low-dimensional topology.
\newblock \url{ http://regina-normal.github.io/}, 1999--2023.

\bibitem[BISvdO]{Normaliz}
W.~Bruns, B.~Ichim, C.~S\"oger, and U.~von~der Ohe.
\newblock Normaliz algorithms for rational cones and affine monoids.
\newblock \url{https://www.normaliz.uni-osnabrueck.de}.

\bibitem[BJR21]{BJR}
Birch Bryant, William Jaco, and J.~Hyam Rubinstein.
\newblock Efficient triangulations and boundary slopes.
\newblock {\em Topology Appl.}, 297:Paper No. 107689, 18, 2021.

\bibitem[Bur25]{benprivate}
Ben Burton.
\newblock Census of $1$-efficient triangulations of cusped $3$-manifolds.
\newblock private communication, 2025.

\bibitem[CDGW]{SnapPy}
Marc Culler, Nathan~M. Dunfield, Matthias Goerner, and Jeffrey~R. Weeks.
\newblock Snap{P}y, a computer program for studying the geometry and topology
  of $3$-manifolds.
\newblock Available at \url{http://snappy.computop.org}.

\bibitem[CHR25]{github}
Daniele Celoria, Craig~D. Hodgson, and J.~Hyam Rubinstein.
\newblock 3{D}index.
\newblock \url{https://github.com/dceloriamaths/3DIndex}, 2025.

\bibitem[DGG13]{dimofte20133}
Tudor Dimofte, Davide Gaiotto, and Sergei Gukov.
\newblock 3-manifolds and 3d indices.
\newblock {\em Advances in theoretical and mathematical physics},
  17(5):975--1076, 2013.

\bibitem[DGG14]{dimofte2014gauge}
Tudor Dimofte, Davide Gaiotto, and Sergei Gukov.
\newblock Gauge theories labelled by three-manifolds.
\newblock {\em Communications in Mathematical Physics}, 325(2):367--419, 2014.

\bibitem[DGHR]{DGHR}
Nathan Dunfield, Stavros Garoufalidis, Craig~D Hodgson, and J~Hyam Rubinstein.
\newblock Counting genus two surfaces in 3-manifolds.
\newblock in preparation.

\bibitem[EP88]{epstein1988euclidean}
David~BA Epstein and Robert~C Penner.
\newblock Euclidean decompositions of noncompact hyperbolic manifolds.
\newblock {\em Journal of Differential Geometry}, 27(1):67--80, 1988.

\bibitem[FG11]{FuterGueritaud}
David Futer and Fran\c{c}ois Gu\'eritaud.
\newblock From angled triangulations to hyperbolic structures.
\newblock In {\em Interactions between hyperbolic geometry, quantum topology
  and number theory}, volume 541 of {\em Contemp. Math.}, pages 159--182. Amer.
  Math. Soc., Providence, RI, 2011.

\bibitem[FKB08]{frohman2008quantum}
Charles Frohman and Joanna Kania-Bartoszynska.
\newblock The quantum content of the normal surfaces in a three-manifold.
\newblock {\em Journal of Knot Theory and Its Ramifications},
  17(08):1005--1033, 2008.

\bibitem[Gan18]{gang2018quantum}
Dongmin Gang.
\newblock Quantum approach to {D}ehn surgery problem.
\newblock {\em arXiv preprint arXiv:1803.11143}, 2018.

\bibitem[Gar16]{garoufalidis20163d_angle}
Stavros Garoufalidis.
\newblock The 3{D} index of an ideal triangulation and angle structures.
\newblock {\em The Ramanujan Journal}, 40:573--604, 2016.

\bibitem[GHH08]{GHH08}
Oliver Goodman, Damian Heard, and Craig Hodgson.
\newblock Commensurators of cusped hyperbolic manifolds.
\newblock {\em Experiment. Math.}, 17(3):283--306, 2008.

\bibitem[GHHR16]{garoufalidis20163d}
Stavros Garoufalidis, Craig~D Hodgson, Neil~R Hoffman, and J~Hyam Rubinstein.
\newblock The 3d-index and normal surfaces.
\newblock {\em Illinois Journal of Mathematics}, 60(1):289--352, 2016.

\bibitem[GHRS15]{garoufalidis20151}
Stavros Garoufalidis, Craig~D Hodgson, J~Hyam Rubinstein, and Henry Segerman.
\newblock 1--efficient triangulations and the index of a cusped hyperbolic
  3--manifold.
\newblock {\em Geometry \& Topology}, 19(5):2619--2689, 2015.

\bibitem[GK19]{garoufalidis2019meromorphic}
Stavros Garoufalidis and Rinat Kashaev.
\newblock A meromorphic extension of the 3{D} index.
\newblock {\em Research in the Mathematical Sciences}, 6(1):8, 2019.

\bibitem[GR04]{GasperRahman}
George Gasper and Mizan Rahman.
\newblock {\em Basic hypergeometric series}, volume~96.
\newblock Cambridge University Press, 2004.

\bibitem[GS10]{gueritaud2010canonical}
Fran{\c{c}}ois Gu{\'e}ritaud and Saul Schleimer.
\newblock Canonical triangulations of {D}ehn fillings.
\newblock {\em Geometry \& Topology}, 14(1):193--242, 2010.

\bibitem[GvdV23]{garoufalidis2023fkb}
Stavros Garoufalidis and Roland van~der Veen.
\newblock The {FKB} invariant is the 3d index.
\newblock {\em Quantum Topol.}, 13:525--538, 2023.

\bibitem[GW22]{garoufalidis2022periods}
Stavros Garoufalidis and Campbell Wheeler.
\newblock Periods, the meromorphic 3d-index and the {T}uraev--{V}iro invariant.
\newblock {\em arXiv preprint arXiv:2209.02843}, 2022.

\bibitem[GY18]{gang2018symmetry}
Dongmin Gang and Kazuya Yonekura.
\newblock Symmetry enhancement and closing of knots in 3d/3d correspondence.
\newblock {\em Journal of High Energy Physics}, 2018(7):1--58, 2018.

\bibitem[GY24]{garoufalidis20243d}
Stavros Garoufalidis and Tao Yu.
\newblock The 3d-index of the 3d-skein module via the quantum trace map.
\newblock {\em arXiv preprint arXiv:2406.04918}, 2024.

\bibitem[HKS21]{hodgson2021asymptotics}
Craig~D Hodgson, Andrew~J Kricker, and Rafa{\l}~M Siejakowski.
\newblock On the asymptotics of the meromorphic 3{D}-index.
\newblock {\em arXiv preprint arXiv:2109.05355}, 2021.

\bibitem[HMP20]{howie2020polynomials}
Joshua~A Howie, Daniel~V Mathews, and Jessica~S Purcell.
\newblock A-polynomials, {P}tolemy equations and {D}ehn filling.
\newblock {\em arXiv preprint arXiv:2002.10356}, 2020.

\bibitem[HMW92]{hodgson1992surgeries}
Craig~D Hodgson, G~Robert Meyerhoff, and Jeffrey~R Weeks.
\newblock Surgeries on the {W}hitehead link yield geometrically similar
  manifolds.
\newblock {\em Topology}, 90(195-206):15, 1992.

\bibitem[Hof]{neil_list}
Neil Hoffman.
\newblock Data for manifolds in the closed census.
\newblock
  \url{https://math.okstate.edu/people/nhoffman/smalltriangulations.html}.
\newblock Accessed: 05-2025.

\bibitem[Jol16]{jolley}
Max Jolley.
\newblock The 3d index for torus knots.
\newblock Honours thesis, Monash University, October 2016.

\bibitem[JR03]{jaco20030}
William Jaco and J~Hyam Rubinstein.
\newblock 0-efficient triangulations of 3-manifolds.
\newblock {\em Journal of Differential Geometry}, 65(1):61--168, 2003.

\bibitem[JR06]{jaco2006layered}
William Jaco and J~Hyam Rubinstein.
\newblock Layered-triangulations of 3-manifolds.
\newblock {\em arXiv preprint math/0603601}, 2006.

\bibitem[JS03]{jaco2003decision}
William Jaco and Eric Sedgwick.
\newblock Decision problems in the space of {D}ehn fillings.
\newblock {\em Topology}, 42(4):845--906, 2003.

\bibitem[KR04]{kangrubinstein}
Ensil Kang and J.~Hyam Rubinstein.
\newblock Ideal triangulations of 3-manifolds. {I}. {S}pun normal surface
  theory.
\newblock In {\em Proceedings of the {C}asson {F}est}, volume~7 of {\em Geom.
  Topol. Monogr.}, pages 235--265. Geom. Topol. Publ., Coventry, 2004.

\bibitem[Mos71]{moser1971elementary}
Louise Moser.
\newblock Elementary surgery along a torus knot.
\newblock {\em Pacific Journal of Mathematics}, 38(3):737--745, 1971.

\bibitem[MP06]{martelli2006dehn}
Bruno Martelli and Carlo Petronio.
\newblock Dehn filling of the ``magic'' 3-manifold.
\newblock {\em Communications in analysis and geometry}, 14(5):969--1026, 2006.

\bibitem[Neu92]{neumann92}
Walter~D. Neumann.
\newblock Combinatorics of triangulations and the {C}hern-{S}imons invariant
  for hyperbolic {$3$}-manifolds.
\newblock In {\em Topology '90 ({C}olumbus, {OH}, 1990)}, volume~1 of {\em Ohio
  State Univ. Math. Res. Inst. Publ.}, pages 243--271. de Gruyter, Berlin,
  1992.

\bibitem[NZ85]{NZ85}
Walter~D. Neumann and Don Zagier.
\newblock Volumes of hyperbolic three-manifolds.
\newblock {\em Topology}, 24(3):307--332, 1985.

\bibitem[SAG24]{sage}
W.\thinspace{}A. Stein et~al.
\newblock {\em {S}age {M}athematics {S}oftware ({V}ersion 10.5)}.
\newblock The Sage Development Team, 2024.
\newblock \url{ http://www.sagemath.org}.

\bibitem[Sta99]{Stanley}
Richard~P Stanley.
\newblock {\em Enumerative {C}ombinatorics, {V}olume 2}.
\newblock Cambridge University Press, 1999.

\bibitem[{\v{S}}ta16]{Nevanlinna}
Franti{\v{s}}ek {\v{S}}tampach.
\newblock Nevanlinna extremal measures for polynomials related to
  $q$-$1$-{F}ibonacci polynomials.
\newblock {\em Advances in Applied Mathematics}, 78:56--75, 2016.

\bibitem[Tho25]{thompson2025triangulations}
Em~K Thompson.
\newblock Triangulations of the magic manifold and families of census knots.
\newblock {\em arXiv preprint arXiv:2503.06198}, 2025.

\bibitem[Thu78]{Thu78notes}
William~P Thurston.
\newblock \emph{The Geometry and Topology of Three-Manifolds}.
\newblock Princeton University lecture notes, available at
  \url{http://library.msri.org/books/gt3m/}, 1978.

\bibitem[Thu82]{thurston1982three}
William~P Thurston.
\newblock Three-dimensional manifolds, {K}leinian groups and hyperbolic
  geometry.
\newblock {\em Bulletin of the American Mathematical Society}, 6(3):357--381,
  1982.

\bibitem[Til08]{tillmann}
Stephan Tillmann.
\newblock Normal surfaces in topologically finite 3-manifolds.
\newblock {\em Enseign. Math. (2)}, 54(3-4):329--380, 2008.

\bibitem[Tol98]{tollefson1998normal}
Jeffrey~L Tollefson.
\newblock Normal surface {Q}-theory.
\newblock {\em Pacific Journal of Mathematics}, 183(2):359--374, 1998.

\bibitem[{Wol}]{Mathematica}
{Wolfram Research{,} Inc.}
\newblock Mathematica, {V}ersion 14.2.
\newblock Champaign, IL, 2024.

\end{thebibliography}
\end{document}